\documentclass[a4paper,11pt,reqno,noindent]{amsart}
\usepackage[centertags]{amsmath}
\usepackage{amsfonts,amssymb,amsthm,dsfont,cases,amscd,esint,enumerate}
\usepackage[T1]{fontenc}
\usepackage[english]{babel}
\usepackage[applemac]{inputenc}
\usepackage{newlfont}
\usepackage{color}
\usepackage[body={15cm,21.5cm},centering]{geometry} 
\usepackage{fancyhdr}
\usepackage{enumerate}

\theoremstyle{plain}
\newtheorem{theor10}{Theorem}

\newtheorem{prop10}[theor10]{Proposition}

\newtheorem{theor0}{Theorem}[section]
\newenvironment{theor}
  {\pushQED{\qed}\begin{theor0}}
  {\popQED\end{theor0}}
\newtheorem{lem0}[theor0]{Lemma}
\newenvironment{lem}
  {\pushQED{\qed}\begin{lem0}}
  {\popQED\end{lem0}}
\newtheorem{prop0}[theor0]{Proposition}
\newenvironment{prop}
  {\pushQED{\qed}\begin{prop0}}
  {\popQED\end{prop0}}
\newtheorem{cor0}[theor0]{Corollary}

\theoremstyle{definition}
\newtheorem{rem0}[theor0]{Remark}
\newenvironment{rem}
  {\pushQED{\qed}\begin{rem0}}
  {\popQED\end{rem0}}

\numberwithin{equation}{section}

\newcommand{\Sp}{\mathbb S}
\newcommand{\e}{\varepsilon}

\newcommand{\dd}{\operatorname{d}}
\newcommand{\Pc}{\mathcal{P}}

\newcommand{\Gc}{\mathcal G}
\newcommand{\dist}{\operatorname{dist}}
\newcommand{\R}{\mathbb R}
\newcommand{\Ic}{\mathcal I}
\newcommand{\loc}{{\operatorname{loc}}}
\newcommand{\Id}{\operatorname{Id}}
\newcommand{\D}{\operatorname{D}}
\newcommand{\ee}{e}
\newcommand{\supess}{\operatorname{sup\,ess}}
\newcommand{\supessd}[1]{\mathrel{\mathop{\supess}\limits_{#1}}}
\newcommand{\Bb}{\bar{\boldsymbol B}}
\newcommand{\Ld}{\operatorname{L}}
\newcommand{\supp}{\operatorname{supp}}
\newcommand{\Div}{{\operatorname{div}}}
\newcommand{\Sym}{{\operatorname{sym}}}
\newcommand{\Skew}{{\operatorname{skew}}}

\newcommand{\step}[1]{\noindent \textit{Step} #1.}
\newcommand{\substep}[1]{\noindent \textit{Substep} #1.}

\usepackage[colorlinks,citecolor=black,urlcolor=black]{hyperref}

\title[Semi-dilute rheology of particle suspensions]{Semi-dilute rheology of particle suspensions: derivation of Doi-type models}

\author[M. Duerinckx]{Mitia Duerinckx}
\address[Mitia Duerinckx]{Universit\'e Libre de Bruxelles, D\'epartement de Math\'ematique, 1050~Brussels, Belgium}
\email{mitia.duerinckx@ulb.be}

\begin{document}
\selectlanguage{english}

\begin{abstract}
This work is devoted to the large-scale rheology of suspensions of non-Brownian inertialess rigid particles, possibly self-propelling, suspended in Stokes flow. Starting from a hydrodynamic model, we derive a semi-dilute mean-field description in form of a Doi-type model, which is given by a `macroscopic' effective Stokes equation coupled with a `microscopic' Vlasov equation for the statistical distribution of particle positions and orientations. This accounts for some non-Newtonian effects since the viscosity in the effective Stokes equation depends on the local distribution of particle orientations via Einstein's formula. The main difficulty is the detailed analysis of multibody hydrodynamic interactions between the particles, which we perform by means of a cluster expansion combined with a multipole expansion in a suitable dilute regime.
\end{abstract}

\maketitle
\setcounter{tocdepth}{1}
\tableofcontents

\section{Introduction and main results}

\subsection{General overview}\label{sec:overview}
Suspensions of inertialess rigid particles in a Stokes flow are omnipresent both in natural phenomena and in practical applications, and they display reputedly complex rheological behaviors on large scales, including non-Newtonian effects, e.g.~\cite{Doi-Edwards-88,Saintillan-18}. Such behaviors are easily understandable heuristically. Indeed, on the one hand, averaging is expected to occur on large scales, leading to a notion of effective viscosity for the suspension. On the other hand, this effective viscosity depends on the local spatial arrangement of the particles at the microscopic scale, which itself evolves with the fluid flow and can thus adapt to external forces over time. This flow-induced microstructure can thus result in a nonlinear response to external forces, and hence in non-Newtonian effects on large scales. The complete understanding of such behaviors from a micro-macro perspective would require coupling homogenization with microstructure dynamics, which remains a completely open problem.
A natural simplification amounts to focussing on the dilute regime: in that case, particles are sparse and interact little, so only reduced information on the microstructure matters, for which a dynamical description could be simpler.
In this line of research, previous work has mainly been devoted to the following two preliminary questions:
\begin{enumerate}[---]
\item {\it Averaging {of the fluid flow} for a `given' microstructure:}\\
Given instantaneous particle positions, the Stokes problem defining the fluid velocity can be approximated on large scales by an effective Stokes equation with some effective viscosity, say $\Bb$. This is now well-understood in the framework of homogenization theory~\cite{DG-21a,DG-21c} and we refer to~\cite{DG-22-review} for a review.
\smallskip\item {\it Semi-dilute expansion of the effective viscosity:}\\
In the dilute regime, the effective viscosity $\Bb$ can be expanded with respect to the particle volume fraction $\lambda\ll1$.
To first order, this expansion takes form of the celebrated Einstein formula
\begin{equation}\label{eq:Einstein}
\Bb=\Id+\lambda\Bb_{1}+\ldots,
\end{equation}
where $\Bb_{1}$ only depends on the single-particle distribution of orientations.
The next-order correction further involves the statistical distribution of pairs of particles on the microscale. This correction can actually be as large as $O(\lambda)$ {if an $O(1)$ fraction of the particles form clusters}, while it is $O(\lambda^2)$ in case of well-separated particles.
The expansion can be pursued to higher orders in form of a cluster expansion and is now well understood. For details, see~\cite{DG-21b,DG-22-review} and references therein.
\end{enumerate}
With these results at hand, it remains to couple homogenization of the fluid flow with microstructure dynamics in the semi-dilute regime. As Einstein's first-order approximation~\eqref{eq:Einstein} only involves the single-particle distribution of orientations, and as corrections to the effective viscosity are~$O(\lambda^2)$ in case of well-separated particles, we can expect to achieve a mean-field description of the dynamics with accuracy $O(\lambda^2)$ in that case:
it would take form of an effective Stokes equation coupled to a Vlasov equation for the distribution of orientations; see~\eqref{eq:Doi0-1}--\eqref{eq:Doi0-2} below. This corresponds to the so-called Doi model first derived formally in~\cite{Jeffery-22,Hinch-Leal-72,Brenner-74}, which accounts for non-Newtonian effects due to the collective orientation of the particles.

The first rigorous results on the dilute dynamics~\cite{Jabin-Otto-04,Hofer-18,Mecherbet-19} focused on the leading-order dilute mean-field description, neglecting Einstein's $O(\lambda)$ correction to the effective viscosity~\eqref{eq:Einstein}: this leads to a simpler transport-Stokes system devoid of any non-Newtonian effect.
In~\cite{Hofer-Schubert-21}, Höfer and Schubert went one step further and managed to capture Einstein's correction in the effective Stokes equation, but their analysis was restricted to the case of spherical particles: orientations then play no role and no non-Newtonian effect is described.
In the present work, we consider non-spherical particles, we describe the mean-field distribution of their orientations, and we derive a Doi-type model, rigorously describing non-Newtonian effects for the first time in a micro-macro limit. We further extend our analysis to the case of active suspensions, including the effects of particle self-propulsion: this leads to an additional elastic stress in the effective Stokes equation as was indeed predicted in~\cite{Saintillan-Shelley-08,Haines-Aranson-09,Saintillan-10,Potomkin-Ryan-Berlyand-16,Degond-19} and first derived in~\cite{Girodroux-Lavigne-22,Bernou-Duerinckx-Gloria-22} in an equilibrium setting.

We emphasize the following two fundamental limitations in the dilute mean-field regime that we consider, and we explain that they cannot be overriden due to a lack of universality of the macroscopic particle dynamics:
\begin{enumerate}[---]
\item On the one hand, we focus on the case of well-separated particles; see~assumption~\eqref{eq:prepared-thmain} below. As pointed out by Mecherbet in~\cite{Mecherbet-20}, the presence of clusters of particles would completely change the~$O(\lambda)$ dynamics since they would affect Einstein's effective viscosity formula~\eqref{eq:Einstein}. Beyond the particular case of well-separated particles, no universal macroscopic description is therefore expected to hold in general as the dynamics depends on the structure of clusters.
\smallskip\item On the other hand, even for well-separated particles, we only aim to describe the macroscopic dynamics up to accuracy $O(\lambda^2)$: no mean-field description can hold beyond this accuracy as higher-order corrections to the effective viscosity~\eqref{eq:Einstein} involve statistical information on the arrangement of particles on the microscale.
This was recently highlighted by Höfer, Mecherbet, and Schubert in~\cite{Hofer-Mecherbet-Schubert-22}, and we further refer to Remark~\ref{rem:noMFL} below for a complementary perspective: the failure of mean-field theory due to the dependence of the macroscopic dynamics on the microscopic geometry is in fact a well-known issue that occurs whenever particle interactions have critical singularity, see e.g.~\cite{Oelschlager-90,Jabin-14}.
\end{enumerate}

\begin{rem}
Note that the present contribution focusses on non-Brownian inertialess particles.
The inclusion of Brownian rotary effects on particle orientations was recently discussed in a simplified setting in~\cite{Hofer-Leocata-Mecherbet-22}, while the general mean-field description of Brownian suspensions with spatial diffusion remains a delicate open problem and is postponed to future work. The inclusion of inertial effects of the particles is also an important open problem and we refer to~\cite{Hofer-Schubert-23} for recent partial results in that direction.
\end{rem}

\subsection{Hydrodynamic model for particle suspension}
We consider a system of $N$ non-Brownian inertialess rigid particles, denoted by $I_{\e,N}^{n}$ for $1\le n\le N$, of typical size $O(\e)$, possibly self-propelling, suspended in a Stokes flow.
We start by introducing the precise model that we are going to study: we describe the set of rigid particles, then turn to their possible self-propulsion, before describing the underlying viscous solvent and the particle dynamics.
We assume that the space dimension is $d>2$ (the case $d=2$ can be treated similarly, up to obvious modifications due to the logarithmic growth of the Stokeslet in the whole plane).

\subsubsection*{$\bullet$ Elongated rigid particles}
Let $I^\circ\subset B$ be an axisymmetric connected closed set, which we take to be centered at $\int_{I^\circ}x\,dx=0$ and to be of class $C^2$. We then consider $N$ particles that are disjoint rigid copies $\{I_{\e,N}^n\}_{1\le n\le N}$ of the rescaled set~$\e I^\circ$. More precisely, each particle~$I_{\e,N}^n$ is characterized by its center~\mbox{$X_{\e,N}^n\in\R^d$} and by the direction~$R_{\e,N}^n\in\Sp^{d-1}$ of its axis, in the sense of
\[I_{\e,N}^{n}\,:=\,I_\e(X_{\e,N}^{n},R_{\e,N}^{n})\,:=\,X_{\e,N}^{n}+\e I^\circ(R_{\e,N}^n),\]
where we have set $I^\circ(r):=\Theta(r)I^\circ$,
where for a direction $r\in\Sp^{d-1}$ we denote by $\Theta(r):\R^d\to\R^d$ the rotation that maps the axis of~$I^\circ$ to~$r$ (the sign is fixed by choosing $r\mapsto\Theta(r)$ to be continuous).
The set of all rigid particles is denoted by
\[\Ic_{\e,N}\,:=\,\textstyle\bigcup_{n=1}^NI_{\e,N}^n.\]
We also consider $\e$-neighborhoods $I_{\e,N}^{n;+}:=I_{\e,N}^n+\e B$ and we assume that they are disjoint,
\begin{equation}\label{eq:disjoint-as}
I_{\e,N}^{n;+}\cap I_{\e,N}^{m;+}\,=\,\varnothing,\qquad\text{for all $1\le n\ne m\le N$},
\end{equation}
which will be shown to be preserved along the dynamics in our regime of interest.
The particle volume fraction is then
\[\lambda\,:=\,|\Ic_{\e,N}|\,=\,N|\e I^\circ|\,=\,N\e^d|I^\circ|,\]
and we shall consider the macroscopic limit $N\uparrow\infty$, $\e\downarrow0$, in the dilute regime $\lambda\ll1$.

\subsubsection*{$\bullet$ Particle activity}
We consider particles that may be active and propel themselves in the fluid (e.g.\@ by consuming some underlying chemical energy, which is not included in the model for simplicity). By a balance of forces, self-propulsion must be described by a couple of forces of same intensity and opposite direction on each rigid particle and on the surrounding fluid:
\begin{enumerate}[---]
\item Each particle $I_{\e,N}^n$ is assumed to propel itself in the direction~$R_{\e,N}^n$ of its own axis.
\smallskip\item The force of each particle on the surrounding fluid is typically exerted via a flagellar bundle, but the detail of the propulsion mechanism is not included in the model for simplicity: as e.g.\@ in~\cite[Section~2.1]{Decoene-Martin-Maury-11} (see also~\cite{Bernou-Duerinckx-Gloria-22}), we assume that the force exerted by particle~$I_{\e,N}^n$ on the surrounding fluid can be effectively described by a force field
\[f_{\e,N}^n\,:=\,f_\e(\cdot-X_{\e,N}^n,R_{\e,N}^n)~:~\R^d\setminus\Ic_{\e,N}^n\to\R^d.\]
where the function $f_\e$ takes the form $f_\e(x,r):=\e^{-d}f(\frac x\e,r)$ for some bounded function~$f:\R^d\times\Sp^{d-1}\to\R^d$. We also assume that $f(\cdot,r)$ is axisymmetric around direction~$r$, just like $I^\circ(r)$, for all~$r\in\Sp^{d-1}$.
\end{enumerate}
The balance of propulsion forces then takes form of the following assumption,
\[R_{\e,N}^n+\int_{\R^d\setminus\Ic_{\e,N}^n}f_{\e,N}^n\,=\,0,\qquad\text{for all {$1\le n\le N$}}.\]
For notational convenience, we extend inside each particle
\begin{equation}\label{eq:extend-f-I}
f_{\e,N}^n|_{\Ic_{\e,N}^n}\,:=\,R_{\e,N}^n{|I_{\e,N}^n|^{-1}},
\end{equation}
or equivalently $f(\cdot,r)|_{I^\circ(r)}=r|I^\circ|^{-1}$, so the balance of forces reads
\begin{equation}\label{eq:local-balance}
\int_{\R^d}f(\cdot,r)\,=\,0,\qquad\text{for all $r$}.
\end{equation}
We further assume that $f(\cdot,r)$ is compactly supported, say in $I^\circ(r)+B$, for all $r$, meaning that each particle propels itself only by acting on the surrounding fluid at bounded distance.
By the separation assumption~\eqref{eq:disjoint-as}, we then note that
the force fields $\{f_{\e,N}^n\}_{1\le n\le N}$ have pairwise disjoint supports.

In most previous work, e.g.~\cite{Haines-Aranson-08,Haines-Aranson-09,Girodroux-Lavigne-22}, the action of a particle on the surrounding fluid was represented for simplicity by a point force, typically setting $f(x,r):=-\delta(x-\theta r)r$ for $x\notin I^\circ(r)$, for some parameter~\mbox{$\theta\in\R$} with $\theta r\notin I^\circ(r)$. This point-force model is a special case of ours (regularity issues for $f$ play no important role), and the cases~$\theta>0$ and~$\theta<0$ then correspond to so-called puller and pusher particles, respectively. We refer e.g.\@ to~\cite[Sections~2.1--2.2]{Saintillan-Shelley-13} for a review of other models for self-propulsion, such as squirmer models, which are different but which we believe could be treated analogously.

\subsubsection*{$\bullet$ Inertialess particle dynamics in viscous solvent}
Particles are suspended in a homogeneous viscous fluid, which we assume to be described by the steady Stokes equation with unit viscosity.
More precisely, given the set $\Ic_{\e,N}$ of particles at a given time, the fluid velocity~$u_{\e,N}$ and pressure~$p_{\e,N}$ satisfy the following Stokes equation in the fluid domain,
\begin{equation}\label{eq:Stokes-1}
-\triangle u_{\e,N}+\nabla p_{\e,N}=h+\kappa\sum_{n=1}^Nf_{\e,N}^n,\qquad\Div(u_{\e,N})=0,\qquad\text{in $\R^d\setminus\Ic_{\e,N}$},
\end{equation}
where $h$ stands for some internal force in the fluid domain and where $\kappa\ge0$ is the self-propulsion intensity. We assume for convenience that $h$ is compactly supported and Lipschitz continuous.
Next, we assume that the fluid flow satisfies no-slip conditions at particle boundaries, so we may implicitly extend the fluid velocity $u_{\e,N}$ inside the particles to coincide with the particle velocities. The rigidity of the particles then translates into a boundary condition,
\begin{equation}\label{eq:Stokes-2}
\D(u_{\e,N})=0,\qquad\text{in $\Ic_{\e,N}$},
\end{equation}
where $\D(u)=\frac12(\nabla u+(\nabla u)')$ stands for the symmetric gradient. Equivalently, this means for all $1\le n\le N$,
\begin{equation}\label{eq:Stokes-2bis}
u_{\e,N}=V_{\e,N}^n+\Omega_{\e,N}^n(x-x_{\e,N}^n),\qquad\text{in $I_{\e,N}^n$},
\end{equation}
for some translational velocity $V_{\e,N}^n\in\R^d$ and angular velocity tensor $\Omega_{\e,N}^n\in\R^{d\times d}_\Skew$. As we neglect the inertia of the particles, Newton's equations of motion reduce to the balance of forces and torques, which take form of additional boundary conditions,
\begin{eqnarray}
me+\kappa R_{\e,N}^n+\int_{\partial I_{\e,N}^n}\sigma(u_{\e,N},p_{\e,N})\nu&=&0,\nonumber\\
\int_{\partial I_{\e,N}^n}(x-X_{\e,N}^n)\times\sigma(u_{\e,N},p_{\e,N})\nu&=&0,\qquad\text{for all $1\le n\le N$},\label{eq:Stokes-3}
\end{eqnarray}
where $\sigma(u,p):=2\D(u)-p\Id$ is the Cauchy stress tensor and where $me\in\R^d$ stands for the buoyancy of the particles. (Henceforth, for $a,b\in\R^d$, we use the vectorial notation $a\times b\in\R^{d\times d}_\Skew$ with $(a\times b)_{ij}:=a_i b_j-a_j b_i$.)
We choose the buoyancy~$me$ and the self-propulsion intensity $\kappa$ as
\[me\,:=\, |I_{\e,N}^n|e\,=\,\tfrac\lambda Ne,\qquad\kappa\,:=\,\tfrac{\kappa_0}{\e}|I_{\e,N}^n|\,=\,\kappa_0\tfrac{\lambda}{\e N},\]
for some $e\in\R^d$ and $\kappa_0\ge0$ with
\begin{equation}\label{eq:e-kappa0}
|e|\,\le\,1,\qquad\kappa_0\,\le\,1.
\end{equation}
This choice corresponds to the scaling that leads to~$O(\lambda)$ mean forces in the macroscopic limit.
The case $e=0$ amounts to particles with neutral buoyancy, and the case~$\kappa_0=0$ to passive particles.
Note that our scaling for the buoyancy $me$ differs from previous work on the topic~\cite{Hofer-18,Mecherbet-19,Hofer-Schubert-21}, where it was rather chosen to create a~$O(1)$ mean force in the macroscopic limit: we are not able to consider such a stronger scaling in case of non-spherical particles due to singularity issues in the mean-field analysis of orientations.\footnote{In particular, while a drag force of order $O(\frac1\lambda\e^2)$ appeared in the Vlasov equation in~\cite{Hofer-18,Mecherbet-19,Hofer-Schubert-21} due to the buoyancy, this term would become~$O(\e^2)$ in our scaling and is neglected in this work.}

\medskip
Summing up, given the set $\Ic_{\e,N}$ of particles at a given time, the instantaneous fluid velocity $u_{\e,N}$ is obtained as the unique weak solution in~$\dot H^1(\R^d)^d$ of the Stokes problem~\mbox{\eqref{eq:Stokes-1}--\eqref{eq:Stokes-3}}, that is,
\begin{equation}\label{eq:main-Stokes}
\left\{\begin{array}{ll}
-\triangle u_{\e,N}+\nabla p_{\e,N}=h+\kappa_0\tfrac{\lambda}{\e N}\sum_{n=1}^Nf_{\e,N}^n,&\text{in $\R^d\setminus\Ic_{\e,N}^n$},\\
\Div(u_{\e,N})=0,&\text{in $\R^d\setminus\Ic_{\e,N}^n$},\\
\D(u_{\e,N})=0,&\text{in $\Ic_{\e,N}^n$},\\
\frac\lambda Ne+\kappa_0\frac{\lambda}{\e N} R_{\e,N}^n+\int_{\partial{I_{\e,N}^n}}\sigma(u_{\e,N},p_{\e,N})\nu=0,&\text{for all $1\le n\le N$},\\
\int_{\partial{I_{\e,N}^n}}(x-X_{\e,N}^n)\times\sigma(u_{\e,N},p_{\e,N})\nu=0,&\text{for all $1\le n\le N$}.
\end{array}\right.
\end{equation}
We recall that the weak formulation of this system takes on the following simple guise: for any test function~$v\in\dot H^1(\R^d)^d$ that is incompressible, i.e.\@ $\Div(v)=0$, and that is rigid inside particles, i.e.\@ $\D(v)=0$ in~$\Ic_{\e,N}$, we have
\begin{equation}\label{eq:main-Stokes-weak}
2\int_{\R^d\setminus\Ic_{\e,N}}\D(v):\D(u_{\e,N})=
\int_{\R^d}v\cdot \Big(h\mathds1_{\R^d\setminus\Ic_{\e,N}}+e\mathds1_{\Ic_{\e,N}}+\kappa_0\tfrac{\lambda}{\e N}\sum_{n=1}^Nf_{\e,N}^n\Big).
\end{equation}
Once the Stokes system~\eqref{eq:main-Stokes} is solved for the instantaneous fluid velocity $u_{\e,N}$, particle positions and orientations can be updated according to
\begin{equation}\label{eq:dynamics}
\partial_tX_{\e,N}^n=V_{\e,N}^n,\qquad\partial_tR_{\e,N}^n=\Omega_{\e,N}^nR_{\e,N}^n,\qquad\text{for all $1\le n\le N$},
\end{equation}
where $V_{\e,N}^n,\Omega_{\e,N}^n$ are given by~\eqref{eq:Stokes-2bis}, or alternatively
\begin{equation}\label{eq:VTheta}
V_{\e,N}^n\,:=\,\fint_{I_{\e,N}^n}u_{\e,N}~\in\R^d,\qquad\Omega_{\e,N}^n\,:=\,\fint_{I_{\e,N}^n}\nabla u_{\e,N}~\in\R^{d\times d}_\Skew.
\end{equation}
In this way, the particles follow the fluid flow and interact with one another via the flow disturbance that they generate. The resulting dynamics is reputedly complex in view of the multibody, long-range, and singular nature of hydrodynamic interactions.

\subsection{Semi-dilute mean-field description}
We aim to investigate the collective macroscopic behavior of the fluid velocity~$u_{\e,N}$ and of the particle empirical measures
\begin{eqnarray}
\nu_{\e,N}&:=&\sum_{n=1}^N\delta_{(X_{\e,N}^n,R_{\e,N}^n)}~\in~\Pc(\R^d\times\Sp^{d-1}),\label{eq:def-empirical}\\
\mu_{\e,N}&:=&\sum_{n=1}^N\delta_{X_{\e,N}^n}~\in~\Pc(\R^d).\nonumber
\end{eqnarray}
In the macroscopic limit $N\uparrow\infty$, $\e\downarrow0$, in the dilute regime $\lambda\ll1$, {provided that particles do not form clusters,} formal considerations lead to expect
\begin{equation}\label{eq:formal-exp}
(u_{\e,N},\mu_{\e,N})=(u_{\lambda,\e},\mu_{\lambda,\e})+O((\lambda+\e)^2),\qquad\nu_{\e,N}=\nu_{\lambda,\e}+O(\lambda+\e),
\end{equation}
where $(u_{\lambda,\e},\nu_{\lambda,\e},\mu_{\lambda,\e})$ solves the following coupled system on $\R^d$: the macroscopic fluid velocity~$u_{\lambda,\e}$ satisfies an effective Stokes equation,
\begin{equation}\label{eq:Doi0-1}
\hspace{-0.2cm}\left\{\begin{array}{l}
-\Div\big[2\big(1+\lambda\langle\Sigma^\circ\nu_{\lambda,\e}\rangle\big)\!\D(u_{\lambda,\e})\big]+\nabla p_{\lambda,\e}\\
\hspace{2cm}=(1-\lambda\mu_{\lambda,\e})h+\lambda\mu_{\lambda,\e} e+\kappa_0\lambda\beta_f\,\Div\big[\langle(r\otimes r-\tfrac1d\Id)\,\nu_{\lambda,\e}\rangle\big],\\
\Div(u_{\lambda,\e})=0,\\
\mu_{\lambda,\e}=\langle\nu_{\lambda,\e}\rangle,\\
\end{array}\right.
\end{equation}
which is coupled to a Vlasov equation for the mean-field distribution $\nu_{\lambda,\e}$ of particle positions and orientations,
\begin{equation}\label{eq:Doi0-2}
\hspace{-0.2cm}\left\{\begin{array}{l}
\partial_t\nu_{\lambda,\e}+\Div_x\big[\nu_{\lambda,\e} (u_{\lambda,\e}+\kappa_0\e \alpha_fr)\big]+\Div_r\big[\nu_{\lambda,\e}(\Omega^\circ\nabla u_{\lambda,\e})r\big]=0,\\
\nu_{\lambda,\e}|_{t=0}=\nu^\circ.
\end{array}\right.
\end{equation}
Here, we use the short-hand notation $\langle g\rangle(x):=\int_{\Sp^{d-1}}g(x,r)\,d\sigma(r)$ for angular averaging,
and the coefficient fields $r\mapsto\Sigma^\circ(r),\Omega^\circ(r)$ and the constants $\alpha_f,\beta_f$ are defined as follows:
\begin{enumerate}[---]
\item \emph{Passive effective viscosity:}\\
The coefficient $1+\lambda\langle\Sigma^\circ\nu_{\lambda,\e}\rangle$ in the effective Stokes equation~\eqref{eq:Doi0-1} corresponds to Einstein's formula for the dilute correction of the plain fluid viscosity due to the presence of rigid particles, see e.g.~\cite[Section~2]{DG-21b} or~\cite{Hillairet-Wu-20}: for all~$r\in\Sp^{d-1}$, the tensor $\Sigma^\circ(r)$ is the symmetric linear map on the set $\R^{d\times d}_{\Sym,0}$ of trace-free symmetric matrices, given by
\begin{equation}\label{eq:def-S(r)}
E:\Sigma^\circ(r)E\,:=\,|I^\circ|^{-1}\int_{\R^d}|\!\D(u_{r,E}^\circ)|^2,\qquad\text{for all $E\in\R^{d\times d}_{\Sym,0}$,}
\end{equation}
where $u_{r,E}^\circ$ is the unique decaying solution of the single-particle problem
\begin{equation}\label{eq:eqn-uRE0}
\left\{\begin{array}{ll}
-\triangle u^\circ_{r,E}+\nabla p^\circ_{r,E}=0,&\text{in $\R^d\setminus I^\circ(r)$},\\
\Div(u^\circ_{r,E})=0,&\text{in $\R^d\setminus I^\circ(r)$},\\
\D(u^\circ_{r,E}+Ex)=0,&\text{in $I^\circ(r)$},\\
\int_{\partial I^\circ(r)}\sigma(u^\circ_{r,E},p^\circ_{r,E})\nu=0,&\\
\int_{\partial I^\circ(r)}x\times\sigma(u^\circ_{r,E},p^\circ_{r,E})\nu=0,&
\end{array}\right.
\end{equation}
which describes the flow disturbance generated by a strain rate $E$ at a single rigid particle~$I^\circ(r)$ oriented in direction $r$. In other words, $\Sigma^\circ(r)E$ measures the reaction of a particle oriented in direction $r$ to a strain rate $E$, and it can equivalently be written as (half) the associated stresslet,
\begin{equation}
\Sigma^\circ(r)E\,=\,\tfrac12|I^\circ|^{-1}\int_{\partial I^\circ(r)} \sigma(u_{r,E}^\circ+Ex,p_{r,E}^\circ)\nu\otimes_s^\circ x,\label{eq:def-S(r)-bis}
\end{equation}
where $a\otimes_s^\circ b:=\frac12(a\otimes b+b\otimes a)-\frac1d\Id(a\cdot b)$ is the trace-free symmetric tensor product. Note that the map $r\mapsto \Sigma^\circ(r)$ is easily checked to be smooth.

\smallskip
\item \emph{Passive particle rotation:}\\
The local fluid deformation makes each particle rotate in a nontrivial way:
in the dilute regime, we naturally define $\Omega^\circ(r)H$ as the angular velocity of a single particle oriented in direction $r$ due to a local fluid deformation~$H$. More precisely, $\Omega^\circ(r)$ is the linear map $\R^{d\times d}_0\to\R^{d\times d}_\Skew$ given by
\begin{equation}\label{eq:defin-Omega0}
\Omega^\circ(r)H\,:=\,\fint_{I^\circ(r)}\nabla u^\circ_{r,H^\Sym}+H,\qquad\text{for all $H\in\R^{d\times d}_0$},
\end{equation}
where $u^\circ_{r,H^\Sym}$ is the solution of~\eqref{eq:eqn-uRE0} with $E=H^\Sym$ the symmetric part of $H$.
Note that the map $r\mapsto \Omega^\circ(r)$ is also easily checked to be smooth.

\smallskip
\item \emph{Active elastic stress:}\\
The particle self-propulsion generates an elastic stress in the Stokes equation (cf.~last right-hand side term in~\eqref{eq:Doi0-1}): in the dilute regime, it is naturally given in terms of the stresslet
\begin{equation}\label{eq:def-S(r)gamma-re}
\Sigma^\circ_f(r)\,:=\,\int_{\partial I^\circ(r)}\sigma(u^\circ_{r,f},p^\circ_{r,f})\nu\otimes_s^\circ x-\int_{\R^d}f(\cdot,r)\otimes^\circ x,
\end{equation}
where $a\otimes^\circ b:=a\otimes b-\frac1d\Id(a\cdot b)$ is the trace-free tensor product
and where~$u_{r,f}^\circ$ is the unique decaying solution of the single-particle problem
\begin{equation}\label{eq:eqn-defin-ugamma}
\left\{\begin{array}{ll}
-\triangle u^\circ_{r,f}+\nabla p^\circ_{r,f}=f(\cdot,r),&\text{in $\R^d\setminus I^\circ(r)$},\\
\Div(u^\circ_{r,f})=0,&\text{in $\R^d\setminus I^\circ(r)$},\\
\D(u^\circ_{r,f})=0,&\text{in $I^\circ(r)$},\\
r+\int_{\partial I^\circ(r)}\sigma(u^\circ_{r,f},p^\circ_{r,f})\nu=0,&\\
\int_{\partial I^\circ(r)}x\times\sigma(u^\circ_{r,f},p^\circ_{r,f})\nu=0,&
\end{array}\right.
\end{equation}
which describes the flow disturbance generated by the self-propulsion of a single particle oriented in direction $r$.
Since the inclusion $I^\circ(r)$ and the propulsion force $f(\cdot,r)$ are both axisymmetric around direction $r$, the single-particle stresslet~\eqref{eq:def-S(r)gamma-re} can be written by symmetry as
\begin{equation}\label{eq:def-ugamma-sym}
\Sigma_f^\circ(r)\,=\,\beta_f \,r\otimes^\circ r,\qquad\text{for some $\beta_f\in\R$.}
\end{equation}
The cases $\beta_f>0$ and $\beta_f<0$ correspond to so-called puller and pusher particles, respectively.
In the macroscopic limit, the active elastic stress is then given by the angular average
\[\langle\Sigma_f^\circ\nu_{\lambda,\e}\rangle\,=\,\beta_f\big\langle(r\otimes r-\tfrac1d\Id)\nu_{\lambda,\e}\big\rangle,\]
which appears as the last right-hand side term in the effective Stokes equation~\eqref{eq:Doi0-1}. It coincides with the expression predicted e.g.\@ in~\cite{Saintillan-Shelley-08,Haines-Aranson-09,Saintillan-10,Potomkin-Ryan-Berlyand-16,Degond-19} and first derived in~\cite{Girodroux-Lavigne-22,Bernou-Duerinckx-Gloria-22} in the equilibrium setting.

\smallskip
\item \emph{Swimming velocities:}\\
The particle self-propulsion also generates a drag force on the particles, leading to effective swimming velocities: in the dilute regime, we naturally define $V_f^\circ(r)$ as the drag velocity of a single particle oriented in direction $r$ due to its self-propulsion,
\begin{equation}\label{eq:def-act-drag-0}
V^\circ_f(r)\,:=\,\int_{I^\circ(r)}u^\circ_{r,f}~~\in~~\R^d.
\end{equation}
Since the propulsion force $f(\cdot,r)$ is axisymmetric around $r$, this swimming velocity can be written by symmetry as
\begin{equation}\label{eq:def-act-drag}
V^\circ_f(r)\,=\,\alpha_f\,r,\qquad\text{for some $\alpha_f>0$}.
\end{equation}
\end{enumerate}

The above kinetic model~\eqref{eq:Doi0-1}--\eqref{eq:Doi0-2} is a variant of the so-called Doi model, which was first introduced in~\cite{Jeffery-22,Hinch-Leal-72,Brenner-74} for passive suspensions $\kappa_0=0$ (see also~\cite{Doi-Edwards-78,Doi-81,Doi-Edwards-88}), and which was adapted to active suspensions $\kappa_0\ne0$ by Saintillan and Shelley~\cite{Saintillan-Shelley-08,Saintillan-10,Saintillan-18}. The difference with the standard form of the Doi model is twofold:
\begin{enumerate}[---]
\item Brownian effects are not considered in the present work.
The inclusion of Brownian rotary effects on particle orientations was recently discussed in a simplified setting in~\cite{Hofer-Leocata-Mecherbet-22}, but the general mean-field description of Brownian suspensions with spatial diffusion remains a delicate open problem and is postponed to future work.
\smallskip\item While we consider particles with a given axisymmetric shape, the Doi model rather corresponds to the limit of very elongated particles as computed by slender-body theory, e.g.~\cite{Brenner-74}. This amounts to replacing the effective coefficients $\Sigma^\circ,\Omega^\circ$ in~\eqref{eq:Doi0-1}--\eqref{eq:Doi0-2} by
\begin{eqnarray*}
\Sigma^\circ(r)&\leadsto& \alpha_1 (r\otimes^\circ r)\otimes (r\otimes^\circ r),\\
\Omega^\circ(r)(\nabla u)r&\leadsto&\big(\!\Id-r\otimes r\big)\big(\alpha_2 \D(u)+\tfrac12\nabla\times u\big)\,r,
\end{eqnarray*}
for some shape factors $\alpha_1,\alpha_2>0$.
\end{enumerate}
Regardless of these differences, just as the usual Doi model, the kinetic model~\eqref{eq:Doi0-1}--\eqref{eq:Doi0-2} describes the emergence of non-Newtonian effects due to mean-field particle orientations: particle orientations adapt collectively to the local fluid deformation and in turn modify the effective viscosity via Einstein's effective viscosity formula $1+\lambda\langle\Sigma^\circ\nu_{\lambda,\e}\rangle$.
We refer in particular to~\cite{Helzel-Otto-06,Otto-Tzavaras-08} for a study of some properties of the Doi model.

\subsection{Main result}
Standard arguments ensure the well-posedness of the particle dynamics~\eqref{eq:main-Stokes}--\eqref{eq:VTheta} until the first collision time, and Hillairet and Sabbagh~\cite{Hillairet-Sabbagh-23} have recently shown that collisions cannot occur in finite time.
As explained in Section~\ref{sec:overview}, since the presence of clusters of particles would affect the $O(\lambda)$ macroscopic dynamics, we shall require an additional control on the minimal interparticle distance,
\begin{equation*}
\dd_{\e,N}^{\min}\,:=\,\min_{1\le n\ne m\le N}|X_{\e,N}^n-X_{\e,N}^m|.
\end{equation*}
More precisely, we shall assume that particles are well-separated initially and do not form clusters, in the sense that $\dd_{\e,N}^{\min}(0)\gtrsim N^{-\frac1d}$ (see~\eqref{eq:prepared-thmain} below). As e.g.\@ in~\cite{Hofer-18,Mecherbet-19}, we shall check that under suitable assumptions this is propagated in finite time. (Note that the condition $\dd_{\e,N}^{\min}\gtrsim N^{-1/d}$ ensures in particular the validity of~\eqref{eq:disjoint-as} as we have $\e\ll N^{-1/d}$ in the dilute regime~$\lambda\ll1$.)

We turn to the justification of the kinetic model~\eqref{eq:Doi0-1}--\eqref{eq:Doi0-2} in the macroscopic limit.
Since the Wasserstein distance between empirical measures and their mean-field approximation cannot be smaller than the typical interparticle distance $O(N^{-1/d})$, the kinetic model can be derived at best with accuracy $O((\lambda+\e)^2+N^{-1/d})$. Therefore, as swimming velocities in~\eqref{eq:Doi0-1}--\eqref{eq:Doi0-2} are of order $O(\kappa_0\e)$ and as we have $\e\ll N^{-1/d}$ in the dilute regime, we cannot capture their effect in the limit.\footnote{This would motivate to rather consider a different regime with larger self-propulsion intensity $\kappa_0\gg1$. However, we do not manage to control interparticle distances in that case as the Lipschitz estimate of Proposition~\ref{prop:Lip-est} would only hold in the form $\|u_{\e,N}\|_{W^{1,\infty}(\R^d)}\le C_h\kappa_0$, cf.~\eqref{eq:Lipestim-pr}.}
We are thus led to the following simplified model, where swimming velocities are neglected,
\begin{equation}\label{eq:Doi0}
\hspace{-0.2cm}\left\{\begin{array}{l}
-\Div\big[2\big(1+\lambda\langle\Sigma^\circ\nu_{\lambda}\rangle\big)\!\D(u_{\lambda})\big]+\nabla p_{\lambda}\\
\hspace{2cm}=(1-\lambda\mu_{\lambda})h+\lambda\mu_{\lambda} e+\kappa_0\lambda\beta_f\,\Div\big[\langle(r\otimes r-\tfrac1d\Id)\nu_{\lambda}\rangle\big],\\
\Div(u_{\lambda})=0,\\
\mu_{\lambda}=\langle\nu_{\lambda}\rangle,\\
\partial_t\nu_{\lambda}+\Div_x(\nu_{\lambda} u_{\lambda})+\Div_r(\nu_{\lambda}(\Omega^\circ\nabla u_{\lambda})r)=0,\\
\nu_{\lambda}|_{t=0}=\nu^\circ.
\end{array}\right.
\end{equation}
We start by stating the (perturbative) well-posedness of this system in the smooth class. The proof is standard and is included in Section~\ref{sec:well-posed-Doi} for completeness.

\begin{prop}[Well-posedness of macroscopic model]\label{prop:well-posed-Doi}$ $
Given $T<\infty$, given~\mbox{$\gamma>1$} non-integer, given an initial condition $\nu^\circ\in \Pc\cap W^{1,1}\cap W^{\gamma,\infty}(\R^d\times\Sp^{d-1})$, and given $0<\lambda\ll1$ small enough (depending on $T,h,\gamma,\nu^\circ$), there is a unique solution $(\nu_\lambda,u_\lambda)$ of~\eqref{eq:Doi0} up to time~$T$ with $\nu_\lambda\in \Ld^\infty([0,T];\Pc\cap W^{1,1}\cap W^{\gamma,\infty}(\R^d\times\Sp^{d-1}))$ and $u_\lambda\in \Ld^\infty([0,T];W^{\gamma+1,\infty}(\R^d)^d)$.
\end{prop}

Our main result concerns the derivation of this Doi-type kinetic model~\eqref{eq:Doi0} from the particle dynamics~\eqref{eq:main-Stokes}--\eqref{eq:VTheta}, thus providing a rigorous version of~\eqref{eq:formal-exp}.
To our knowledge, this is the first time that non-Newtonian macroscopic models are rigorously derived from a hydrodynamic description of non-Brownian particle suspensions.
It constitutes both a generalization of~\cite{Hofer-Schubert-21} to non-spherical and possibly active particles, and a generalization of~\cite{DG-22a,Girodroux-Lavigne-22} to the particle dynamics.
 
\begin{theor}[Semi-dilute mean-field approximation]\label{th:main}
Let $1\le p\le\infty$.
Assume that initial particle positions satisfy, for some $C_0\ge1$,
\begin{equation}\label{eq:prepared-thmain}
\dd_{\e,N}^{\min}(0)\,\ge\,\tfrac1{C_0}N^{-\frac1d},\qquad\max_{1\le n\le N}|X_{\e,N}^{n;\circ}|\,\le\,C_0,
\end{equation}
and assume that the empirical measure~\eqref{eq:def-empirical} is initially close in the $p$-Wasserstein metric to some compactly supported density $\nu^\circ\in \Pc\cap W^{\gamma,\infty}(\R^d\times\Sp^{d-1})$ with $\gamma>1$,
\[W_p(\nu_{\e,N}^\circ,\nu^\circ)\,\le\,C_0\big(\lambda\log N+N^{-\frac1d}\big).\]
Let $\mu^\circ:=\langle\nu^\circ\rangle$ be the corresponding spatial density.
Given $T<\infty$, assume that $\lambda\log N\ll1$ is small enough (depending on $T,h,\nu^\circ$), and let $(\nu_\lambda,\mu_\lambda,u_\lambda)$ be the unique solution of~\eqref{eq:Doi0} up to time~$T$ as given by Proposition~\ref{prop:well-posed-Doi}.
Then we have for all~$t\in[0, T]$,
\begin{eqnarray*}
W_p(\mu_{\e,N}^t,\mu^t_\lambda)&\lesssim&\lambda^2|\!\log\lambda|\log N+N^{-\frac1d}+W_p(\mu_{\e,N}^\circ,\mu^\circ),\\
\|u_{\e,N}^t- u_\lambda^t\|_{(\Ld^1+\Ld^\infty)(\R^d)}
&\lesssim&\lambda^2|\!\log\lambda|\log N
+N^{-\frac1d},\\
W_p(\nu_{\e,N}^t,\nu^t_\lambda)&\lesssim&\lambda\log N+N^{-\frac1d},
\end{eqnarray*}
up to multiplicative constants depending on $C_0,t,h,p,\nu^\circ$.
\end{theor}

We emphasize again that the mean-field description cannot be pursued beyond the accuracy~$O(\lambda^2)$: the next-order $O(\lambda^2)$ correction to the approximate effective viscosity $1+\lambda\langle\Sigma^\circ\nu_{\lambda,\e}\rangle$ in~\eqref{eq:Doi0} would involve the statistical distribution of pairs of particles on the microscale, and such geometric information is beyond the scope of mean-field theory.

\medskip
\subsection*{Notation}
\begin{enumerate}[$\bullet$]
\item In order to control the particle dynamics, on top of the minimal interparticle distance,
\begin{equation}\label{eq:def-min}
\dd_{\e,N}^{\min}\,:=\,\min_{1\le n\ne m\le N}|X_{\e,N}^n-X_{\e,N}^m|,
\end{equation}
we shall also keep track of the following quantities, as in~\cite{Hofer-18,Hofer-Schubert-21}, for $\sigma\in[0,d]$,
\begin{equation}\label{eq:def-alpha}
\alpha_{\e,N}^\sigma\,:=\,\max_{1\le n\le N}\tfrac1N\sum_{m:m\ne n}^N|X_{\e,N}^n-X_{\e,N}^m|^{\sigma-d}.
\end{equation}
\item We denote by $\R^{d\times d}_0$, $\R^{d\times d}_{\Sym,0}$, and $\R^{d\times d}_{\Skew}$ the set of trace-free matrices, trace-free symmetric matrices, and skew-symmetric matrices, respectively.
\item For $a,b\in\R^d$, we denote by $a\otimes_sb:=\frac12(a\otimes b+b\otimes a)$ the symmetric tensor product, by $a\otimes^\circ b:=a\otimes b-\frac1d\Id(a\cdot b)$ the trace-free tensor product, and by $a\otimes_s^\circ b:=a\otimes_s b-\frac1d\Id(a\cdot b)$ the trace-free symmetric tensor product. We use the vectorial notation $a\times b\in\R^{d\times d}_\Skew$ with $(a\times b)_{ij}:=a_i b_j-a_j b_i$. For matrices $A,B$, we let $A:B:=A_{ij}B_{ij}$, systematically using Einstein's summation convention on repeated indices. We also use the notation~$A^\Sym$ for the symmetric part of $A$, that is, $(A^\Sym)_{ij}:=\frac12(A_{ij}+A_{ji})$.
\item For a vector field $u$ and a matrix field $T$, we set $(\nabla u)_{ij}:=\nabla_j u_i$, $\D( u):=(\nabla u)^{\Sym}$, $(\nabla T)_{ijk}:=\nabla_k T_{ij}$, $\Div(T)_i:=\nabla_j T_{ij}$.
For a pressure field $p$, we denote the Cauchy stress tensor by $\sigma(u,p):=2\!\D(u)-p\Id$.
We define $(T\ast u)_i:=T_{ij}\ast u_j$ as the convolution product of a vector field with a matrix kernel, and similarly for higher-order tensors.
\item We use the short-hand notation $g\,\hat\ast\,\mu(x):=\int_{\R^d\setminus\{x\}}g(x-y)\,d\mu(y)$ for the diagonal-free convolution, which is equivalent to standard convolution if the measure $\mu$ is continuous.
\item We let $\langle g\rangle(x):=\int_{\Sp^{d-1}}g(x,r)\,d\sigma(r)$ be the angular averaging of a function $g$ on~\mbox{$\R^d\times\Sp^{d-1}$}.
\item We denote by $C\ge1$ any constant that only depends on the dimension $d$, on the~$C^2$ property of $I^\circ$, and on the propulsion force $f$. We use the notation $\lesssim$ (resp.~$\gtrsim$) for $\le C\times$ (resp.~$\ge\frac1C\times$) up to such a multiplicative constant $C$. We write $\ll$ (resp.~$\gg$) for $\le C\times$ (resp.~$\ge C\times$) up to a sufficiently large multiplicative constant $C$. We add subscripts to indicate dependence on other parameters.
\item The ball centered at $x$ of radius $r$ in $\R^d$ is denoted by $B_r(x)$, and we set $B:=B_1(0)$.
\end{enumerate}

\medskip
\section{Preliminary on Stokes analysis}
In this section, we recall a series of preliminary results for the analysis of the steady Stokes equation with rigid inclusions.
We start with the following standard lemma, showing how rigidity constraints can be viewed as creating source terms concentrated at particle boundaries in the Stokes equation; a short proof is included for convenience.

\begin{lem}\label{lem:eqn-pseudo}
Given $h\in\Ld^{2d/(d+2)}(\R^d)^d$, if $u\in\dot H^1(\R^d)^d$ satisfies
\begin{equation}\label{eq:Stokes-rewr}
\left\{\begin{array}{ll}
-\triangle u+\nabla p=h,&\text{in $\R^d\setminus\Ic_{\e,N}$},\\
\Div(u)=0,&\text{in $\R^d\setminus\Ic_{\e,N}$},\\
\D(u)=0,&\text{in $\Ic_{\e,N}$},\end{array}\right.
\end{equation}
then the following relation holds in the weak sense in the whole space $\R^d$,
\[-\triangle u+\nabla(p\mathds1_{\R^d\setminus\Ic_{\e,N}})\,=\,h\mathds1_{\R^d\setminus\Ic_{\e,N}}-\sum_{n=1}^N\delta_{\partial I_{\e,N}^n}\sigma(u,p)\nu.\qedhere\]
\end{lem}

\begin{proof}
For a test function $v\in \dot H^1(\R^d)^d$, the incompressibility of $u$ and the rigidity constraint in $\Ic_{\e,N}$ yield
\[\int_{\R^d}\nabla v:\nabla u\,=\,2\int_{\R^d}\nabla v:\D(u)\,=\,2\int_{\R^d\setminus\Ic_{\e,N}}\nabla v:\D(u),\]
or equivalently, inserting the definition of the Cauchy stress tensor $\sigma(u,p)=2\D(u)-p\Id$,
\[\int_{\R^d}\nabla v:\nabla u\,=\,\int_{\R^d\setminus\Ic_{\e,N}}\nabla v:p\Id+\int_{\R^d\setminus\Ic_{\e,N}}\nabla v:\sigma(u,p).\]
Integrating by parts in the last right-hand side term and noting that the Stokes equation in~\eqref{eq:Stokes-rewr} yields $-\Div(\sigma(u,p))=h$ in $\R^d\setminus\Ic_{\e,N}$, we deduce
\[\int_{\R^d}\nabla v:\nabla u
\,=\,\int_{\R^d\setminus\Ic_{\e,N}}\nabla v:p\Id+\int_{\R^d\setminus\Ic_{\e,N}}v\cdot h-\sum_{n=1}^N\int_{\partial I^n_{\e,N}}v\cdot\sigma(u,p)\nu,\]
which is the conclusion.
\end{proof}

The following basic trace estimate is used repeatedly to control force terms concentrated at particle boundaries; a short proof is also included for convenience.

\begin{lem}[Trace estimate]\label{lem:trace}
Given $1\le n\le N$, if $(u,p)\in H^1(I_{\e,N}^{n;+})^d\times\Ld^2(I_{\e,N}^{n;+})$ satisfies
\begin{equation}\label{eq:lem-trace-eqn}
\left\{\begin{array}{ll}
-\triangle u+\nabla p=0,&\text{in $I_{\e,N}^{n;+}\setminus I_{\e,N}^n$},\\
\Div(u)=0,&\text{in $I_{\e,N}^{n;+}\setminus I_{\e,N}^n$},\\
\D(u)=0,&\text{in $I_{\e,N}^n$},\\
\int_{\partial I_{\e,N}^n}\sigma(u,p)\nu=0,\\
\int_{\partial I_{\e,N}^n}(x-X_{\e,N}^n)\times\sigma(u,p)\nu=0,\\
\end{array}\right.
\end{equation}
then we have for all $F\in H^1_\loc(\R^d)^d$ with $\Div(F)=0$,
\[\Big|\int_{\partial I_{\e,N}^n}F\cdot\sigma(u,p)\nu\Big|\,\lesssim\,\|\!\D(F)\|_{\Ld^2(I_{\e,N}^{n;+})}\|\!\D(u)\|_{\Ld^2(I_{\e,N}^{n;+})},\]
where we recall $I_{\e,N}^{n;+}=I_{\e,N}^n+\e B$.
\end{lem}

\begin{proof}
The condition $\int_{\partial I_{\e,N}^n}\sigma(u,p)\nu=0$ allows to rewrite
\[\int_{\partial I_{\e,N}^n}F\cdot\sigma(u,p)\nu\,=\,\int_{\partial I_{\e,N}^n}\Big(F-\fint_{I_{\e,N}^{n;+}}F\Big)\cdot\sigma(u,p)\nu.\]
Choosing a cut-off function $\chi_\e$ with
\[\chi_\e|_{I_{\e,N}^n}=1,\qquad\supp\chi_\e\subset I_{\e,N}^{n;+},\qquad|\nabla\chi_\e|\lesssim\e^{-1},\]
and noting that equation~\eqref{eq:lem-trace-eqn} yields $\Div(\sigma(u,p))=0$ in the annulus $I_{\e,N}^{n;+}\setminus I_{\e,N}^n$,
we find by integration by parts,
\[\int_{\partial I_{\e,N}^n}F\cdot\sigma(u,p)\nu=\int_{I_{\e,N}^{n;+}\setminus I_{\e,N}^n}\nabla\bigg(\Big(F-\fint_{I_{\e,N}^{n;+}}F\Big)\chi_\e\bigg):\sigma(u,p).\]
Hence, by the Cauchy--Schwarz inequality followed by Poincaré's inequality, using properties of the cut-off function $\chi_\e$,
\[\bigg|\int_{\partial I_{\e,N}^n}F\cdot\sigma(u,p)\nu\bigg|\,\lesssim\,\|\nabla F\|_{\Ld^2(I_{\e,N}^{n;+})}\|\sigma(u,p)\|_{\Ld^2(I_{\e,N}^{n;+}\setminus I_{\e,N}^n)}.\]
For any $\Omega\in\R^{d\times d}_{\Skew}$, as the last condition in~\eqref{eq:lem-trace-eqn} yields
\[\int_{\partial I_{\e,N}^n}\Omega(x-X_{\e,N}^n)\cdot\sigma(u,p)\nu=0\]
we can subtract $\Omega(x-X_{\e,N}^n)$ to $F$ in the above estimate, to the effect of
\[\bigg|\int_{\partial I_{\e,N}^n}F\cdot\sigma(u,p)\nu\bigg|\,\lesssim\,\|\nabla F-\Omega\|_{\Ld^2(I_{\e,N}^{n;+})}\|\sigma(u,p)\|_{\Ld^2(I_{\e,N}^{n;+}\setminus I_{\e,N}^n)},\]
and thus, taking the infimum over $\Omega\in\R^{d\times d}_{\Skew}$ and appealing to Korn's inequality,
\[\bigg|\int_{\partial I_{\e,N}^n}F\cdot\sigma(u,p)\nu\bigg|\,\lesssim\,\|\!\D(F)\|_{\Ld^2(I_{\e,N}^{n;+})}\|\sigma(u,p)\|_{\Ld^2(I_{\e,N}^{n;+}\setminus I_{\e,N}^n)}.\]
It remains to estimate the pressure field $p$ in the right-hand side. As the incompressibility of $F$ implies $\int_{\partial I_{\e,N}^n}F\cdot\nu=0$, any constant can be subtracted to the pressure~$p$ in the above, hence in particular
\[\bigg|\int_{\partial I_{\e,N}^n}F\cdot\sigma(u,p)\nu\bigg|\,\lesssim\,\|\!\D(F)\|_{\Ld^2(I_{\e,N}^{n;+})}\bigg(\|\!\D(u)\|_{\Ld^2(I_{\e,N}^{n;+})}+\Big\|p-\fint_{I_{\e,N}^{n;+}\setminus I_{\e,N}^n}p\Big\|_{\Ld^2(I_{\e,N}^{n;+}\setminus I_{\e,N}^n)}\bigg).\]
Now appealing to a local pressure estimate for the steady Stokes equation, which follows from a standard argument based on the Bogovskii operator, e.g.~\cite[Lemma~3.3]{DG-22a}, the conclusion follows.
\end{proof}

Next, we recall the usual definition and the pointwise decay of the Stokeslet $\Gc$, which is the Green's function for the steady Stokes equation.

\begin{lem}[Stokeslet]\label{lem:Green}
For all $1\le i\le d$, we can define $\Gc_{i}\in W^{1,1}_\loc(\R^d)^d$ as the unique decaying distributional solution of
\[-\triangle\Gc_i+\nabla P_i=\ee_i\delta_0,\qquad\Div(\Gc_i)=0,\qquad\text{in $\R^d$},\]
and we then set $\Gc_{i}=(\Gc_{ij})_{1\le j\le d}$ and $\Gc=(\Gc_{ij})_{1\le i,j\le d}$.
This Green's function is explicitly given by
\begin{equation*}
\Gc(x)\,=\,\tfrac{1}{2(d-2)|\partial B|}|x|^{2-d}\Big(\Id+(d-2)\tfrac{x}{|x|}\otimes\tfrac{x}{|x|}\Big),
\end{equation*}
hence it satisfies the following pointwise estimates,
\[|\Gc(x)|\,\lesssim\,|x|^{2-d},\qquad|\nabla\Gc(x)|\,\lesssim\,|x|^{1-d},\qquad|\nabla^2\Gc(x)|\,\lesssim\,|x|^{-d}.\qedhere\]
\end{lem}

We also define a corresponding notion of Stokeslet for the steady Stokes problem with a single rigid inclusion, which satisfies a similar pointwise decay. This is a particular case of the analysis in~\cite[Appendix~A]{DG-21b}, where we further get a corresponding result for any finite family of well-separated rigid inclusions.

\begin{lem}[Stokeslet with rigid inclusions; \cite{DG-21b}]\label{lem:Green-rig}
For all $1\le n\le N$, $y\in\R^d\setminus I_{\e,N}^n$, and $1\le i\le d$, we can define $\Gc_{\e,N;i}^n(\cdot,y)\in W^{1,1}_\loc(\R^d)^d$ as the unique decaying distributional solution of
\[\left\{\begin{array}{ll}
-\triangle\Gc^{n}_{\e,N;i}(\cdot,y)+\nabla P^{n}_{\e,N;i}(\cdot,y)=\ee_i\delta_y,&\text{in $\R^d\setminus I^n_{\e,N}$},\\
\Div(\Gc^n_{\e,N;i}(\cdot,y))=0,&\text{in $\R^d\setminus I^n_{\e,N}$},\\
\D(\Gc^n_{\e,N;i}(\cdot,y))=0,&\text{in $I^n_{\e,N}$},\\
\int_{\partial I_{\e,N}^n}\sigma(\Gc^n_{\e,N;i}(\cdot,y),P^n_{\e,N;i}(\cdot,y))\nu=0,&\\
\int_{\partial I_{\e,N}^n}(\cdot-X_{\e,N}^i)\times\sigma(\Gc^n_{\e,N;i}(\cdot,y),P^n_{\e,N;i}(\cdot,y))\nu=0,&
\end{array}\right.\]
and we then set $\Gc_{\e,N;i}^n=(\Gc_{\e,N;ij}^n)_{1\le j\le d}$ and $\Gc_{\e,N}^n=(\Gc_{\e,N;ij}^n)_{1\le i,j\le d}$.
This Green's function satisfies the following pointwise estimates,
\[|\nabla_x\Gc^n_{\e,N}(x,y)|\,\lesssim\,|x-y|^{1-d},\qquad|\nabla_x\nabla_y\Gc_{\e,N}^n(x,y)|\,\lesssim\,|x-y|^{-d}.\qedhere\]
\end{lem}

\medskip
\section{Lipschitz estimate on the fluid velocity}\label{sec:Lip}
This section is devoted to the proof of the following a priori Lipschitz estimate on the fluid velocity in the dilute regime, which holds as long as particles are sufficiently well separated.
In case of spherical passive particles, this was first established by H\"ofer in~\cite[Lemma~3.16]{Hofer-18}, based on an expansion of the fluid velocity by means of the method of reflections.
Note that the argument in that work was for spherical passive particles and indeed relied on the spherical shape of the particles, see e.g.~\cite[Lemma~3.10]{Hofer-18}.
In the present contribution, we give an alternative, more robust, perturbative argument that avoids the method of reflections and further allows to cover the case of non-spherical and active particles.

\begin{prop}[Lipschitz estimate on fluid velocity]\label{prop:Lip-est}
Recall the notation~\eqref{eq:def-min}--\eqref{eq:def-alpha}, and assume that $\dd_{\e,N}^{\min}\ge4\e$ and that $\lambda\alpha_{\e,N}^0\ll1$ is small enough.
Then we have
\begin{equation*}
\|u_{\e,N}\|_{W^{1,\infty}(\R^d)}\,\lesssim_h\,1.\qedhere
\end{equation*}
\end{prop}

\begin{proof}
We split the proof into four steps.

\medskip
\step1 Preliminary estimate on the fluid velocity away from rigid particles: proof that, for all~$x\in\R^d$ with $\dist(x,\{X_{\e,N}^n\}_n)\ge\frac12\dd_{\e,N}^{\min}$, we have
\begin{eqnarray}
|u_{\e,N}(x)|&\lesssim&C_h
+|e|\lambda\alpha_{\e,N}^2+\kappa_0\lambda\alpha_{\e,N}^1\nonumber\\
&&\hspace{0.7cm}+\tfrac\lambda N\sum_{n=1}^N|x-X_{\e,N}^n|^{1-d}\Big(\fint_{I_{\e,N}^{n}+\e B}|\!\D(u_{\e,N})|^2\Big)^\frac12,\nonumber\\
|\nabla u_{\e,N}(x)|&\lesssim&C_h
+|e|\lambda\alpha_{\e,N}^1+\kappa_0\lambda\alpha_{\e,N}^0\nonumber\\
&&\hspace{0.7cm}+\tfrac\lambda N\sum_{n=1}^N|x-X_{\e,N}^n|^{-d}\Big(\fint_{I_{\e,N}^{n}+\e B}|\!\D(u_{\e,N})|^2\Big)^\frac12.\label{eq:pre-concl-lipest}
\end{eqnarray}
From the Stokes problem~\eqref{eq:main-Stokes}, using Lemma~\ref{lem:eqn-pseudo}, we find that the fluid velocity~$u_{\e,N}$ satisfies the following equation in~$\R^d$,
\begin{multline}\label{eq:eqn-pseudo-ex}
-\triangle u_{\e,N}+\nabla (p_{\e,N}\mathds1_{\R^d\setminus\Ic_{\e,N}})\\
\,=\,h\mathds1_{\R^d\setminus\Ic_{\e,N}}+\kappa_0\tfrac{\lambda}{\e N}\sum_{n=1}^Nf_{\e,N}^n\mathds1_{\R^d\setminus\Ic_{\e,N}}
-\sum_{n=1}^N\delta_{\partial I_{\e,N}^n}\sigma(u_{\e,N},p_{\e,N})\nu.
\end{multline}
In terms of the Stokeslet $\Gc$, cf.~Lemma~\ref{lem:Green}, noting that the assumption $\dd_{\e,N}^{\min}\ge4\e$ ensures that propulsion forces $\{f_{\e,N}^n\}_n$ have pairwise disjoint supports, we deduce
\begin{multline*}
u_{\e,N}(x)\,=\,\int_{\R^d\setminus\Ic_{\e,N}}\Gc(x-\cdot)h
+\kappa_0\tfrac{\lambda}{\e N}\sum_{n=1}^N\int_{I_{\e,N}^{n;+}\setminus I_{\e,N}^n}\Gc(x-\cdot)f_{\e,N}^n\\
-\sum_{n=1}^N\int_{\partial I_{\e,N}^n}\Gc(x-\cdot)\,\sigma(u_{\e,N},p_{\e,N})\nu,
\end{multline*}
hence, using boundary conditions for $u_{\e,N}$, cf.~\eqref{eq:main-Stokes}, and recalling~\eqref{eq:extend-f-I},
\begin{multline}\label{eq:decomp-ueps-Green}
u_{\e,N}(x)\,=\,\int_{\R^d\setminus\Ic_{\e,N}}\Gc(x-\cdot)h
+\sum_{n=1}^N\int_{I_{\e,N}^n}\Gc(x-\cdot)e
+\kappa_0\tfrac{\lambda}{\e N}\sum_{n=1}^N\int_{I_{\e,N}^{n;+}}\Gc(x-\cdot)f_{\e,N}^n\\
-\sum_{n=1}^N\int_{\partial I_{\e,N}^n}\Big(\Gc(x-\cdot)-\fint_{I_{\e,N}^n}\Gc(x-\cdot)\Big)\sigma(u_{\e,N},p_{\e,N})\nu.
\end{multline}
Appealing to a trace estimate and to pointwise bounds on $\Gc$, cf.~Lemmas~\ref{lem:trace} and~\ref{lem:Green}, recalling the local balance condition~\eqref{eq:local-balance} for propulsion forces, performing local integrals, and recalling $\e^d|I^\circ|=\frac\lambda N$,
we find for all $x\in\R^d$ with $\dist(x,\Ic_{\e,N})\ge2\e$,
\begin{multline*}
|u_{\e,N}(x)|\,\lesssim\,\int_{\R^d}|x-\cdot|^{2-d}|h|
+|e|\tfrac\lambda N\sum_{n=1}^N| x-X_{\e,N}^n|^{2-d}+\kappa_0\tfrac{\lambda}{N}\sum_{n=1}^N| x-X_{\e,N}^n|^{1-d}\\
+\tfrac{\lambda}N\sum_{n=1}^N| x-X_{\e,N}^n|^{1-d}\Big(\fint_{I_{\e,N}^{n;+}}|\!\D(u_{\e,N})|^2\Big)^\frac12.
\end{multline*}
Similarly, first differentiating in space,
\begin{multline}\label{eq:preconcl-claim11}
|\nabla u_{\e,N}(x)|\,\lesssim\,\int_{\R^d}|x-\cdot|^{1-d}|h|
+|e|\tfrac{\lambda}N\sum_{n=1}^N|x-X_{\e,N}^n|^{1-d}
+\kappa_0\tfrac\lambda N\sum_{n=1}^N| x-X_{\e,N}^n|^{-d}\\
+\tfrac\lambda N\sum_{n=1}^N|x-X_{\e,N}^n|^{-d}\Big(\fint_{I_{\e,N}^{n;+}}|\!\D(u_{\e,N})|^2\Big)^\frac12.
\end{multline}
Now note that, for all $x\in\R^d$ with $\dist(x,\{X^n_{\e,N}\}_{n})\ge\tfrac12\dd_{\e,N}^{\min}$,
if $X_{\e,N}^p$ is a particle that is the closest to $x$, we can estimate for any $\sigma\in[0,d]$,
\begin{equation}\label{eq:bnd-x-alpha}
\tfrac{1}N\sum_{n=1}^N| x-X_{\e,N}^n|^{\sigma-d}\,\lesssim\,\tfrac{1}N\sum_{n:n\ne p}^N|X_{\e,N}^p-X_{\e,N}^n|^{\sigma-d}\,\le\,\alpha_{\e,N}^\sigma,
\end{equation}
where we recall that $\alpha_{\e,N}^\sigma$ is defined in~\eqref{eq:def-alpha}.
Using this bound in~\eqref{eq:preconcl-claim11}, the claim~\eqref{eq:pre-concl-lipest} follows.

\medskip
\step2 Preliminary estimate on the fluid velocity close to rigid particles:
for all $x\in\R^d\setminus\Ic_{\e,N}$ with $|x-X_{\e,N}^n|\le\frac12\dd_{\e,N}^{\min}$ for some $1\le n\le N$, we have
\begin{eqnarray}
|u_{\e,N}(x)|&\lesssim&C_h+|e|(\lambda\alpha_{\e,N}^2+\e^2)+\kappa_0(\lambda\alpha_{\e,N}^1+\e)\nonumber\\
&&\hspace{0.7cm}+\tfrac\lambda N\sum_{m:m\ne n}^N|X_{\e,N}^n-X_{\e,N}^m|^{1-d}\Big(\fint_{I_{\e,N}^{m;+}}|\!\D(u_{\e,N})|^2\Big)^\frac12,\nonumber\\
|\nabla u_{\e,N}(x)|&\lesssim&C_h
+|e|(\lambda\alpha_{\e,N}^1+\e)+\kappa_0(\lambda\alpha_{\e,N}^0+1)\nonumber\\
&&\hspace{0.7cm}+\tfrac\lambda N\sum_{m:m\ne n}^N|X_{\e,N}^n-X_{\e,N}^m|^{-d}\Big(\fint_{I_{\e,N}^{m;+}}|\!\D(u_{\e,N})|^2\Big)^\frac12.\label{eq:pre-concl-lipest-2}
\end{eqnarray}
Let $x\in\R^d\setminus\Ic_{\e,N}$ be fixed with $|x-X_{\e,N}^n|\le\frac12\dd_{\e,N}^{\min}$ for some $1\le n\le N$.
Note that we then have $\dist(x,\{X_{\e,N}^m\}_{m:m\ne n})\ge\frac12\dd_{\e,N}^{\min}$.
Testing equation~\eqref{eq:eqn-pseudo-ex} in $\R^d\setminus I_{\e,N}^n$ with the Stokeslet~$\Gc^n_{\e,N}(\cdot,x)$ for the Stokes problem with a single rigid inclusion, cf.~Lemma~\ref{lem:Green-rig}, we find, instead of~\eqref{eq:decomp-ueps-Green},
\begin{multline*}
u_{\e,N}(x)
\,=\,\int_{\R^d\setminus\Ic_{\e,N}}\Gc_{\e,N}^n(\cdot,x) h
+\sum_{m=1}^N \int_{I_{\e,N}^m}\Gc_{\e,N}^n(\cdot,x)e
+\kappa_0\tfrac{\lambda}{\e N}\sum_{m=1}^N\int_{I_{\e,N}^{m;+}}\Gc_{\e,N}^n(\cdot,x) f_{\e,N}^m\\
-\sum_{m:m\ne n}^N\int_{\partial I_{\e,N}^m}\Big(\Gc_{\e,N}^n(\cdot,x)-\fint_{I_{\e,N}^m}\Gc_{\e,N}^n(\cdot,x)\Big)\sigma(u_{\e,N},p_{\e,N})\nu.
\end{multline*}
Now appealing to a trace estimate and to the pointwise bounds on $\Gc^n_{\e,N}$, cf.~Lemmas~\ref{lem:trace} and~\ref{lem:Green-rig}, and noting that $|x-X_{\e,N}^m|\gtrsim|X_{\e,N}^n-X_{\e,N}^m|$ for all $m$, the claim~\eqref{eq:pre-concl-lipest-2} follows similarly as in Step~2.
The only difference is that we now need to treat diagonal contributions separately: for instance,
\begin{eqnarray*}
\bigg|\sum_{m=1}^N\int_{I_{\e,N}^m}\Gc_{\e,N}^n(\cdot,x)e\bigg|
&\lesssim&|e|\tfrac\lambda N\fint_{I_{\e,N}^n}|\cdot-x|^{2-d}+|e|\tfrac\lambda N\sum_{m:m\ne n}^N\fint_{I_{\e,N}^m}|\cdot-x|^{2-d}\\
&\lesssim&|e|\tfrac\lambda N\langle x-X_{\e,N}^n\rangle_\e^{2-d}+|e|\tfrac\lambda N\sum_{m:m\ne n}^N| x-X_{\e,N}^m|^{2-d}\\
&\lesssim&|e|(\e^2+\lambda\alpha_{\e,N}^2),
\end{eqnarray*}
where we use the short-hand notation $\langle x\rangle_\e:=(\e^2+|x|^2)^{1/2}$,
and where the last inequality follows from~\eqref{eq:bnd-x-alpha} with $\dist(x,\{X_{\e,N}^m\}_{m:m\ne n})\ge\frac12\dd_{\e,N}^{\min}$.

\medskip
\step3 Closed estimate on velocity gradients: proof that for all $1\le n\le N$,
\begin{equation}\label{eq:est-fintDu}
\Big(\fint_{I_{\e,N}^{n;+}}|\!\D(u_{\e,N})|^2\Big)^\frac12
\,\lesssim\,C_h+|e|(\lambda\alpha_{\e,N}^1+\e)+\kappa_0(\lambda\alpha_{\e,N}^0+1).
\end{equation}
This estimate is obtained by post-processing the results of the first two steps in the dilute regime $\lambda\ll1$.
From~\eqref{eq:pre-concl-lipest-2}, we deduce in particular, for all~$1\le n\le N$,
\begin{multline}\label{eq:reform-step2-intDu}
\Big(\fint_{I_{\e,N}^{n;+}}|\!\D(u_{\e,N})|^2\Big)^\frac12\,\lesssim\,
C_h+|e|(\lambda\alpha_{\e,N}^1+\e)+\kappa_0(\lambda\alpha_{\e,N}^0+1)\\
+\tfrac\lambda N\sum_{m:m\ne n}^N|X_{\e,N}^n-X_{\e,N}^m|^{-d}\Big(\fint_{I_{\e,N}^{m;+}}|\!\D(u_{\e,N})|^2\Big)^\frac12.
\end{multline}
After summation, recalling~\eqref{eq:def-alpha}, this leads us to
\begin{multline*}
\sup_{1\le n\le N}\bigg[\tfrac1 N\sum_{m:m\ne n}^N|X_{\e,N}^n-X_{\e,N}^m|^{-d}\Big(\fint_{I_{\e,N}^{m;+}}|\!\D(u_{\e,N})|^2\Big)^\frac12\bigg]\\
\,\lesssim\,
\alpha_{\e,N}^0
\big(C_h+|e|(\lambda\alpha_{\e,N}^1+\e)+\kappa_0(\lambda\alpha_{\e,N}^0+1)\big)\\
+\lambda\alpha_{\e,N}^0\,\sup_{1\le n\le N}\bigg[\tfrac1N\sum_{m:m\ne n}^N|X_{\e,N}^n-X_{\e,N}^m|^{-d}\Big(\fint_{I_{\e,N}^{m;+}}|\!\D(u_{\e,N})|^2\Big)^\frac12\bigg].
\end{multline*}
Provided that $\lambda\alpha_{\e,N}^0\ll1$ is small enough, the last right-hand side term can be absorbed, to the effect of
\begin{multline*}
\sup_{1\le n\le N}\bigg[\tfrac1N\sum_{m:m\ne n}^N|X_{\e,N}^n-X_{\e,N}^m|^{-d}\Big(\fint_{I_{\e,N}^{m;+}}|\!\D(u_{\e,N})|^2\Big)^\frac12\bigg]\\
\,\lesssim\,
\alpha_{\e,N}^0
\big(C_h+|e|(\lambda\alpha_{\e,N}^1+\e)+\kappa_0(\lambda\alpha_{\e,N}^0+1)\big).
\end{multline*}
Inserting this back into~\eqref{eq:reform-step2-intDu}, the claim~\eqref{eq:est-fintDu} follows.

\medskip
\step4 Conclusion.\\
Inserting~\eqref{eq:est-fintDu} into~\eqref{eq:pre-concl-lipest} and~\eqref{eq:pre-concl-lipest-2}, using again~\eqref{eq:bnd-x-alpha}, using Hölder's inequality in form of
\begin{equation}\label{eq:Holder-alpha}
(\alpha_{\e,N}^1)^2\le\alpha_{\e,N}^0\alpha_{\e,N}^2,\qquad\alpha^1_{\e,N}\le(\alpha^0_{\e,N})^\frac{d-1}d,\qquad\alpha^2_{\e,N}\le(\alpha^1_{\e,N})^\frac{d-2}{d-1}\le(\alpha^0_{\e,N})^\frac{d-2}d,
\end{equation}
and using $\lambda\alpha_{\e,N}^0\le1$,
we deduce for all $x\in\R^d\setminus\Ic_{\e,N}$,
\begin{equation}\label{eq:Lipestim-pr}
|u_{\e,N}(x)|+|\nabla u_{\e,N}(x)|\,\lesssim\,C_h+|e|\big(\lambda(\alpha_{\e,N}^1+1)+\e\big)+\kappa_0,
\end{equation}
which yields the conclusion as $|e|,\kappa_0\le1$.
\end{proof}

\section{Dilute expansion of particle velocities}\label{sec:cluster}
This section is devoted to the following dilute expansion of particle velocities.
In case of spherical passive particles, a corresponding expansion of translational velocities was already obtained in~\cite{Hofer-Schubert-21}.
While previous contributions on the topic were based on the reflection method~\cite{Jabin-Otto-04,Hofer-18,Mecherbet-19,Hofer-Schubert-21},
we take a different path inspired by our recent work with Gloria~\cite{DG-21b}: the dilute expansion is obtained here instead by means of a cluster expansion combined with a monopole approximation.
The proof is split into several parts and is concluded by combining Lemmas~\ref{lem:cluster-trunc} and~\ref{lem:monopole} below.

\begin{prop}[Dilute expansion of particle velocities]\label{prop:dilute-exp}
Assume that \mbox{$\dd_{\e,N}^{\min}\ge4\e$} and that~$\lambda\alpha_{\e,N}^0\ll1$ is small enough.
Particle velocities~\eqref{eq:VTheta} can then be expanded as follows, for all $1\le n\le N$,
\begin{enumerate}[---]
\item First-order expansion of translational velocities:
\begin{equation*}
\quad|V_{\e,N}^n-(\Gc\ast h)(X_{\e,N}^n)|
\,\lesssim_h\,\lambda(\alpha_{\e,N}^1+1)+\e;
\end{equation*}
\item Second-order expansion of translational velocities:
\begin{multline*}
\quad\Big|V_{\e,N}^n-\big[\Gc\,\hat\ast\, \big((1-\lambda\mu_{\e,N})h+\lambda\mu_{\e,N}e\big)\big](X_{\e,N}^n)\\
-\lambda \big[\nabla\Gc\,\hat\ast\,\big(2\langle\Sigma^\circ\nu_{\e,N}\rangle\!\D(\Gc\ast h)+\kappa_0\langle\Sigma^\circ_f\nu_{\e,N}\rangle\big)\big](X_{\e,N}^n)
-\kappa_0\e V^\circ_f(R_{\e,N}^n)\Big|\\
\,\lesssim_h\,
(\lambda\alpha_{\e,N}^1+\e)\big(\lambda(\alpha_{\e,N}^0+1)+\e\big);
\end{multline*}
\item First-order expansion of angular velocities:
\begin{equation*}
\quad\big|\Omega_{\e,N}^n-\Omega^\circ(R_{\e,N}^n)\nabla(\Gc\ast h)(X_{\e,N}^n)\big|
\,\lesssim_h\,\lambda(\alpha_{\e,N}^0+1)+\e;
\end{equation*}
\end{enumerate}
where we recall the short-hand notation $g\,\hat\ast\,\mu(x)=\int_{\R^d\setminus\{x\}}g(x-y)d\mu(y)$ for diagonal-free convolutions, and where $\Sigma^\circ,\Omega^\circ,\Sigma_f^\circ,V_f^\circ$ are defined in~\eqref{eq:def-S(r)}--\eqref{eq:def-act-drag}.
Moreover, away from the particles, we have the following expansion for the fluid velocity, for any boundary-layer thickness~\mbox{$\delta\in[\e,1]$},
\begin{align*}
&\Big\|u_{\e,N}-\Gc\ast \big((1-\lambda\mu_{\e,N})h+\lambda\mu_{\e,N}e\big)\\
&\hspace{2cm}-\lambda\nabla\Gc\ast\big(2\langle\Sigma^\circ\nu_{\e,N}\rangle\!\D(\Gc\ast h)
+\kappa_0\langle\Sigma^\circ_f\nu_{\e,N}\rangle\big)\Big\|_{\Ld^\infty(\R^d\setminus(\Ic_{\e,N}+\delta B))}\nonumber\\
&\hspace{7cm}\,\lesssim_h\,\lambda(\lambda\alpha_{\e,N}^1+\e)(\alpha_{\e,N}^0+1)+\e(\tfrac\e\delta)^d.\nonumber\qedhere
\end{align*}
\end{prop}

\begin{rem}[Higher-order corrections and failure of mean-field theory]\label{rem:noMFL}
These dilute expansions can be easily pursued to higher order. For instance, the next-order expansion of the translational velocity $V_{\e,N}^n$ would involve the three-body contribution
\begin{equation}\label{eq:3body-contr}
\lambda^2\Big[\nabla\Gc\,\hat\ast\,\Big(2\langle\Sigma^\circ\nu_{\e,N}\rangle\!\D\!\big(\nabla\Gc\,\hat\ast\,(2\langle\Sigma^\circ\nu_{\e,N}\rangle\!\D(\Gc\ast h))\big)\Big)\Big](X_{\e,N}^n),
\end{equation}
which describes the flow disturbance due to a stress difference at a particle boundary generated by the flow disturbance due to a stress difference at another particle boundary. The higher-order expansion of particle velocities can however not be used to improve on the accuracy in the mean-field approximation.
Indeed, the above three-body contribution~\eqref{eq:3body-contr} involves particle interactions via the kernel $\nabla^2\Gc$, which has critical decay $|\nabla^2\Gc(x)|\simeq|x|^{-d}$,
and the criticality of this kernel is known to imply the failure of mean-field approximations: by scaling, any macroscopic limit should still depend on the microscopic arrangement of the particles. We refer to~\cite{Oelschlager-90,Jabin-14} for this classical limitation to mean-field theory. For particle suspensions, this was recently illustrated by Höfer, Mecherbet, and Schubert in~\cite{Hofer-Mecherbet-Schubert-22}.
More precisely, the $O(\lambda^2)$ correction to the dilute mean-field approximation should involve the statistical distribution of pairs of particles on the microscale --- as was indeed anticipated in the introduction, in link with corrections to Einstein's formula~\eqref{eq:Einstein}, cf.~\cite{DG-21b,GV-21}. This leads us beyond the scope of propagation of chaos and mean-field theory. (We emphasize that it is not about the two-particle {\it macroscopic} density, which commonly appears when describing corrections to mean field, e.g.~\cite{D-21a}: instead, it is here about the distribution of pairs of particles {\it on the microscale}.)
\end{rem}

\subsection{Cluster expansion}
In view of~\eqref{eq:VTheta}, Proposition~\ref{prop:dilute-exp} amounts to establishing a dilute expansion for the fluid velocity $u_{\e,N}$. The latter is defined as the solution of the Stokes problem~\eqref{eq:main-Stokes}, which, in view of its weak formulation~\eqref{eq:main-Stokes-weak}, is equivalent to setting
\[u_{\e,N}\,=\,\pi_{\e,N}v_{\e,N},\]
where $\pi_{\e,N}$ is the orthogonal projection
\[\pi_{\e,N}~:~\big\{u\in\dot H^1(\R^d)^d\,:\,\Div(u)=0\big\}~\twoheadrightarrow~\big\{u\in \dot H^1(\R^d)^d\,:\,\Div(u)=0,~\D(u)|_{\Ic_{\e,N}}=0\big\},\]
and where $v_{\e,N}$ is the solution of
\begin{equation}\label{eq:veps}
\left\{\begin{array}{ll}
-\triangle v_{\e,N}+\nabla q_{\e,N}=h\mathds1_{\R^d\setminus\Ic_{\e,N}}+e\mathds1_{\Ic_{\e,N}}+\kappa_0\tfrac{\lambda}{\e N}\sum_{n=1}^Nf_{\e,N}^n,&\text{in $\R^d$},\\
\Div(v_{\e,N})=0,&\text{in $\R^d$}.
\end{array}\right.
\end{equation}
This decomposition is particularly useful as $v_{\e,N}$ depends linearly on the set of particles: this simplifies the hydodynamic problem to pairwise interactions between the particles, while all multibody effects are contained in the projection~$\pi_{\e,N}$.
The dilute expansion of~$u_{\e,N}$ then amounts to expanding $\pi_{\e,N}$.
For that purpose, rather than appealing to the method of reflections, as was done in previous work on the topic, e.g.~\cite{Jabin-Otto-04,Hofer-18,Mecherbet-19,Hofer-21}, we start from its cluster expansion as inspired by our work with Gloria~\cite{DG-16a,DG-21b,DG-22-review}.

In a nutshell, the cluster expansion of~$\pi_{\e,N}$ amounts to decomposing multibody effects into a series of contributions involving subsets of particles of increasing cardinality.
We start with some notation. For any index subset $J\subset\{1,\ldots,N\}$, we define the orthogonal projection
\[\pi_{\e,N}^J~:~\big\{u\in\dot H^1(\R^d)^d:\Div(u)=0\big\}~\twoheadrightarrow~\big\{u\in \dot H^1(\R^d)^d:\Div(u)=0,~\D(u)|_{\cup_{n\in J}I^n_{\e,N}}=0\big\},\]
that is, the partial projection only taking into account particles with indices in $J$.
Note that by definition we have in particular $\pi_{\e,N}^\varnothing=\Id$ and $\pi_{\e,N}^{\{1,\ldots,N\}}=\pi_{\e,N}$.
Next, we consider differences of partial projections,
\[\delta^{\{n\}}\pi_{\e,N}^J\,:=\,\pi_{\e,N}^{J\cup\{n\}}-\pi_{\e,N}^J,\qquad 1\le n\le N,\]
as well as higher-order differences,
\[\delta^K\pi_{\e,N}^J\,:=\,\sum_{I\subset K}(-1)^{\sharp(K\setminus I)}\pi_{\e,N}^{J\cup I},\qquad K\subset\{1,\ldots,N\}.\]
Note that by definition we have
\[\delta^\varnothing\pi_{\e,N}^J=\pi_{\e,N}^J,\qquad\text{and}\qquad \delta^K\pi_{\e,N}^J=0\quad\text{whenever $K\cap J\ne\varnothing$.}\]
In these terms, the cluster expansion of $u_{\e,N}=\pi_{\e,N}v_{\e,N}$ takes on the following guise, where the~$k$th term $u_{\e,N}^{(k)}$ describes the contribution of $k$-body interactions.

\begin{lem}[Cluster expansion]
The following identity holds,
\begin{equation}\label{eq:cluster-u}
u_{\e,N}\,=\,\sum_{k=0}^Nu_{\e,N}^{(k)},
\end{equation}
where we have defined $u_{\e,N}^{(0)}:=v_{\e,N}$ and for all $k\ge1$,
\begin{equation}\label{eq:cluster-u(n)}
u_{\e,N}^{(k)}\,:=\,\sum_{1\le n_1<\ldots<n_k\le N}(\delta^{\{n_1,\ldots,n_k\}}\pi_{\e,N}^\varnothing)v_{\e,N}.\qedhere
\end{equation}
\end{lem}
\begin{proof}
Identity~\eqref{eq:cluster-u} follows from the simple observation that
\begin{eqnarray*}
\sum_{k=0}^N\,\sum_{1\le n_1<\ldots<n_k\le N}\delta^{\{n_1,\ldots,n_k\}}\pi_{\e,N}^\varnothing
&=&\sum_{K\subset\{1,\ldots,N\}}\delta^K\pi_{\e,N}^\varnothing\\
&=&\sum_{K\subset\{1,\ldots,N\}}\sum_{I\subset K}(-1)^{\sharp (K\setminus I)}\pi_{\e,N}^I\\
&=&\sum_{I\subset\{1,\ldots,N\}}\pi_{\e,N}^I\sum_{k=0}^{N-\sharp I}(-1)^{k}\binom{N-\sharp I}{k}\\
&=&\pi_{\e,N}^{\{1,\ldots,N\}}~=~\pi_{\e,N},
\end{eqnarray*}
where we use the fact that $\sum_{k=0}^K(-1)^k\binom{K}k=\mathds1_{K=0}$.
\end{proof}

\subsection{Cluster expansion errors}
We turn to the accuracy of the cluster expansion~\eqref{eq:cluster-u} upon truncation in the dilute regime $\lambda\ll1$ (note that the condition $\lambda\alpha_{\e,N}^0\ll1$ amounts to considering well-separated particles). For the purpose of this work, we restrict ourselves to the second-order expansion, but higher orders can be dealt with analogously, cf.~Remark~\ref{rem:higher-cluster-trunc}.

\begin{lem}[Cluster expansion errors]\label{lem:cluster-trunc}
Assume that $\dd_{\e,N}^{\min}\ge4\e$ and that $\lambda\alpha_{\e,N}^0\ll1$ is small enough.
Then we have
\begin{eqnarray}
\|u_{\e,N}-u_{\e,N}^{(0)}\|_{\Ld^{\infty}(\R^d)}&\lesssim_h&\lambda\alpha_{\e,N}^1+\e,\label{eq:approx-0th-u}\\
\|u_{\e,N}-u_{\e,N}^{(0)}-u_{\e,N}^{(1)}\|_{\Ld^{\infty}(\R^d)}&\lesssim_h&\lambda\alpha_{\e,N}^0(\lambda\alpha_{\e,N}^1+\e),\label{eq:approx-1st-u}
\end{eqnarray}
and moreover, for velocity gradients, for all $1\le n\le N$,
\begin{equation}\label{eq:approx-0th-nabu}
\big\|\nabla u_{\e,N}-\nabla\pi^{\{n\}}_{\e,N}u^{(0)}_{\e,N}\big\|_{\Ld^{\infty}(B(X_{\e,N}^n,\frac12\dd_{\e,N}^{\min}))}\,\lesssim_h\,\lambda\alpha_{\e,N}^0.\qedhere
\end{equation}
\end{lem}

\begin{rem}\label{rem:higher-cluster-trunc}
The proof below can be immediately generalized to higher orders: we can show for all $1\le K< N$,
\[\Big\|u_{\e,N}-\sum_{k=0}^Ku_{\e,N}^{(k)}\Big\|_{\Ld^\infty(\R^d)}\,\lesssim_{h,K}\,(\lambda\alpha_{\e,N}^0)^K(\lambda\alpha_{\e,N}^1+\e),\]
and moreover, for velocity gradients, for all $1\le n\le N$,
\[\Big\|\nabla u_{\e,N}-\nabla\pi_{\e,N}^{\{n\}}\sum_{k=0}^K u_{\e,N}^{(k)}\Big\|_{\Ld^{\infty}(B(X_{\e,N}^n,\frac12\dd_{\e,N}^{\min}))}\,\lesssim_{h,K}\,(\lambda\alpha_{\e,N}^0)^{K+1}.\]
As this will not be used in this work, we omit the detail for shortness.
\end{rem}

\begin{proof}[Proof of Lemma~\ref{lem:cluster-trunc}]
We split the proof into three steps, separately proving the different estimates in the statement.

\medskip
\step1 First-order expansion of $u_{\e,N}$: proof of~\eqref{eq:approx-0th-u}.\\
Subtracting~\eqref{eq:eqn-pseudo-ex} from the defining equation for $u_{\e,N}^{(0)}=v_{\e,N}$, cf.~\eqref{eq:veps}, we get the following equation in $\R^d$,
\begin{equation}\label{eq:rel-u-u0}
-\triangle( u_{\e,N}- u_{\e,N}^{(0)})+\nabla p
\,=\,-\sum_{n=1}^N\Big( e\mathds1_{I^n_{\e,N}}+\kappa_0\tfrac{\lambda}{\e N}f_{\e,N}^n\mathds1_{I^n_{\e,N}}
+\delta_{\partial I_{\e,N}^n}\sigma(u_{\e,N},p_{\e,N})\nu\Big).
\end{equation}
In terms of the Stokeslet $\Gc$, using the boundary conditions for $u_{\e,N}$, we deduce
\begin{equation*}
( u_{\e,N}- u_{\e,N}^{(0)})(x)
\,=\,-\sum_{n=1}^N\int_{\partial I_{\e,N}^n}\Big(\Gc(\cdot-x)-\fint_{I_{\e,N}^n}\Gc(\cdot-x)\Big)\sigma(u_{\e,N},p_{\e,N})\nu.
\end{equation*}
Appealing to a trace estimate and to pointwise bounds on $\Gc$, cf.~Lemmas~\ref{lem:trace} and~\ref{lem:Green}, and evalutating local integrals, this leads us to
\begin{equation*}
|(u_{\e,N}-u_{\e,N}^{(0)})(x)|\,\lesssim\,\tfrac\lambda N\sum_{n=1}^N \langle x-X_{\e,N}^n\rangle_\e^{1-d}\Big(\fint_{I_{\e,N}^{n;+}}|\!\D(u_{\e,N})|^2\Big)^\frac12,
\end{equation*}
where we recall the short-hand notation $\langle x\rangle_\e=(\e^2+|x|^2)^{1/2}$.
Now note that~\eqref{eq:bnd-x-alpha} yields for all $x\in\R^d$, after separating the diagonal contribution,
\begin{equation}\label{eq:bnd-x-alpha-re}
\tfrac\lambda N\sum_{n=1}^N \langle x-X_{\e,N}^n\rangle_\e^{1-d}\,\lesssim\,\lambda\alpha_{\e,N}^1+\e.
\end{equation}
Inserting this into the above, and combining with the Lipschitz estimate of Proposition~\ref{prop:Lip-est}, we deduce for all $x\in\R^d$,
\begin{equation*}
|(u_{\e,N}-u_{\e,N}^{(0)})(x)|\,\lesssim_h\,\lambda\alpha_{\e,N}^1+\e,
\end{equation*}
and the conclusion~\eqref{eq:approx-0th-u} follows.

\medskip
\step2 Second-order expansion of $u_{\e,N}$: proof of~\eqref{eq:approx-1st-u}.\\
We use the short-hand notation $u_{\e,N}^n:=\pi_{\e,N}^{\{n\}}v_{\e,N}$ and we denote by $p_{\e,N}^n$ the associated pressure field in $\R^d\setminus I_{\e,N}^n$.
Comparing~\eqref{eq:rel-u-u0} with the corresponding equation for
\[u_{\e,N}^{(1)}\,=\,\sum_{n=1}^N(u_{\e,N}^n-v_{\e,N}),\]
we obtain the following in $\R^d$,
\begin{equation*}
-\triangle\big(u_{\e,N}-u_{\e,N}^{(0)}-u_{\e,N}^{(1)}\big)+\nabla p\,=\,-\sum_{n=1}^N\mathds1_{\partial I_{\e,N}^n}\sigma\big(u_{\e,N}-u_{\e,N}^n,p_{\e,N}-p_{\e,N}^n\big)\nu.
\end{equation*}
In terms of the Stokeslet $\Gc$, using boundary conditions, and appealing to a trace estimate and to pointwise bounds, we deduce as in Step~1,
\begin{equation}\label{eq:u-u0u1-appr0}
\big|\big(u_{\e,N}-u_{\e,N}^{(0)}-u_{\e,N}^{(1)}\big)(x)\big|\,\lesssim\,\tfrac\lambda N\sum_{n=1}^N\langle x-X_{\e,N}^n\rangle_\e^{1-d}\Big(\fint_{I_{\e,N}^{n;+}}|\!\D(u_{\e,N}-u_{\e,N}^n)|^2\Big)^\frac12.
\end{equation}
It remains to estimate the last factor.
For that purpose, given $1\le n\le N$, we start by noting that the difference $u_{\e,N}-u_{\e,N}^n$ satisfies the following equation in $\R^d\setminus I_{\e,N}^n$,
\begin{equation*}
-\triangle\big(u_{\e,N}-u_{\e,N}^n\big)+\nabla p\,=\,-\sum_{m:m\ne n}^N\Big(e\mathds1_{I_{\e,N}^m}+\kappa_0\tfrac{\lambda}{\e N}f_{\e,N}^m\mathds1_{I_{\e,N}^m}+\delta_{\partial I_{\e,N}^m}\sigma(u_{\e,N},p_{\e,N})\nu\Big).
\end{equation*}
Testing this with the Stokeslet~$\Gc^n_{\e,N}(\cdot,x)$ corresponding to the problem with a single rigid inclusion at $I_{\e,N}^n$, cf.~Lemma~\ref{lem:Green-rig}, and using boundary conditions, we get in $\R^d\setminus I_{\e,N}^n$,
\begin{equation}\label{eq:Greenrep-u-un}
(u_{\e,N}-u_{\e,N}^n)(x)\,=\,-\sum_{m:m\ne n}^N\int_{\partial I_{\e,N}^m}\Big(\Gc_{\e,N}^n(\cdot,x)-\fint_{I_{\e,N}^m}\Gc_{\e,N}^n(\cdot,x)\Big)\,\sigma(u_{\e,N},p_{\e,N})\nu,
\end{equation}
hence, appealing to a trace estimate and to the pointwise bounds on $\Gc^n_{\e,N}$, cf.~Lemmas~\ref{lem:trace} and~\ref{lem:Green-rig},
\[\Big(\fint_{I_{\e,N}^n}|\!\D(u_{\e,N}-u_{\e,N}^n)|^2\Big)^\frac12\,\lesssim\,\tfrac\lambda N\sum_{m:m\ne n}^N|X_{\e,N}^n-X_{\e,N}^m|^{-d}\Big(\fint_{I_{\e,N}^{m;+}}|\!\D(u_{\e,N})|^2\Big)^\frac12.\]
Inserting this into~\eqref{eq:u-u0u1-appr0}, combining with the Lipschitz estimate of Proposition~\ref{prop:Lip-est}, and using~\eqref{eq:bnd-x-alpha-re} again, we get for all $x\in\R^d$,
\begin{equation*}
\big|\big(u_{\e,N}-u_{\e,N}^{(0)}-u_{\e,N}^{(1)}\big)(x)\big|\,\lesssim_h\,\lambda\alpha_{\e,N}^0(\lambda\alpha_{\e,N}^1+\e),
\end{equation*}
and the conclusion~\eqref{eq:approx-1st-u} follows.

\medskip
\step3 First-order expansion  of $\nabla u_{\e,N}$: proof of~\eqref{eq:approx-0th-nabu}.\\
Let $x\in\R^d\setminus\Ic_{\e,N}$ be fixed with $|x-X_{\e,N}^n|\le\frac12\dd_{\e,N}^{\min}$ for some $1\le n\le N$.
Starting from~\eqref{eq:Greenrep-u-un} and appealing again to a trace estimate and to the pointwise bounds on $\Gc^n_{\e,N}$, cf.~Lemmas~\ref{lem:trace} and~\ref{lem:Green-rig}, we get
\begin{equation*}
|\nabla(u_{\e,N}-u^n_{\e,N})(x)|
\,\lesssim\,
\tfrac\lambda N\sum_{m:m\ne n}^N|X_{\e,N}^n-X_{\e,N}^m|^{-d}\Big(\fint_{I_{\e,N}^{m;+}}|\!\D(u_{\e,N})|^2\Big)^\frac12,
\end{equation*}
and the conclusion~\eqref{eq:approx-0th-nabu} then follows as above.
\end{proof}

\subsection{Analysis of cluster terms}\label{sec:proof-prop-cluster}
In order to conclude the proof of Proposition~\ref{prop:dilute-exp} and obtain the desired asymptotics for particle velocities~\eqref{eq:VTheta},
we build on the cluster estimates of Lemma~\ref{lem:cluster-trunc}
by further performing a multipole expansion of the cluster terms in the limit of small well-separated particles $\e\ll1$.
For the purpose of this work, we restrict ourselves to the first two cluster terms and to their monopole approximation, but the description of higher-order cluster terms and their full multipole expansion could be pursued analogously.
The conclusion of Proposition~\ref{prop:dilute-exp} directly follows by combining the following result with Lemma~\ref{lem:cluster-trunc}, further using~\eqref{eq:Holder-alpha} to slightly simplify the bounds.

\begin{lem}[Monopole approximation of cluster terms]\label{lem:monopole}
Assume that $\dd_{\e,N}^{\min}\ge4\e$ and that $\lambda\alpha_{\e,N}^0\ll1$ is small enough.
Then we have
\begin{eqnarray}
\Big|\fint_{I_{\e,N}^n}\!\!u_{\e,N}^{(0)}-(\Gc\ast h)(X_{\e,N}^n)\Big|
\!\!&\lesssim_h&\!\!\!\lambda(\alpha^2_{\e,N}+\kappa_0\alpha^1_{\e,N})+\e(\e+\kappa_0),\label{eq:exp-u0}\\
\hspace{-0.9cm}\Big|\fint_{I_{\e,N}^n}\!\!\nabla \pi^{\{n\}}_{\e,N}u^{(0)}_{\e,N}-\Omega^\circ(R_{\e,N}^n)\nabla(\Gc\ast h)(X_{\e,N}^n)\Big|
\!\!&\lesssim_h&\!\!\!\lambda(\alpha_{\e,N}^1+\kappa_0\alpha^0_{\e,N})+\e,\label{eq:exp-nabu0}
\end{eqnarray}
and moreover,
\begingroup\allowdisplaybreaks
\begin{multline}\label{eq:exp-u1}
\bigg|\fint_{I_{\e,N}^n}\!\!\big(u^{(0)}_{\e,N}+u^{(1)}_{\e,N}\big)
-\big[\Gc\,\hat\ast\,\big((1-\lambda\mu_{\e,N})h+\lambda\mu_{\e,N}e\big)\big](X_{\e,N}^n)\\
-\lambda\big[\nabla\Gc\,\hat\ast\,\big(2\langle\Sigma^\circ\nu_{\e,N}\rangle\!\D(\Gc\ast h)+\kappa_0\langle\Sigma_f^\circ\nu_{\e,N}\rangle\big)\big](X_{\e,N}^n)
-\kappa_0\e V^\circ_f(R_{\e,N}^n)\bigg|\\
\,\lesssim_h\,(\lambda\alpha_{\e,N}^1+\e)\big(\lambda(\alpha_{\e,N}^0+1)+\e\big),
\end{multline}
\endgroup
where we recall that $\Sigma^\circ,\Omega^\circ,\Sigma^\circ_f,V^\circ_f$ are defined in~\eqref{eq:def-S(r)}--\eqref{eq:def-act-drag}.
Moreover, away from the particles, the fluid velocity satisfies for any boundary-layer thickness $\delta\in[\e,1]$,
\begin{equation}\label{eq:estim-u0u1}
\|u_{\e,N}^{(0)}-\Gc\ast h\|_{\Ld^\infty(\R^d\setminus(\Ic_{\e,N}+\delta B))}\,\lesssim_h\,\lambda(\alpha_{\e,N}^2+\kappa_0\alpha_{\e,N}^1)+\delta(\delta+\kappa_0)(\tfrac\e\delta)^d.
\end{equation}
and
\begin{align}
&\Big\|\big(u^{(0)}_{\e,N}+u^{(1)}_{\e,N}\big)-\Gc\ast \big((1-\lambda\mu_{\e,N})h+\lambda\mu_{\e,N}e\big)\label{eq:estim-u0u1-bis}\\
&\hspace{2cm}-\lambda\nabla\Gc\ast\big(2\langle\Sigma^\circ\nu_{\e,N}\rangle\!\D(\Gc\ast h)+\kappa_0\langle\Sigma_f^\circ\nu_{\e,N}\rangle\big)\Big\|_{\Ld^\infty(\R^d\setminus(\Ic_{\e,N}+\delta B))}\nonumber\\
&\hspace{7cm}\,\lesssim_h\,\lambda(\lambda\alpha_{\e,N}^1+\e)(\alpha_{\e,N}^0+1)+\e(\tfrac\e\delta)^d.\nonumber\qedhere
\end{align}
\end{lem}

\begin{proof}
We split the proof into five steps.

\medskip
\step1 First-order cluster contribution: proof of~\eqref{eq:exp-u0}.\\
Recall that $u_{\e,N}^{(0)}=v_{\e,N}$ is defined by equation~\eqref{eq:veps}. In terms of the Stokeslet $\Gc$, it can be written as
\begin{multline}\label{eq:uepsN-0}
v_{\e,N}(x)
\,=\,(\Gc\ast (h\mathds1_{\R^d\setminus\Ic_{\e,N}}))(x)+\sum_{m=1}^N\bigg(\int_{I_{\e,N}^m}\Gc(x-\cdot)e+\kappa_0\tfrac{\lambda}{\e N}\int_{I_{\e,N}^{m;+}}\Gc(x-\cdot)f_{\e,N}^m\bigg)\\
\,=\,(\Gc\ast h)(x)
+\sum_{m=1}^N\int_{I_{\e,N}^m}\Gc(x-\cdot)( e-h)+\kappa_0\tfrac{\lambda}{\e N}\sum_{m=1}^N\int_{I_{\e,N}^{m;+}}\Gc(x-\cdot)f_{\e,N}^m.
\end{multline}
We average this over $x\in I_{\e,N}^n$ for some $1\le n\le N$ and it then remains to compute the local integrals.
For the first right-hand side term, using a second-order Taylor expansion, and recalling $\int_{I_{\e,N}^n}(x-X_{\e,N}^n)\,dx=0$, we find
\[\Big|\fint_{I_{\e,N}^n}\Gc\ast h-(\Gc\ast h)(X_{\e,N}^n)\Big|\,\lesssim_h\,\e^2.\]
For the remaining terms, separating the diagonal contributions, the pointwise bounds on~$\Gc$ directly yield
\begin{eqnarray*}
\lefteqn{\sum_{m=1}^N\fint_{I_{\e,N}^n}\Big|\int_{I_{\e,N}^m}\Gc(x-\cdot)(e-h)\Big|\,dx}\\
&\lesssim_h&\e^d\sum_{m:m\ne n}|X_{\e,N}^n-X_{\e,N}^m|^{2-d}+\fint_{I_{\e,N}^n}\Big(\int_{I_{\e,N}^n}|x-\cdot|^{2-d}\Big)dx\\
&\lesssim&\lambda\alpha_{\e,N}^2+\e^2,
\end{eqnarray*}
and similarly, further recalling the definition $f_{\e,N}^m=\e^{-d}f(\frac1\e(\cdot-X_{\e,N}^m),R_{\e,N}^m)$ and the balance of forces~\eqref{eq:local-balance},
\begingroup\allowdisplaybreaks
\begin{eqnarray*}
\lefteqn{\tfrac{\lambda}{\e N}\sum_{m=1}^N\fint_{I_{\e,N}^n}\Big|\int_{I_{\e,N}^{m;+}}\Gc(x-\cdot)f_{\e,N}^m\Big|\,dx}\\
&=&\tfrac{\lambda}{\e N}\sum_{m=1}^N\fint_{I_{\e,N}^n}\bigg|\int_{I_{\e,N}^{m;+}}\Big(\Gc(x-\cdot)-\fint_{I_{\e,N}^{m;+}}\Gc(x-\cdot)\Big)f_{\e,N}^m\bigg|\,dx\\
&\lesssim&\e^{d}\sum_{m:m\ne n}^N|X_{\e,N}^n-X_{\e,N}^m|^{1-d}+\fint_{I_{\e,N}^n}\Big(\int_{I_{\e,N}^{n;+}}|x-\cdot|^{1-d}\Big)\,dx\\
&\lesssim&\lambda\alpha_{\e,N}^1+\e.
\end{eqnarray*}
\endgroup
Inserting these different estimates into~\eqref{eq:uepsN-0}, the conclusion~\eqref{eq:exp-u0} follows.

\medskip
\step{2} Detailed estimates on the linear proxy $u_{\e,N}^{(0)}=v_{\e,N}$: for all $1\le n\le N$, we have the following version of~\eqref{eq:est-fintDu},
\begin{equation}\label{eq:estim-aver-u0}
\Big(\fint_{I_{\e,N}^{n;+}}|\nabla v_{\e,N}|^2\Big)^\frac12
\,\lesssim_h\,1.
\end{equation}
This can be refined as follows, further capturing the leading contribution: denoting by $v_{\e,N}^n$ the unique decaying solution of the following Stokes equation in $\R^d$,\footnote{This equation coincides with~\eqref{eq:veps} up to removing the propulsion force of the $n$-th particle, which would create an additional $O(1)$ self-interaction term at $I_{\e,N}^n$.}
\begin{equation}\label{eq:veps-n}
\left\{\begin{array}{ll}
-\triangle v_{\e,N}^n+\nabla q_{\e,N}^n=h\mathds1_{\R^d\setminus\Ic_{\e,N}}+ e\mathds1_{\Ic_{\e,N}}+\kappa_0\tfrac{\lambda}{\e N}\sum_{p:p\ne n}f_{\e,N}^p,&\text{in $\R^d$},\\
\Div(v_{\e,N}^n)=0,&\text{in $\R^d$},
\end{array}\right.
\end{equation}
we have
\begin{equation}\label{eq:comput-aver-Dv}
\|\nabla v_{\e,N}^n-\nabla (\Gc\ast h)(X_{\e,N}^n)
\|_{\Ld^\infty(I_{\e,N}^{n;+})}
\,\lesssim_h\,
\lambda(\alpha^1_{\e,N}+\kappa_0\alpha^0_{\e,N})+\e.
\end{equation}
The bound~\eqref{eq:estim-aver-u0} is easily obtained by similar estimates as in Step~1, and we rather focus on the proof of~\eqref{eq:comput-aver-Dv}.
For that purpose, we start with the following identity for the solution of~\eqref{eq:veps-n},
\begin{multline*}
\nabla_i v_{\e,N}^n(x)\,=\,\nabla_i(\Gc\ast h)(x)\\
+\sum_{m=1}^N\int_{I_{\e,N}^m}\nabla_i\Gc(x-\cdot)(e-h)
+\kappa_0\tfrac{\lambda}{\e N}\sum_{m:m\ne n}^N\int_{I_{\e,N}^{m;+}}\nabla_i\Gc(x-\cdot)f_{\e,N}^m.
\end{multline*}
As in Step~1, using again pointwise bounds on $\Gc$, we can estimate the last two right-hand side terms in~$I_{\e,N}^{n;+}$ pointwise by $C_h(\lambda\alpha_{\e,N}^1+\e)$ and by $\kappa_0\lambda\alpha_{\e,N}^0$, respectively,
and the claim~\eqref{eq:comput-aver-Dv} follows.

\medskip
\step3 Second-order cluster contribution: proof of~\eqref{eq:exp-u1}.\\
By definition, cf.~\eqref{eq:cluster-u(n)}, we have
\begin{equation}\label{eq:decomp-ueN-weN}
u_{\e,N}^{(1)}\,=\,\sum_{m=1}^Nw_{\e,N}^m,
\end{equation}
in terms of $w_{\e,N}^m:=u_{\e,N}^m-v_{\e,N}$, where we recall the short-hand notation $u_{\e,N}^m:=\pi^{\{m\}}v_{\e,N}$.
As in~\eqref{eq:rel-u-u0}, we find the following equation for $w_{\e,N}^m$ in $\R^d$,
\begin{equation}\label{eq:def-0-1-wepsn-00}
-\triangle w_{\e,N}^m+\nabla p\,=\,-\Big(e\mathds1_{I_{\e,N}^m}+\kappa_0\tfrac{\lambda}{\e N}f_{\e,N}^m\mathds1_{I_{\e,N}^m}+\delta_{\partial I_{\e,N}^m}\sigma(u_{\e,N}^m,p_{\e,N}^m)\nu\Big).
\end{equation}
In terms of the Stokeslet $\Gc$, using boundary conditions for $u_{\e,N}^m$, we get for all $x\in\R^d$,
\begin{equation}\label{eq:decomp-ueN-weN-plop}
w_{\e,N}^m(x)\,=\,-\int_{\partial I_{\e,N}^m}\Big(\Gc(x-\cdot)-\fint_{I_{\e,N}^m}\Gc(x-\cdot)\Big)\sigma(u_{\e,N}^m,p_{\e,N}^m)\nu.
\end{equation}
Averaging this expression over $x\in I_{\e,N}^n$, summing over $m\ne n$, replacing $\Gc$ by its Taylor expansion,
and appealing to a trace estimate and to pointwise bounds on $\Gc$, cf.~Lemmas~\ref{lem:trace} and~\ref{lem:Green}, we are led to
\begin{multline}\label{eq:expand-w0}
\bigg|\sum_{m:m\ne n}^N\fint_{I_{\e,N}^n}\!\! w_{\e,N}^m-\sum_{m:m\ne n}^N\nabla_i\Gc(X_{\e,N}^n-X_{\e,N}^m)\int_{\partial I_{\e,N}^m}(x-X_{\e,N}^m)_i\,\sigma(u_{\e,N}^m, p_{\e,N}^m)\nu\bigg|\\
\,\lesssim\,\tfrac{\e\lambda}N\sum_{m:m\ne n}^N|X_{\e,N}^m-X_{\e,N}^n|^{-d}\Big(\fint_{I_{\e,N}^{m;+}}|\!\D(u_{\e,N}^m)|^2\Big)^\frac12.
\end{multline}
In order to estimate the last factor in the right-hand side, we appeal to an energy estimate:
testing the equation~\eqref{eq:def-0-1-wepsn-00} for $w_{\e,N}^m$ with $w_{\e,N}^m$ itself, and using boundary conditions for $w_{\e,N}^m=u_{\e,N}^m-v_{\e,N}$, we get the energy identity
\begin{eqnarray*}
\int_{\R^d}|\nabla w_{\e,N}^m|^2
&=&-\int_{\partial I_{\e,N}^m}\Big(w_{\e,N}^m-\fint_{I_{\e,N}^m}w_{\e,N}^m\Big)\cdot\sigma(u_{\e,N}^m,p_{\e,N}^m)\nu\\
&=&\int_{\partial I_{\e,N}^m}\Big(v_{\e,N}-\fint_{I_{\e,N}^m}v_{\e,N}\Big)\cdot\sigma(u_{\e,N}^m,p_{\e,N}^m)\nu,
\end{eqnarray*}
hence, by a trace estimate, cf.~Lemma~\ref{lem:trace},
\begin{equation*}
\int_{\R^d}|\!\D (w_{\e,N}^m)|^2
\,\lesssim\,\Big(\int_{I_{\e,N}^{m;+}}|\!\D(v_{\e,N})|^2\Big)^\frac12\Big(\int_{I_{\e,N}^{m;+}}|\!\D(u_{\e,N}^m)|^2\Big)^\frac12,
\end{equation*}
which entails, by the triangle inequality, with $u_{\e,N}^m=w_{\e,N}^m+v_{\e,N}$,
\begin{equation}\label{eq:bnd-apriori-w0m}
\int_{\R^d}|\!\D (w_{\e,N}^m)|^2+\int_{I_{\e,N}^{m;+}}|\!\D (u_{\e,N}^m)|^2
\,\lesssim\,\int_{I_{\e,N}^{m;+}}|\!\D(v_{\e,N})|^2.
\end{equation}
Combining this with~\eqref{eq:estim-aver-u0} and inserting into~\eqref{eq:expand-w0}, we deduce
\begin{multline}\label{eq:expand-w0-1}
\bigg|\sum_{m:m\ne n}^N\fint_{I_{\e,N}^n}\!\! w_{\e,N}^m-\sum_{m:m\ne n}^N\nabla_i\Gc(X_{\e,N}^n-X_{\e,N}^m)\int_{\partial I_{\e,N}^m}(x-X_{\e,N}^m)_i\,\sigma( u_{\e,N}^m, p_{\e,N}^m)\nu\bigg|\\
\,\lesssim_h\,\e\lambda\alpha_{\e,N}^0.
\end{multline}
Recalling~\eqref{eq:decomp-ueN-weN} and writing $w_{\e,N}^n=u_{\e,N}^n-u_{\e,N}^{(0)}$ for the diagonal term, we are then led to
\begin{equation}\label{eq:decomp-u0u1-pre}
\bigg|\fint_{I_{\e,N}^n}\!\!\big(u^{(0)}_{\e,N}+u^{(1)}_{\e,N}\big)-\fint_{I_{\e,N}^n}\!\!u_{\e,N}^n-\sum_{m:m\ne n}^N\nabla\Gc(X_{\e,N}^n-X_{\e,N}^m)\,\Sigma_{\e,N}^m\bigg|
\,\lesssim_h\,\e\lambda\alpha_{\e,N}^0,
\end{equation}
where we have defined the stresslets 
\begin{equation}\label{eq:Sigma-eNm}
\Sigma_{\e,N}^m\,:=\,\int_{\partial I_{\e,N}^m}\sigma(u_{\e,N}^m, p_{\e,N}^m)\nu\otimes_s^\circ(x-X_{\e,N}^m),\qquad1\le m\le N.
\end{equation}
Here, we recall that $\otimes_s^\circ$ stands for the trace-free symmetric tensor product:
we have used both the vanishing torque condition for $u_{\e,N}^m$ and the incompressibility constraint \mbox{$\Div(\Gc)=0$} to restrict $\Sigma_{\e,N}^m$ to its trace-free symmetric part in~\eqref{eq:decomp-u0u1-pre}.
It remains to evaluate the diagonal term and the stresslets in~\eqref{eq:decomp-u0u1-pre}: we claim that for all $1\le m\le N$,
\begin{align}
&\Big|\Sigma_{\e,N}^m-2|I_{\e,N}^m|\Sigma^\circ(R_{\e,N}^m)\D(\Gc\ast h)(X_{\e,N}^m)-\kappa_0|I_{\e,N}^m|\Sigma'_f(R_{\e,N}^m)\Big|\nonumber\\
&\hspace{6.5cm}\,\lesssim_h\,\e^d\big(\lambda(\alpha_{\e,N}^1+\kappa_0\alpha_{\e,N}^0)+\e\big),\label{eq:u01-stress}\\
&\Big|\fint_{I_{\e,N}^m}u_{\e,N}^m-\big[\Gc\,\hat\ast\,\big((1-\lambda\mu_{\e,N})h+\lambda\mu_{\e,N}e\big)\big](X_{\e,N}^m)\nonumber\\
&\hspace{3cm}+\kappa_0\lambda[\nabla\Gc\,\hat\ast\,(\langle\nu_{\e,N}M_f\rangle)](X_{\e,N}^m)-\kappa_0\e V^\circ_f(R_{\e,N}^m)\Big|\nonumber\\
&\hspace{6.5cm}\,\lesssim_h\,\e\big(\lambda(\alpha_{\e,N}^2+\alpha_{\e,N}^1+\kappa_0\alpha_{\e,N}^0)+\e\big),\label{eq:expand-fintum}
\end{align}
where $M_f$ stands for the first moment of the propulsion force,
\begin{equation}\label{eq:def-Mf}
M_f(r)\,:=\,\int_{\R^d}f(\cdot,r)\otimes^\circ x,
\end{equation}
and where we have set
\begin{equation}\label{eq:def-S(r)gamma-re-plop}
\Sigma'_f(r)\,:=\,\int_{\partial I^\circ(r)}\sigma(u^\circ_{r,f},p^\circ_{r,f})\nu\otimes_s^\circ x,
\end{equation}
recalling that $(u_{r,f}^\circ,p_{r,f}^\circ)$ is the unique decaying solution of the single-particle problem~\eqref{eq:eqn-defin-ugamma}.
The proof of these two estimates~\eqref{eq:u01-stress}--\eqref{eq:expand-fintum} is split into the following three substeps.
Inserting them into~\eqref{eq:decomp-u0u1-pre}, noting that we recover $\Sigma_f'-M_f=\Sigma_f^\circ$ as defined in~\eqref{eq:def-S(r)gamma-re}, further using convolution notation, and using~\eqref{eq:Holder-alpha}, the conclusion~\eqref{eq:exp-u1} follows.

\medskip
\substep{3.1} A suitable decomposition of~$u_{\e,N}^m$.\\
Recalling that $u_{\e,N}^m$ satisfies the following single-particle problem,
\[\left\{\begin{array}{ll}
-\triangle u_{\e,N}^m+\nabla p_{\e,N}^m=h\mathds1_{\R^d\setminus\Ic_{\e,N}}+e\mathds1_{\Ic_{\e,N}\setminus I_{\e,N}^m}+\kappa_0\tfrac{\lambda}{\e N}\sum_{p=1}^Nf_{\e,N}^p,&\text{in $\R^d\setminus I_{\e,N}^m$},\\
\Div(u_{\e,N}^m)=0,&\text{in $\R^d\setminus I_{\e,N}^m$},\\
\D(u_{\e,N}^m)=0,&\text{in $I_{\e,N}^m$},\\
\tfrac\lambda Ne+\kappa_0\tfrac{\lambda}{\e N}R_{\e,N}^m+\int_{\partial I_{\e,N}^m}\sigma(u_{\e,N}^m,p_{\e,N}^m)\nu=0,&\\
\int_{\partial I_{\e,N}^m}(x-X_{\e,N}^m)\times\sigma(u_{\e,N}^m,p_{\e,N}^m)\nu=0,&
\end{array}\right.\]
we may naturally decompose it as
\begin{equation}\label{eq:decomp-ueNm}
u_{\e,N}^m\,=\,v_{\e,N}^m+w_{\e,N;0}^m+w_{\e,N;1}^m+e_{\e,N}^m,
\end{equation}
where we recall that $v_{\e,N}^m$ is defined in~\eqref{eq:veps-n}, and where:
\begin{enumerate}[---]
\item $w_{\e,N;0}^m$ is the unique decaying solution of the single-particle problem
\begin{equation}\label{eq:def-0-1-wepsn}
\left\{\begin{array}{ll}
-\triangle w_{\e,N;0}^m+\nabla q_{\e,N;0}^m=0,&\text{in $\R^d\setminus I_{\e,N}^m$},\\
\Div(w_{\e,N;0}^m)=0,&\text{in $\R^d\setminus I_{\e,N}^m$},\\
\D(w_{\e,N;0}^m)+H_{\e,N}^m=0,&\text{in $I_{\e,N}^m$},\\
\int_{\partial I_{\e,N}^m}\sigma(w_{\e,N;0}^m,q_{\e,N;0}^m)\nu=0,&\\
\int_{\partial I_{\e,N}^m}(x-X_{\e,N}^m)\times\sigma(w_{\e,N;0}^m,q_{\e,N;0}^m)\nu=0,&
\end{array}\right.
\end{equation}
where the strain rate $H_{\e,N}^m$ is chosen as
\begin{equation}\label{eq:not-HeNm}
H_{\e,N}^m\,:=\,\D(\Gc\ast h)(X_{\e,N}^m);
\end{equation}
\item $w_{\e,N;1}^m$ is the unique decaying solution of the single-particle problem
\begin{equation}\label{eq:def-0-11-wepsn}
\left\{\begin{array}{ll}
-\triangle w_{\e,N;1}^m+\nabla q_{\e,N;1}^m=\kappa_0\tfrac{\lambda}{\e N}f_{\e,N}^m,&\text{in $\R^d\setminus I_{\e,N}^m$},\\
\Div(w_{\e,N;1}^m)=0,&\text{in $\R^d\setminus I_{\e,N}^m$},\\
\D(w_{\e,N;1}^m)=0,&\text{in $I_{\e,N}^m$},\\
\kappa_0\tfrac{\lambda}{\e N}R_{\e,N}^m+\int_{\partial I_{\e,N}^m}\sigma(w_{\e,N;1}^m,q_{\e,N;1}^m)\nu=0,&\\
\int_{\partial I_{\e,N}^m}(x-X_{\e,N}^m)\times\sigma(w_{\e,N;1}^m,q_{\e,N;1}^m)\nu=0.&
\end{array}\right.
\end{equation}
\end{enumerate}
Using Lemma~\ref{lem:eqn-pseudo}, a direct computation then entails that the remainder
\[e_{\e,N}^m\,:=\,u_{\e,N}^m-v_{\e,N}^m-w_{\e,N;0}^m-w_{\e,N;1}^m\]
satisfies the following equation in $\R^d$,
\begin{equation*}
-\triangle e_{\e,N}^m+\nabla p=-\Big( e\mathds1_{I^m_{\e,N}}
+\delta_{\partial I_{\e,N}^m}\sigma\big(u_{\e,N}^m-w_{\e,N;0}^m-H_{\e,N}^mx-w_{\e,N;1}^m,q_{\e,N}^m-q_{\e,N;0}^m-q_{\e,N;1}^m\big)\nu\Big).
\end{equation*}
Testing this equation with $e_{\e,N}^m$ itself, and using boundary conditions, we get
\begin{eqnarray*}
\int_{\R^d}|\nabla e_{\e,N}^m|^2&=&-\int_{\partial I_{\e,N}^m} \Big(e_{\e,N}^m-\fint_{I_{\e,N}^m}e_{\e,N}^m\Big)\\
&&\hspace{1cm}\cdot\sigma\big(u_{\e,N}^m-w_{\e,N;0}^m-H_{\e,N}^mx-w_{\e,N;1}^m,q_{\e,N}^m-q_{\e,N;0}^m-q_{\e,N;1}^m\big)\nu\\
&=&\int_{\partial I_{\e,N}^m} \Big((v_{\e,N}^m-H_{\e,N}^mx)-\fint_{I_{\e,N}^m}(v_{\e,N}^m-H_{\e,N}^mx)\Big)\\
&&\hspace{1cm}\cdot\sigma\big(u_{\e,N}^m-w_{\e,N;0}^m-H_{\e,N}^mx-w_{\e,N;1}^m,q_{\e,N}^m-q_{\e,N;0}^m-q_{\e,N;1}^m\big)\nu,
\end{eqnarray*}
hence, by a trace estimate, cf.~Lemma~\ref{lem:trace},
\begin{equation*}
\int_{\R^d}|\nabla e_{\e,N}^m|^2
\,\lesssim\,\Big(\int_{I_{\e,N}^m}|\!\D(v_{\e,N}^m)-H_{\e,N}^m|^2\Big)^\frac12\Big(\int_{I_{\e,N}^{m;+}}|\!\D(u_{\e,N}^m-w_{\e,N;0}^m-w_{\e,N;1}^m)-H_{\e,N}^m|^2\Big)^\frac12.
\end{equation*}
By the triangle inequality, reconstructing $e_{\e,N}^m$ in the last factor, we are led to
\begin{equation*}
\int_{\R^d}|\nabla e_{\e,N}^m|^2
\,\lesssim\,\int_{I_{\e,N}^m}|\!\D(v_{\e,N}^m)-H_{\e,N}^m|^2,
\end{equation*}
and thus, appealing to~\eqref{eq:comput-aver-Dv} and recalling the choice~\eqref{eq:not-HeNm},
\begin{equation}\label{eq:estim-eNm}
\int_{\R^d}|\nabla e_{\e,N}^m|^2
\,\lesssim_h\,\e^d\big(\lambda(\alpha_{\e,N}^1+\kappa_0\alpha_{\e,N}^0)+\e\big)^2.
\end{equation}

\medskip
\substep{3.2} Proof of~\eqref{eq:u01-stress}.\\
Inserting the above decomposition~\eqref{eq:decomp-ueNm} for $u_{\e,N}^m$ into the definition~\eqref{eq:Sigma-eNm} of the stresslet, and using the remainder estimate~\eqref{eq:estim-eNm} together with a trace estimate, cf.~Lemma~\ref{lem:trace}, we find
\begin{multline}\label{eq:SigmaeNm-errorestim}
\bigg|\Sigma_{\e,N}^m-\int_{\partial I_{\e,N}^m}\sigma(v_{\e,N}^m,q_{\e,N}^m)\nu\otimes_s^\circ(x-X_{\e,N}^m)
-\int_{\partial I_{\e,N}^m}\sigma(w_{\e,N;0}^m,q_{\e,N;0}^m)\nu\otimes_s^\circ(x-X_{\e,N}^m)\\
-\int_{\partial I_{\e,N}^m}\sigma(w_{\e,N;1}^m,q_{\e,N;1}^m)\nu\otimes_s^\circ(x-X_{\e,N}^m)\bigg|\\
\,\lesssim\,\e^d\Big(\fint_{I_{\e,N}^{m;+}}|\!\D(e_{\e,N}^m)|^2\Big)^\frac12
\,\lesssim_h\,\e^d\big(\lambda(\alpha_{\e,N}^1+\kappa_0\alpha_{\e,N}^0)+\e\big).
\end{multline}
It remains to evaluate the different terms in the left-hand side.
First, we compute
\[\int_{\partial I_{\e,N}^m}\sigma(v_{\e,N}^m,q_{\e,N}^m)\nu\otimes_s^\circ(x-X_{\e,N}^m)\,=\,\int_{I_{\e,N}^m}2\!\D(v_{\e,N}^m),\]
and thus, by~\eqref{eq:comput-aver-Dv},
\begin{multline*}
\Big|\int_{\partial I_{\e,N}^m}\sigma(v_{\e,N}^m,q_{\e,N}^m)\nu\otimes_s^\circ(x-X_{\e,N}^m)
-2|I_{\e,N}^m|\D(\Gc\ast h)(X_{\e,N}^m)\Big|\\
\,\lesssim_h\,\e^d\big(\lambda(\alpha^1_{\e,N}+\kappa_0\alpha^0_{\e,N})+\e\big).
\end{multline*}
Next, as $w_{\e,N;0}^m$ and $w_{\e,N;1}^m$ satisfy the single-particle problems~\eqref{eq:def-0-1-wepsn}--\eqref{eq:def-0-11-wepsn}, which can be compared to~\eqref{eq:eqn-uRE0} and~\eqref{eq:eqn-defin-ugamma}, we simply find by scaling
\begin{eqnarray*}
\int_{\partial I_{\e,N}^m}\sigma(w_{\e,N;0}^m,q_{\e,N;0}^m)\nu\otimes_s^\circ(x-X_{\e,N}^m)&=&2|I_{\e,N}^m|\big(\Sigma^\circ(R_{\e,N}^m)H_{\e,N}^m-H_{\e,N}^m\big),\\
\int_{\partial I_{\e,N}^m}\sigma(w_{\e,N;1}^m,q_{\e,N;1}^m)\nu\otimes_s^\circ(x-X_{\e,N}^m)&=&\kappa_0|I_{\e,N}^m|\Sigma'_f(R_{\e,N}^m),
\end{eqnarray*}
where we recall the definition of $\Sigma^\circ,\Sigma'_f$ in~\eqref{eq:def-S(r)-bis} and~\eqref{eq:def-S(r)gamma-re-plop}.
Inserting these different computations into~\eqref{eq:SigmaeNm-errorestim}, and recalling $H_{\e,N}^m=\D(\Gc\ast h)(X_{\e,N}^m)$, cf.~\eqref{eq:not-HeNm}, the claim~\eqref{eq:u01-stress} follows.

\medskip
\substep{3.3} Proof of~\eqref{eq:expand-fintum}.\\
Using again the decomposition~\eqref{eq:decomp-ueNm} for $u_{\e,N}^m$, and noting that, by the Sobolev embedding, the remainder estimate~\eqref{eq:estim-eNm} yields
\begin{eqnarray*}
\Big|\fint_{I_{\e,N}^m}e_{\e,N}^m\Big|
&\lesssim&\e^{1-\frac{d}{2}}\Big(\int_{I_{\e,N}^m}|e_{\e,N}^m|^{\frac{2d}{d-2}}\Big)^\frac{d-2}{2d}\\
&\lesssim&\e^{1-\frac{d}{2}}\Big(\int_{I_{\e,N}^m}|\nabla e_{\e,N}^m|^2\Big)^\frac12\\
&\lesssim_h&\e\big(\lambda(\alpha_{\e,N}^1+\kappa_0\alpha_{\e,N}^0)+\e\big),
\end{eqnarray*}
we deduce
\begin{equation}\label{eq:expand-fintum-pre}
\Big|\fint_{I_{\e,N}^m}u_{\e,N}^m-\fint_{I_{\e,N}^m}v_{\e,N}^m-\fint_{I_{\e,N}^m}w_{\e,N;0}^m-\fint_{I_{\e,N}^m}w_{\e,N;1}^m\Big|\,\lesssim_h\,\e\big(\lambda(\alpha_{\e,N}^1+\kappa_0\alpha_{\e,N}^0)+\e\big),
\end{equation}
and it remains to evaluate the different terms in the left-hand side.
First, recalling equation~\eqref{eq:veps-n} for $v_{\e,N}^m$, we can represent, in terms of the Stokeslet $\Gc$,
\begin{equation}\label{eq:ven-repG}
v_{\e,N}^m(x)\,=\,(\Gc\ast h)(x)+\sum_{p=1}^N\int_{I_{\e,N}^p}\Gc(x-\cdot)(e-h)+\kappa_0\tfrac\lambda{\e N}\sum_{p:p\ne m}^N\int_{I_{\e,N}^{p;+}}\Gc(x-\cdot)f_{\e,N}^p.
\end{equation}
Averaging this over $x\in I_{\e,N}^m$, using a Taylor expansion and the pointwise bounds on $\Gc$, cf.~Lemma~\ref{lem:Green},
and recalling $f_{\e,N}^m=\e^{-d}f(\frac1\e(\cdot-X_{\e,N}^m),R_{\e,N}^m)$ and the balance of forces~\eqref{eq:local-balance},
we find for the off-diagonal contributions,
\begin{align*}
&\Big|\fint_{I_{\e,N}^m}\Gc\ast h-(\Gc\ast h)(X_{\e,N}^m)\Big|\,\lesssim_h\,\e^2,\\
&\bigg|\sum_{p:p\ne m}^N\fint_{I_{\e,N}^m}\Big(\int_{I_{\e,N}^p}\Gc(x-\cdot)(e-h)\Big)\,dx-\tfrac\lambda N\sum_{p:p\ne m}^N\Gc(X_{\e,N}^m-X_{\e,N}^p)\big(e-h(X_{\e,N}^p)\big)\bigg|\\
&\hspace{10.8cm}\,\lesssim_h\,\e\lambda(\alpha_{\e,N}^2+\alpha_{\e,N}^1),\\
&\bigg|\tfrac\lambda{\e N}\sum_{p:p\ne m}^N\fint_{I_{\e,N}^m}\Big(\int_{I_{\e,N}^{p;+}}\Gc(x-\cdot)f_{\e,N}^p\Big)+\tfrac\lambda N\sum_{p:p\ne m}^N\nabla\Gc(X_{\e,N}^m-X_{\e,N}^p)M_f(R_{\e,N}^p)\bigg|\,\lesssim\,\e\lambda\alpha_{\e,N}^0,
\end{align*}
in terms of the first moment $M_f$ of the propulsion force, cf.~\eqref{eq:def-Mf}.
In addition, for the diagonal contribution in the second right-hand side term of~\eqref{eq:ven-repG}, we can estimate
\[\fint_{I_{\e,N}^m}\Big|\int_{I_{\e,N}^m}\Gc(x-\cdot)(e-h)\Big|\,\lesssim_h\,\fint_{I_{\e,N}^m}\Big(\int_{I_{\e,N}^m}|x-\cdot|^{2-d}\Big)\,\lesssim\,\e^2.\]
Inserting these different computations into~\eqref{eq:ven-repG}, and using convolution notation, we get
\begin{multline}\label{eq:ven-expand}
\Big|\fint_{I_{\e,N}^m}v_{\e,N}^m
-\big[\Gc\,\hat\ast\,\big((1-\lambda\mu_{\e,N})h+\lambda\mu_{\e,N}e\big)\big](X_{\e,N}^m)
+\kappa_0\lambda \big[\nabla\Gc\ast(\langle\nu_{\e,N}M_f\rangle)\big](X_{\e,N}^m)\Big|\\
\,\lesssim_h\,\e\big(\lambda(\alpha^2_{\e,N}+\alpha^1_{\e,N}+\kappa_0\alpha^0_{\e,N})+\e\big).
\end{multline}
We turn to the analysis of the last two left-hand side contributions in~\eqref{eq:expand-fintum-pre}.
As $w_{\e,N;0}^m$ solves the single-particle problem~\eqref{eq:def-0-1-wepsn}, we note that it actually satisfies the Dirichlet condition $w_{\e,N;0}^m(x)=-H_{\e,N}^m(x-X_{\e,N}^m)$ in $I_{\e,N}^m$, hence
\[\fint_{I_{\e,N}^m}w^m_{\e,N;0}\,=\,0.\] 
 As $w_{\e,N;1}^m$ satisfies the single-particle problem~\eqref{eq:def-0-11-wepsn}, which can be compared to~\eqref{eq:eqn-defin-ugamma}, we find by scaling
\begin{eqnarray*}
\fint_{I_{\e,N}^m}w^m_{\e,N;1}\,=\,\kappa_0\e V^\circ_f(R_{\e,N}^m),
\end{eqnarray*}
where we recall the definition~\eqref{eq:def-act-drag-0} of $V^\circ_f$.
Inserting these identities into~\eqref{eq:expand-fintum-pre}, together with~\eqref{eq:ven-expand}, recalling $H_{\e,N}^m=\D(\Gc\ast h)(X_{\e,N}^m)$, cf.~\eqref{eq:not-HeNm}, and using convolution notation, the claim~\eqref{eq:expand-fintum} follows.

\medskip
\step4 Angular velocities: proof of~\eqref{eq:exp-nabu0}.\\
Recalling the short-hand notation $u_{\e,N}^n=\pi^{\{n\}}_{\e,N}u^{(0)}_{\e,N}$, we start again with the decomposition~\eqref{eq:decomp-ueNm}: using the remainder estimate~\eqref{eq:estim-eNm}, we find
\begin{equation}\label{eq:decomp-nabu}
\Big|\fint_{I_{\e,N}^n}\!\!\nabla u^n_{\e,N}-\fint_{I_{\e,N}^n}\!\!\nabla v_{\e,N}^n-\fint_{I_{\e,N}^n}\!\!\nabla w_{\e,N;0}^n-\fint_{I_{\e,N}^n}\!\!\nabla w_{\e,N;1}^n\Big|\,\lesssim_h\,\lambda(\alpha_{\e,N}^1+\kappa_0\alpha_{\e,N}^0)+\e,
\end{equation}
and it remains to evaluate the different terms in the left-hand side. First, we get from~\eqref{eq:comput-aver-Dv},
\begin{equation*}
\Big|\fint_{I_{\e,N}^n}\!\!\nabla v_{\e,N}^n-\nabla(\Gc\ast h)(X_{\e,N}^n)\Big|
\,\lesssim_h\,
\lambda(\alpha_{\e,N}^1+\kappa_0\alpha_{\e,N}^0)+\e.
\end{equation*}
Next, as $w_{\e,N;0}^n$ satisfies the single-particle problems~\eqref{eq:def-0-1-wepsn}, which can be compared to~\eqref{eq:eqn-uRE0}, we find by scaling
\begin{eqnarray*}
\fint_{I_{\e,N}^n}\nabla w_{\e,N;0}^n&=&\fint_{I^\circ(R_{\e,N}^n)}\nabla u^\circ_{R_{\e,N}^n,H_{\e,N}^n}\\
&=&\Omega^\circ(R_{\e,N}^n)\nabla(\Gc\ast h)(X_{\e,N}^n)-\nabla(\Gc\ast h)(X_{\e,N}^n),
\end{eqnarray*}
where we recall the definition~\eqref{eq:defin-Omega0} of $\Omega^\circ$ and the choice~\eqref{eq:not-HeNm} of $H_{\e,N}^n$ as the symmetric part of $\nabla(\Gc\ast h)(X_{\e,N}^n)$.
Finally, as~$w_{\e,N;1}^n$ satisfies the single-particle problem~\eqref{eq:def-0-11-wepsn} with rigidity constraint $\D(w_{\e,N;1}^n)=0$ in $I_{\e,N}^n$, and as the inclusion $I_{\e,N}^n$ and the propulsion force~$f_{\e,N}^n$ are both axisymmetric in the direction~$R_{\e,N}^n$, centered at $X_{\e,N}^n$, we find by symmetry
\begin{equation*}
\fint_{I_{\e,N}^n}\nabla w_{\e,N;1}^n\,=\,0.
\end{equation*}
Inserting these different computations into~\eqref{eq:decomp-nabu}, the conclusion~\eqref{eq:exp-nabu0} follows.

\medskip
\step5 Fluid velocity away from the particles: proof of~\eqref{eq:estim-u0u1}--\eqref{eq:estim-u0u1-bis}.\\
Let the boundary-layer thickness $\delta\in[\e,1]$ be fixed.
The proof of~\eqref{eq:estim-u0u1}--\eqref{eq:estim-u0u1-bis} is slightly simpler than that of~\eqref{eq:exp-u0}--\eqref{eq:exp-u1} as self-interaction terms can be ignored away from the particles.
We start with~\eqref{eq:estim-u0u1}. Recalling the representation~\eqref{eq:uepsN-0} for $u_{\e,N}^{(0)}=v_{\e,N}$, and using the pointwise bounds on $\Gc$, we get for all $x\in\R^d\setminus(\Ic_{\e,N}+\delta B)$,
\[\big|u_{\e,N}^{(0)}(x)-\Gc\ast h(x)\big|\,\lesssim_h\,\tfrac\lambda N\sum_{m=1}^N\Big(|x-X_{\e,N}^m|^{2-d}+\kappa_0|x-X_{\e,N}^m|^{1-d}\Big).\]
To estimate the right-hand side, we use~\eqref{eq:bnd-x-alpha} in the following modified form away from the particles: for all $x\in \R^d\setminus(\Ic_{\e,N}+\delta B)$, we can estimate for any $\sigma\in[0,d]$, after separating the diagonal contribution as in~\eqref{eq:bnd-x-alpha-re},
\begin{equation}\label{eq:bnd-x-alpha-reint}
\tfrac1N\sum_{m=1}^N|x-X_{\e,N}^m|^{\sigma-d}\,\lesssim\,\alpha_{\e,N}^\sigma+\tfrac1N\delta^{\sigma-d}\,\simeq\,\alpha_{\e,N}^\sigma+\tfrac1\lambda\delta^{\sigma}(\tfrac\e\delta)^d.
\end{equation}
Using this to estimate the above right-hand side, we get the conclusion~\eqref{eq:estim-u0u1}.

\medskip\noindent
We turn to the proof of~\eqref{eq:estim-u0u1-bis} and we start with an improved estimate on~$u_{\e,N}^{(0)}=v_{\e,N}$.
Replacing $\Gc$ by its Taylor expansion in the representation~\eqref{eq:uepsN-0} for the latter, we find for all $x\in\R^d\setminus(\Ic_{\e,N}+\delta B)$,
\begin{multline*}
\Big|u_{\e,N}^{(0)}(x)-\big[\Gc\ast \big((1-\lambda\mu_{\e,N})h+\lambda\mu_{\e,N}e\big)\big](x)+\kappa_0\lambda\big[\nabla\Gc\ast(\langle M_f\nu_{\e,N}\rangle)\big](x)\Big|\\
\,\lesssim_h\,\e\tfrac\lambda N\sum_{m=1}^N\big(|x-X_{\e,N}^m|^{2-d}+|x-X_{\e,N}^m|^{1-d}+\kappa_0|x-X_{\e,N}^m|^{-d}\big),
\end{multline*}
and thus, using~\eqref{eq:bnd-x-alpha-reint} again to estimate the right-hand side,
\begin{multline}\label{eq:improved-exp-u0-far}
\Big\|u_{\e,N}^{(0)}-\Gc\ast \big((1-\lambda\mu_{\e,N})h+\lambda\mu_{\e,N}e\big)+\kappa_0\lambda\nabla\Gc\ast(\langle M_f\nu_{\e,N}\rangle)\Big\|_{\Ld^\infty(\R^d\setminus(\Ic_{\e,N}+\delta B))}\\
\,\lesssim_h\,\e\Big(\lambda\big(\alpha_{\e,N}^2+\alpha_{\e,N}^1+\kappa_0\alpha_{\e,N}^0\big)+(\delta+\kappa_0)(\tfrac\e\delta)^d\Big).
\end{multline}
We turn to the corresponding analysis of $u_{\e,N}^{(1)}$.
As in~\eqref{eq:decomp-ueN-weN}--\eqref{eq:decomp-ueN-weN-plop}, we can write
\begin{equation}\label{eq:decomp-ueN-weN-re}
u_{\e,N}^{(1)}\,=\,-\sum_{m=1}^N\int_{\partial I_{\e,N}^m}\Big(\Gc(x-\cdot)-\fint_{I_{\e,N}^m}\Gc(x-\cdot)\Big)\sigma(u_{\e,N}^m,p_{\e,N}^m)\nu.
\end{equation}
Replacing $\Gc$ by its Taylor expansion, we deduce for all $x\in\R^d\setminus(\Ic_{\e,N}+\delta B)$,
\[\Big|u_{\e,N}^{(1)}(x)-\sum_{m=1}^N\nabla\Gc(x-X^m)\Sigma_{\e,N}^m\Big|\,\lesssim\,\e\tfrac\lambda N\sum_{m=1}^N|x-X_{\e,N}^m|^{-d}\Big(\fint_{I_{\e,N}^m}|\!\D(u_{\e,N}^m)|^2\Big)^\frac12,\]
in terms of the stresslets $\{\Sigma_{\e,N}^m\}_m$ defined in~\eqref{eq:Sigma-eNm}.
Appealing to~\eqref{eq:bnd-apriori-w0m} in combination with~\eqref{eq:estim-aver-u0} to estimate the last factor, and using~\eqref{eq:bnd-x-alpha-reint} again, we get
\[\Big\|u_{\e,N}^{(1)}-\sum_{m=1}^N\nabla\Gc(\cdot-X^m)\Sigma_{\e,N}^m\Big\|_{\Ld^\infty(\R^d\setminus(\Ic_{\e,N}+\delta B))}\,\lesssim_h\,\e\big(\lambda\alpha^0_{\e,N}+(\tfrac\e\delta)^d\big).\]
Now inserting the approximation~\eqref{eq:u01-stress} for the stresslets, using~\eqref{eq:bnd-x-alpha-reint} again, using convolution notation, and noting that $(\tfrac\e\delta)^d\delta\le\e$, we deduce
\begin{multline*}
\Big\|u^{(1)}_{\e,N}-\lambda\nabla\Gc\ast\big(2\langle\Sigma^\circ\nu_{\e,N}\rangle\!\D(\Gc\ast h)+\kappa_0\langle\Sigma_f'\nu_{\e,N}\rangle\big)\Big\|_{\Ld^\infty(\R^d\setminus(\Ic_{\e,N}+\delta B))}\\
\,\lesssim_h\,\lambda\alpha_{\e,N}^1\big(\lambda(\alpha_{\e,N}^1+\kappa_0\alpha_{\e,N}^0)+\e\big)+\e\big(\lambda\alpha_{\e,N}^0+(\tfrac\e\delta)^d\big).
\end{multline*}
Combining this with~\eqref{eq:improved-exp-u0-far}, recalling $\Sigma_f^\circ=\Sigma_f'-M_f$, and using~\eqref{eq:Holder-alpha}, we get the conclusion~\eqref{eq:estim-u0u1-bis}.
\end{proof}

\medskip
\section{Mean-field approximation}\label{sec:MFL}
Given the dilute expansion of particle velocities in Proposition~\ref{prop:dilute-exp}, we now appeal to a mean-field argument to derive a macroscopic equation for the particle density.
For that purpose, as in~\cite{Hofer-18,Mecherbet-19,Hofer-Schubert-21,Hofer-Schubert-23}, we use the Wasserstein method developed by Hauray and Jabin in~\cite{Hauray-Jabin-07,Hauray-Jabin-15} (see also~\cite{Carrillo-Choi-Hauray-14}). It takes form of a buckling argument, based on the following observation: proving the mean-field result requires controlling local information such as the interparticle distance, while bounds on the latter are actually improved if the mean-field result is known to hold.
The method applies whenever particle interactions are strictly less singular than Coulomb forces at short distances. In the present setting, hydrodynamic interactions involve the kernel~$\lambda\nabla\Gc$, cf.~Proposition~\ref{prop:dilute-exp}, which has the {\it same} singularity as Coulomb forces, but the prefactor~$O(\lambda)$ allows to treat it perturbatively using sub-Coulomb techniques.
We shall actually use an improved version of the Wasserstein method by Höfer and Schubert~\cite{Hofer-Schubert-21,Hofer-Schubert-23}, which relies on the following estimate,
describing how the closeness of the empirical measure to some bounded density yields improved bounds on key quantities like~$\{\alpha_{\e,N}^\sigma\}_\sigma$, cf.~\eqref{eq:def-alpha}.

\begin{lem}[\cite{Hofer-Schubert-21,Hofer-Schubert-23}]\label{lem:Richard}
For all $\sigma\in(0,d]$, $1\le p\le\infty$, and $\ell>0$, there holds
\begin{multline*}
\max_{1\le n\le N}\tfrac1N\sum_{m:m\ne n}|X_{\e,N}^n-X_{\e,N}^m|^{\sigma-d}\,\mathds1_{|X_{\e,N}^n-X_{\e,N}^m|\le\ell}\\
\,\lesssim\,
\tfrac1{\sigma}\|\mu\|_{\Ld^1\cap\Ld^\infty(\R^d)}\Big((1\wedge\ell)^\sigma
+\big(1+(N^\frac1dd_{\e,N}^{\min})^{\sigma -d}\big)W_p(\mu_{\e,N},\mu)^{\frac{\sigma p}{p+d}}\Big),
\end{multline*}
and moreover, for $\sigma=0$, there holds for all $\delta,\ell>0$,
\begin{align*}
&\max_{1\le n\le N}\tfrac1N\sum_{m:m\ne n}|X_{\e,N}^n-X_{\e,N}^m|^{-d}\,\mathds1_{\delta\le|X_{\e,N}^n-X_{\e,N}^m|\le\ell}\\
&\hspace{1cm}\,\lesssim\,\|\mu\|_{\Ld^1\cap\Ld^\infty(\R^d)}\bigg(\log\big(2+\delta^{-1}(1\wedge\ell)\big)\\
&\hspace{4cm}+(N^\frac1dd_{\e,N}^{\min})^{-d}\log\Big(2+\delta^{-1}(N^\frac1dd_{\e,N}^{\min})W_p(\mu_{\e,N},\mu)^\frac{p}{d+p}\Big)\bigg).\qedhere
\end{align*}
\end{lem}

\begin{proof}
This is an immediate reformulation of corresponding results in~\cite{Hofer-Schubert-21,Hofer-Schubert-23}.
First, for~$0<\sigma\le d$, we can estimate for any $R>0$,
\[\||\cdot|^{\sigma-d}\ast\psi\|_{\Ld^\infty(\R^d)}\,\lesssim\,\tfrac1{\sigma}R^{\sigma}\|\psi\|_{\Ld^\infty(\R^d)}+R^{\sigma-d}\|\psi\|_{\Ld^1(\R^d)}.\]
and therefore, after optimizing with respect to $R$,
\begin{equation}\label{eq:estim-Linfty-L1}
\||\cdot|^{\sigma-d}\ast\psi\|_{\Ld^\infty(\R^d)}\,\lesssim\,\tfrac1{\sigma}\|\psi\|_{\Ld^1(\R^d)}^{\frac\sigma d}\|\psi\|_{\Ld^\infty(\R^d)}^{1-\frac\sigma d}.
\end{equation}
Using this instead of~\cite[eqn~(2.6)]{Hofer-Schubert-23}, and repeating the proof of~\cite[Lemma~2.3]{Hofer-Schubert-23}, the first part of the statement follows.
Next, in case $\sigma=0$, we rather start with the following observation: given~$\delta>0$, using~\eqref{eq:estim-Linfty-L1}, we can bound for all $0<\tau\le d$,
\begin{eqnarray*}
\|(\delta+|\cdot|)^{-d}\ast\psi\|_{\Ld^\infty(\R^d)}&\le&\delta^{-\tau}\||\cdot|^{\tau-d}\ast\psi\|_{\Ld^\infty(\R^d)}\\
&\lesssim&\tfrac1\tau\delta^{-\tau}\|\psi\|_{\Ld^1(\R^d)}^\frac\tau d\|\psi\|_{\Ld^\infty(\R^d)}^{1-\frac\tau d},
\end{eqnarray*}
and thus, after optimizing with respect to $\tau$,
\[\|(\delta+|\cdot|)^{-d}\ast\psi\|_{\Ld^\infty(\R^d)}\,\lesssim\,\|\psi\|_{\Ld^1\cap\Ld^\infty(\R^d)}\log(2+\delta^{-1}).\]
Now using~this estimate instead of~\eqref{eq:estim-Linfty-L1}, and repeating the proof of~\cite[Lemma~2.3]{Hofer-Schubert-23}, the second part of the statement follows similarly.
\end{proof}

We start with the following mean-field approximation.
This result is obtained in $\infty$-Wasserstein metric, which is best suited to analyze singular interactions, and it will be upgraded to weaker $p$-Wasserstein metric in Section~\ref{sec:conclusion}. 

\begin{prop}\label{prop:MFL}
Assume that there exist solutions $\tilde\nu_0,\tilde\nu_\lambda\in C([0,T];\Pc\cap\Ld^\infty(\R^d\times\Sp^{d-1}))$ of the following linear coupled transport equations up to some time $T>0$,
\begin{equation}\label{eq:tsp-munu}
\left\{\begin{array}{l}
\partial_t\tilde\nu_0+\Div_x(\tilde\nu_0 (\Gc\ast h))+\Div_r(\tilde\nu_0\tilde\Omega r)=0,\\
\partial_t\tilde\nu_\lambda+\Div_x(\tilde\nu_\lambda\tilde u_\lambda)+\Div_r(\tilde\nu_\lambda\tilde\Omega r)=0,\\
\tilde\nu_0|_{t=0}=
\tilde\nu_\lambda|_{t=0}=\nu^\circ,
\end{array}\right.
\end{equation}
where the translational and angular velocity fields $\tilde u_\lambda$ and $\tilde\Omega$ are given by
\begin{eqnarray}
\tilde u_\lambda(x)&:=&\big[\Gc\ast \big((1-\lambda\tilde\mu_0)h+\lambda\tilde\mu_0 e\big)\big](x)\label{eq:def-Vlim}\\
&&\qquad+\lambda\big[\nabla\Gc\ast\big(2\langle\Sigma^\circ\tilde\nu_0\rangle\D(\Gc\ast h)+\kappa_0 \langle \Sigma_f^\circ\tilde\nu_0\rangle\big)\big](x),\nonumber\\
\tilde \Omega(x,r)&:=&\Omega^\circ(r)\nabla(\Gc\ast h)(x),\nonumber
\end{eqnarray}
in terms of spatial densities
$\tilde\mu_0:=\langle\tilde\nu_0\rangle$ and
$\tilde\mu_\lambda:=\langle\tilde\nu_\lambda\rangle$.
We also set $\mu^\circ:=\langle\nu^\circ\rangle$.
Consider the $\infty$-Wasserstein distances between the empirical measures and these mean-field approximations $\tilde\mu_0,\tilde\nu_0$ or~$\tilde\mu_\lambda,\tilde\nu_\lambda$,
\begin{eqnarray}\label{eq:defin-WZinfty}
W_{\e,N}^\infty(t)&:=&W_\infty(\mu_{\e,N}^t,\tilde\mu_\lambda^t),\\
Z_{\e,N}^\infty(t)&:=&W_\infty(\nu_{\e,N}^t,\tilde\nu_\lambda^t)\vee W_\infty(\nu_{\e,N}^t,\tilde\nu_0^t).\nonumber
\end{eqnarray}
Assume that initial particle positions satisfy, for some $\theta_0\ge1$,
\begin{equation}\label{eq:cond-wellsep-prop52}
\dd_{\e,N}^{\min}(0)\,\ge\,(\theta_0N)^{-\frac1d},\qquad
Z_{\e,N}^\infty(0)\,\le\,\theta_0.
\end{equation}
Given $\theta>\theta_0$ with $(\theta N)^{-\frac1d}\ge4\e$ and $\lambda\theta\log(\theta+ N)\|\mu^\circ\|_{\Ld^1\cap\Ld^\infty(\R^d)}\ll1$ small enough, there is a maximal time $T_\theta>C_{h,\mu^\circ}^{-1}\log(\theta/\theta_0)$ such that we have for all $t\in[0,T_\theta\wedge T]$,
\begin{eqnarray}
Z_{\e,N}^\infty(t)&\lesssim_{t,h,\theta,\mu^\circ}&Z_{\e,N}^\infty(0)+\lambda\log N+\e,\label{eq:estim-fin-Z}\\
W_{\e,N}^\infty(t)&\lesssim_{t,h,\theta,\mu^\circ}&W_{\e,N}^\infty(0)
+\lambda|\!\log\lambda|Z_{\e,N}^\infty(0)\label{eq:estim-fin-W}\\
&&\quad+(\lambda|\!\log\lambda|+\e)(\lambda\log N+\e)+\kappa_0\e,\nonumber
\end{eqnarray}
and for any boundary-layer thickness $\delta\in[\e,1]$,
\begin{align}\label{eq:estim-fin-u}
&\|u_{\e,N}-\tilde u_\lambda\|_{\Ld^\infty(\R^d\setminus(\Ic_{\e,N}+\delta B))}
\,\lesssim_{t,h,\theta,\mu^\circ}\,
\lambda|\!\log\lambda|Z_{\e,N}^\infty(0)\\
&\hspace{5cm}+(\lambda|\!\log\lambda|+\e)(\lambda\log N+\e)
+(\tfrac\e\delta)^d(\lambda+ \delta).\nonumber\qedhere
\end{align}
\end{prop}

\begin{proof}
As $\tilde\mu_\lambda$ and $\tilde\nu_\lambda$ are absolutely continuous, the $\infty$-Wasserstein distances $W_\infty(\mu_{\e,N},\tilde\mu_\lambda)$ and $W_\infty(\nu_{\e,N},\tilde\nu_\lambda)$ are known to admit optimal transport maps (see~\cite{Champion-08}).
For all $t\in[0,T]$, we denote by $\hat X_{\e,N;\lambda}^t:\R^d\to\R^d$ an optimal transport map such that
\begin{equation}\label{eq:def-hatX}
(\hat X_{\e,N;\lambda}^t)_*\tilde\mu_\lambda^t\,=\,\mu_{\e,N}^t,\qquad W_{\e,N}^\infty(t)\,=\,W^\infty(\mu_{\e,N}^t,\tilde\mu_\lambda^t)\,=~~\supessd{x;\tilde\mu_\lambda^t}|x-\hat X_{\e,N;\lambda}^t(x)|.
\end{equation}
We also denote by $\hat X_{\e,N;0}^t:\R^d\to\R^d$ an optimal transport map such that
\begin{equation}\label{eq:def-hatX00}
(\hat X_{\e,N;0}^t)_*\tilde\mu_0^t\,=\,\mu_{\e,N}^t,\qquad W^\infty(\mu_{\e,N}^t,\tilde\mu_0^t)\,=~~\supessd{x;\tilde\mu_0^t}|x-\hat X_{\e,N;0}^t(x)|,
\end{equation}
and by $(\tilde X_{\e,N;0}^t,\tilde R_{\e,N;0}^t):\R^d\times\Sp^{d-1}\to\R^d\times\Sp^{d-1}$ an optimal transport map such that 
\begin{gather}
(\tilde X_{\e,N;0}^t,\tilde R_{\e,N;0}^t)_*\tilde\nu_0^t\,=\,\nu_{\e,N}^t,\label{eq:def-tildeXR}\\
W_\infty(\nu_{\e,N}^t,\tilde\nu_0^t)\,=~~\supessd{x,r;\tilde\nu_0^t}\Big(|x-\tilde X_{\e,N;0}^t(x,r)|+|r-\tilde R_{\e,N;0}^t(x,r)|\Big).\nonumber
\end{gather}
Here, we use the notation $\supess_{x;\mu}g(x)$ for the essential supremum of a function $g$ with respect to a measure $\mu$.
Also note that averaging the mean-field transport equation~\eqref{eq:tsp-munu} over $r$ yields
\begin{equation}\label{eq:equation-tildemu}
\partial_t\tilde\mu_0+\Div_x(\tilde\mu_0(\Gc\ast h))=0,\qquad\partial_t\tilde\mu_\lambda+\Div_x(\tilde\mu_\lambda\tilde u_\lambda)=0,
\end{equation}
for which the following a priori estimate holds,
\begin{equation}\label{eq:estim-tildemu}
\sup_t\|\tilde\mu_0^t\|_{\Ld^1\cap\Ld^\infty(\R^d)}\,=\,\sup_t\|\tilde\mu_\lambda^t\|_{\Ld^1\cap\Ld^\infty(\R^d)}\,=\,\|\mu^\circ\|_{\Ld^1\cap\Ld^\infty(\R^d)}.
\end{equation}
With these notations and observations, we now split the proof into four steps.

\medskip
\step1 First-order control on $Z_{\e,N}^\infty=W_\infty(\nu_{\e,N},\tilde\nu_\lambda)\vee W_\infty(\nu_{\e,N},\tilde\nu_0)$: proof that
\begin{equation}\label{eq:control-Zinfty}
\tfrac{d^+}{dt} Z_{\e,N}^\infty\,\lesssim_{h,\mu^\circ}\,Z_{\e,N}^\infty+\lambda(\alpha_{\e,N}^0+1)+\e,
\end{equation}
where $\tfrac{d^+}{dt}$ stands for the right-derivative.

\medskip\noindent
We focus on the bound on $W_\infty(\nu_{\e,N},\tilde\nu_0)$, while the argument is analogous for~$W_\infty(\mu_{\e,N},\tilde\mu_\lambda)$.
In view of the particle dynamics~\eqref{eq:dynamics} and of the mean-field transport equation~\eqref{eq:tsp-munu}, we can estimate the (right) time-derivative of the $\infty$-Wasserstein distance~\eqref{eq:def-tildeXR} by using characteristics, as e.g.\@ in~\cite{Carrillo-Choi-Hauray-14}, to the effect of
\begin{multline*}
\tfrac{d^+}{dt} W_\infty(\nu_{\e,N},\tilde\nu_0)\,\le~~\supessd{x,r;\tilde\nu_0}\Big(|(\Gc\ast h)(x)-V_{\e,N}(\tilde X_{\e,N;0}(x,r))|\\
+\big|\tilde\Omega(x,r)r-\Omega_{\e,N}(\tilde X_{\e,N;0}(x,r))\tilde R_{\e,N;0}(x,r)\big|\Big),
\end{multline*}
where we have defined $V_{\e,N},\Omega_{\e,N}$ by setting $V_{\e,N}(X_{\e,N}^n)=V_{\e,N}^n$ and $\Omega_{\e,N}(X_{\e,N}^n)=\Omega_{\e,N}^n$ for all~$1\le n\le N$.
Inserting the first-order expansions of particle translational and angular velocities stated, cf.\@ Proposition~\ref{prop:dilute-exp}, as well as the definition~\eqref{eq:def-Vlim} of the limiting angular velocity, we deduce
\begin{multline*}
\tfrac{d^+}{dt} W_\infty(\nu_{\e,N},\tilde\nu_0)\,\lesssim_h~~
\supessd{x,r;\tilde\nu_0}\Big(\big|(\Gc\ast h)(x)-(\Gc\ast h)(\tilde X_{\e,N;0}(x,r))\big|\\
+\big|\big(\Omega^\circ(r)\nabla(\Gc\ast h)(x)\big) r-\big(\Omega^\circ(\tilde R_{\e,N;0}(x,r))\nabla(\Gc\ast h)(\tilde X_{\e,N;0}(x,r))\big)\tilde R_{\e,N;0}(x,r)\big|\Big)\\
+\lambda(\alpha_{\e,N}^0+1)+\e.
\end{multline*}
Using the Lipschitz continuity of $\Gc\ast h$, $\nabla(\Gc\ast h)$, and $\Omega^\circ$, we deduce
\[\tfrac{d^+}{dt}W_\infty(\nu_{\e,N},\tilde\nu_0)\,\lesssim_h\,W_\infty(\nu_{\e,N},\tilde\nu_0)+\lambda(\alpha_{\e,N}^0+1)+\e.\]
Arguing similarly for $W_{\infty}(\nu_{\e,N},\tilde\nu_\lambda)$ (and further using the a priori estimate~\eqref{eq:estim-tildemu}), the claim~\eqref{eq:control-Zinfty} follows.

\medskip
\step2 Second-order control on $W_{\e,N}^\infty=W_\infty(\mu_{\e,N},\tilde\mu_\lambda)$: proof that
\begin{multline}\label{eq:control-Winfty-2}
\tfrac{d^+}{dt}W_{\e,N}^\infty\,\lesssim_{h,\mu^\circ}\,W_{\e,N}^\infty
+(\lambda\alpha_{\e,N}^1+\e)\big(\lambda(\alpha_{\e,N}^0+1)+\e\big)+\kappa_0\e\\
+\lambda Z_{\e,N}^\infty\Big(\log\big(2+\tfrac1{Z_{\e,N}^\infty}\big)+\alpha_{\e,N}^1
+(N^\frac1d\dd_{\e,N}^{\min})^{1-d}\Big).
\end{multline}
Comparing the transport equation~\eqref{eq:equation-tildemu} with the particle dynamics~\eqref{eq:dynamics}, we can again estimate the time-derivative of the Wasserstein distance~\eqref{eq:def-hatX} by means of characteristics,
\begin{equation*}
\tfrac{d^+}{dt}W_{\e,N}^\infty\,\le~~\supessd{x;\tilde\mu_\lambda}|\tilde u_\lambda(x)-V_{\e,N}(\hat X_{\e,N;\lambda}(x))|.
\end{equation*}
Inserting the second-order expansion of particle velocities stated in Proposition~\ref{prop:dilute-exp}, as well as the definition~\eqref{eq:def-Vlim} of limiting velocities, we deduce
\begin{multline}\label{eq:estim-W-0}
\tfrac{d^+}{dt}W_{\e,N}^\infty\,\lesssim_h~~\supessd{x;\tilde\mu_\lambda}\big(T_1(x)+\lambda T_2(x)+\kappa_0\lambda T_3(x)+\lambda T_4(x)\big)\\
+(\lambda\alpha_{\e,N}^1+\e)\big(\lambda(\alpha_{\e,N}^0+1)+\e\big)+\kappa_0\e,
\end{multline}
in terms of
\begingroup\allowdisplaybreaks
\begin{eqnarray*}
T_1(x)&:=&\big|(\Gc\ast h)(x)-(\Gc\ast h)(\hat X_{\e,N;\lambda}(x))\big|,\\
T_2(x)&:=&\Big|\big[\Gc\ast(\tilde\mu_0 (e-h))\big](x)-\big[\Gc\,\hat\ast\, (\mu_{\e,N}(e-h))\big](\hat X_{\e,N;\lambda}(x))\Big|,\\
T_3(x)&:=&\Big|\big[\nabla\Gc\ast(\langle\Sigma_f^\circ\tilde\nu_0\rangle)\big](x)-\big[\nabla\Gc\,\hat\ast\,(\langle\Sigma_f^\circ\nu_{\e,N}\rangle)\big](\hat X_{\e,N;\lambda}(x))\Big|,\\
T_4(x)&:=&\Big|\big[\nabla\Gc\ast(\langle\Sigma^\circ\tilde\nu_0\rangle\D(\Gc\ast h))\big](x)-\big[\nabla\Gc\,\hat\ast\, (\langle\Sigma^\circ\nu_{\e,N}\rangle\D(\Gc\ast h))\big](\hat X_{\e,N;\lambda}(x))\Big|.
\end{eqnarray*}
\endgroup
We analyze these different terms separately. First, as in Step~1, using the Lipschitz continuity of~$\Gc\ast h$, we find
\[\sup T_1\,\lesssim_h\,W_{\e,N}^\infty.\]
We turn to $T_2$. Recalling $(\hat X_{\e,N;0})_*\tilde\mu_0=\mu_{\e,N}$, this term can be bounded by
\begin{multline*}
T_2(x)\,\le\,
\int_{\R^d}\Big|\Gc(x-y)\,(e-h(y))\\
-(\mathds1_{\ne}\Gc)(\hat X_{\e,N;\lambda}(x)-\hat X_{\e,N;0}(y))\,(e-h(\hat X_{\e,N;0}(y)))\Big|\,\tilde\mu_0(y)\,dy,
\end{multline*}
where we use the short-hand notation $(\mathds1_{\ne}\Gc)(x-y):=\Gc(x-y)\,\mathds1_{x\ne y}$.
Using the Lipschitz continuity of $h$ and the pointwise decay of the Stokeslet, we deduce
\begin{equation*}
T_2(x)\,\lesssim_h\,Z_{\e,N}^\infty\|\tilde\mu_0\|_{\Ld^1\cap\Ld^\infty(\R^d)}
+\int_{\R^d}\Big|\Gc(x-y)-(\mathds1_{\ne}\Gc)(\hat X_{\e,N;\lambda}(x)-\hat X_{\e,N;0}(y))\Big|\,\tilde\mu_0(y)\,dy.
\end{equation*}
Further using the pointwise decay of the Stokeslet in form of
\begin{equation}\label{eq:pointwise-diff-G}
|\Gc(a)-\Gc(a')|\,\lesssim\,|a-a'|\big(|a|^{1-d}+|a'|^{1-d}\big),
\end{equation}
this becomes
\begin{multline*}
T_2(x)\,\lesssim_h\,Z_{\e,N}^\infty\|\tilde\mu_0\|_{\Ld^1\cap\Ld^\infty(\R^d)}\\
+Z_{\e,N}^\infty\int_{\R^d}\Big(|x-y|^{1-d}+|\hat X_{\e,N;\lambda}(x)-\hat X_{\e,N;0}(y)|^{1-d}\,\mathds1_{\hat X_{\e,N;\lambda}(x)\ne\hat X_{\e,N;0}(y)}\Big)\,
\tilde\mu_0(y)\,dy,
\end{multline*}
and thus, recalling $(\hat X_{\e,N;0})_*\tilde\mu_0=\mu_{\e,N}$,
\begin{eqnarray*}
T_2(x)
&\lesssim_h&Z_{\e,N}^\infty\bigg(\|\tilde\mu_0\|_{\Ld^1\cap\Ld^\infty(\R^d)}+\tfrac1N\sum_{m:X_{\e,N}^m\ne\hat X_{\e,N;\lambda}(x)}|\hat X_{\e,N;\lambda}(x)-X_{\e,N}^m|^{1-d}\bigg)\\
&\le&Z_{\e,N}^\infty\big(\alpha_{\e,N}^1+\|\tilde\mu_0\|_{\Ld^1\cap\Ld^\infty(\R^d)}\big).
\end{eqnarray*}
We turn to the third term $T_3$ in~\eqref{eq:estim-W-0}.
Using the optimal transport map~$(\tilde X_{\e,N;0},\tilde R_{\e,N;0})$ given by~\eqref{eq:def-tildeXR}, with $(\tilde X_{\e,N;0},\tilde R_{\e,N;0})_*\tilde\nu_0=\nu_{\e,N}$, and recalling that $\Sigma_f^\circ$ is smooth (it is actually quadratic, cf.~\eqref{eq:def-ugamma-sym}), we can estimate
\begin{eqnarray*}
T_3(x)&\le&\iint_{\R^d\times\Sp^{d-1}}\Big|\nabla\Gc(x-y)\Sigma_f^\circ(r)\\
&&\hspace{1cm}-(\mathds1_{\ne}\nabla\Gc)\big(\hat X_{\e,N;\lambda}(x)-\tilde X_{\e,N;0}(y,r)\big)\Sigma_f^\circ(\tilde R_{\e,N;0}(y,r))\Big|\,\tilde\nu_0(y,r)\,dyd\sigma(r)\\
&\lesssim&Z_{\e,N}^\infty\|\tilde\mu_0\|_{\Ld^1\cap\Ld^\infty(\R^d)}\\
&&+\iint_{\R^d\times\Sp^{d-1}}\Big|\nabla\Gc(x-y)-(\mathds1_{\ne}\nabla\Gc)\big(\hat X_{\e,N;\lambda}(x)-\tilde X_{\e,N;0}(y,r)\big)\Big|\,\tilde\nu_0(y,r)\,dyd\sigma(r).
\end{eqnarray*}
We turn to the analysis of this last integral and we note that the argument for~$T_2$ cannot be repeated as $\nabla^2\Gc$ is not locally integrable.
Instead, we start by splitting the integral into two parts, distinguishing between the contributions of \mbox{$|x-y|\ge4Z_{\e,N}^\infty$} and~\mbox{$|x-y|<4Z_{\e,N}^\infty$}.
On the one hand, using the pointwise decay of the Stokeslet in form of
\begin{equation}\label{eq:decay-St-nabG-G}
|\nabla\Gc(a)-\nabla\Gc(a')|\,\lesssim\,|a-a'|\big(|a|^{-d}+|a'|^{-d}\big),
\end{equation}
and noting that the condition $|x-y|\ge4Z_{\e,N}^\infty$ implies
\[|\hat X_{\e,N;\lambda}(x)-\tilde X_{\e,N;0}(y,r)|\,\ge\,|x-y|-2Z_{\e,N}^\infty\,\ge\,\tfrac12|x-y|,\]
we find
\begin{align*}
&{\iint_{\{(y,r):|x-y|\ge4Z_{\e,N}^\infty\}}\Big|\nabla\Gc(x-y)-(\mathds1_{\ne}\nabla\Gc)(\hat X_{\e,N;\lambda}(x)-\tilde X_{\e,N;0}(y,r))\Big|\,\tilde\nu_0(y,r)\,dyd\sigma(r)}\\
&\quad\lesssim~Z_{\e,N}^\infty\int_{\{y:|x-y|\ge4Z_{\e,N}^\infty\}}|x-y|^{-d}\,\tilde\mu_0(y)\,dy\\
&\quad\lesssim~Z_{\e,N}^\infty\log\big(2+\tfrac1{Z_{\e,N}^\infty}\big)\|\tilde\mu_0\|_{\Ld^1\cap\Ld^\infty(\R^d)}.
\end{align*}
On the other hand, for the other contribution, using the pointwise decay of the Stokeslet and recalling $(\tilde X_{\e,N;0},\tilde R_{\e,N,0})_*\tilde\nu_0=\nu_{\e,N}$, we can bound
\begin{multline*}
{\iint_{\{(y,r):|x-y|\le4Z_{\e,N}^\infty\}}\Big|\nabla\Gc(x-y)-(\mathds1_{\ne}\nabla\Gc)(\hat X_{\e,N;\lambda}(x)-\tilde X_{\e,N;0}(y,r))\Big|\,\tilde\nu_0(y,r)\,dyd\sigma(r)}\\
\,\lesssim\,\int_{\{y:|x-y|\le4Z_{\e,N}^\infty\}}|x-y|^{1-d}\,\tilde\mu_0(y)\,dy\\
+\tfrac1N\sum_{m:X_{\e,N}^m\ne \hat X_{\e,N;\lambda}(x)}|\hat X_{\e,N;\lambda}(x)-X_{\e,N}^m|^{1-d}\,\mathds1_{|\hat X_{\e,N;\lambda}(x)-X_{\e,N}^m|\le6Z_{\e,N}^\infty},
\end{multline*}
and thus, by Lemma~\ref{lem:Richard},
\begin{multline*}
{\iint_{\{(y,r):|x-y|\le4Z_{\e,N}^\infty\}}\Big|\nabla\Gc(x-y)-(\mathds1_{\ne}\nabla\Gc)(\hat X_{\e,N;\lambda}(x)-\tilde X_{\e,N;0}(y,r))\Big|\,\tilde\nu_0(y,r)\,dyd\sigma(r)}\\
\,\lesssim\,Z_{\e,N}^\infty\big(1+(N^\frac1dd_{\e,N}^{\min})^{1 -d}\big)\|\tilde\mu_0\|_{\Ld^1\cap\Ld^\infty(\R^d)}.
\end{multline*}
Combining these two estimates, we deduce
\begin{equation*}
\sup T_3\,\lesssim\,Z_{\e,N}^\infty\Big(\log\big(2+\tfrac1{Z_{\e,N}^\infty}\big)
+(N^\frac1d\dd_{\e,N}^{\min})^{1-d}\Big)\|\tilde\mu_0\|_{\Ld^1\cap\Ld^\infty(\R^d)}.
\end{equation*}
Noting that the same bound holds for $T_4$,
and recalling the a priori estimate~\eqref{eq:estim-tildemu},
the claim~\eqref{eq:control-Winfty-2} follows.

\medskip
\step3 Approximation of the fluid velocity: proof that for any $\delta\in[\e,1]$,
\begin{multline}\label{eq:estim-udiff-err}
\|u_{\e,N}-\tilde u_\lambda\|_{\Ld^\infty(\R^d\setminus(\Ic_{\e,N}+\delta B))}
\,\lesssim_{h,\mu^\circ}\,\lambda(\lambda\alpha_{\e,N}^1+\e)(\alpha_{\e,N}^0+1)\\
+\lambda Z_{\e,N}^\infty\Big(\log\big(2+\tfrac1{Z_{\e,N}^\infty}\big)+\alpha_{\e,N}^1+(N^\frac1d\dd_{\e,N}^{\min})^{1-d}\Big)
+(\tfrac\e\delta)^{d}\delta(1+ Z_{\e,N}^\infty).
\end{multline}
We start with the expansion of the fluid velocity away from the particles as stated in Proposition~\ref{prop:dilute-exp}: for any $\delta\in[\e,1]$, we have
\begin{multline*}
\Big\|u_{\e,N}-\Gc\ast \big((1-\lambda\mu_{\e,N})h+\lambda\mu_{\e,N}e\big)\\
-\lambda\nabla\Gc\ast\big(2\langle\Sigma^\circ\nu_{\e,N}\rangle\!\D(\Gc\ast h)
+\kappa_0\langle\Sigma^\circ_f\nu_{\e,N}\rangle\big)\Big\|_{\Ld^\infty(\R^d\setminus(\Ic_{\e,N}+\delta B))}\\
\,\lesssim_h\,\lambda(\lambda\alpha_{\e,N}^1+\e)(\alpha_{\e,N}^0+1)+\e(\tfrac\e\delta)^d.
\end{multline*}
Replacing $\mu_{\e,N},\nu_{\e,N}$ by $\tilde\mu_0,\tilde\nu_0$ and recognizing the definition of $\tilde u_\lambda$, this yields
\begin{equation*}
\|u_{\e,N}-\tilde u_\lambda\|_{\Ld^\infty(\R^d\setminus(\Ic_{\e,N}+\delta B))}
\,\lesssim_h\,\lambda(S_1+S_2+\kappa_0S_3)
+\lambda(\lambda\alpha_{\e,N}^1+\e)(\alpha_{\e,N}^0+1)+\e(\tfrac\e\delta)^d,
\end{equation*}
in terms of
\begin{eqnarray*}
S_1&:=&\big\|\Gc\ast \big((\mu_{\e,N}-\tilde\mu_0)(e-h)\big)\big\|_{\Ld^\infty(\R^d\setminus(\Ic_{\e,N}+\delta B))},\\
S_2&:=&\big\|\nabla\Gc\ast\big(\langle\Sigma^\circ(\nu_{\e,N}-\tilde\nu_0)\rangle\!\D(\Gc\ast h)\big)\big\|_{\Ld^\infty(\R^d\setminus(\Ic_{\e,N}+\delta B))},\\
S_3&:=&\big\|\nabla\Gc\ast\langle\Sigma^\circ_f(\nu_{\e,N}-\nu_0)\rangle\big\|_{\Ld^\infty(\R^d\setminus(\Ic_{\e,N}+\delta B))}.
\end{eqnarray*}
We separately analyze these three contributions.
Using the optimal transport map~$\hat X_{\e,N;0}$ given by~\eqref{eq:def-hatX00}, with $(\hat X_{\e,N;0})_*\tilde\mu_0=\mu_{\e,N}$, we can write
\begin{equation*}
S_1
\,\le\,\sup_{x\in\R^d\setminus(\Ic_{\e,N}+\delta B)}\int_{\R^d}\Big|\Gc(x-y)(e-h(y))-\Gc(x-\hat X_{\e,N;0}(y))(e-h(\hat X_{\e,N;0}(y)))\Big|\,\tilde\mu_0(y)\,dy.
\end{equation*}
Using the Lipschitz continuity of $h$ and the pointwise decay of the Stokeslet, in particular in form of~\eqref{eq:pointwise-diff-G}, we deduce
\begin{eqnarray*}
S_1
&\lesssim_h&Z_{\e,N}^\infty\|\tilde\mu_0\|_{\Ld^1\cap\Ld^\infty(\R^d)}+\sup_{x\in\R^d\setminus(\Ic_{\e,N}+\delta B)}\int_{\R^d}\big|\Gc(x-y)-\Gc(x-\hat X_{\e,N;0}(y))\big|\,\tilde\mu_0(y)\,dy\\
&\lesssim&Z_{\e,N}^\infty\bigg(\|\tilde\mu_0\|_{\Ld^1\cap\Ld^\infty(\R^d)}+\sup_{x\in\R^d\setminus(\Ic_{\e,N}+\delta B)}\int_{\R^d}|x-\hat X_{\e,N}^0(y)|^{1-d}\,\tilde\mu_0(y)\,dy\bigg).
\end{eqnarray*}
Recalling $(\hat X_{\e,N}^0)_*\tilde\mu_0=\mu_{\e,N}$ and estimating the sum as in~\eqref{eq:bnd-x-alpha-reint}, we obtain
\begin{equation*}
S_1\,\lesssim_h\,Z_{\e,N}^\infty\Big(\|\tilde\mu_0\|_{\Ld^1\cap\Ld^\infty(\R^d)}+\alpha_{\e,N}^1+\tfrac1\lambda\delta(\tfrac\e\delta)^{d}\Big).
\end{equation*}
We turn to $S_2$. Using the optimal transport map~$(\tilde X_{\e,N;0},\tilde R_{\e,N;0})$ given by~\eqref{eq:def-tildeXR}, we can write
\begin{multline*}
S_2\,\le\,\sup_{x\in\R^d\setminus(\Ic_{\e,N}+\delta B)}\iint_{\R^d\times\Sp^{d-1}}\Big|\nabla\Gc(x-y)\Sigma_0(r)\D(\Gc\ast h)(y)\\
-\nabla\Gc(x-\tilde X_{\e,N}^0(y))\Sigma_0(\tilde R_{\e,N}^0(r))\D(\Gc\ast h)(\tilde X_{\e,N}^0(y))\Big|\,\tilde\nu_0(y,r)\,dyd\sigma(r).
\end{multline*}
Using the Lipschitz continuity of $\Sigma_0$ and $\nabla\Gc\ast h$, and using the pointwise decay of the Stokeslet, we find
\begin{multline}\label{eq:estim-S2-presplit}
S_2\,\lesssim_h\,Z_{\e,N}^\infty\|\tilde\mu_0\|_{\Ld^1\cap\Ld^\infty(\R^d)}\\
+\sup_{x\in\R^d\setminus(\Ic_{\e,N}+\delta B)}\iint_{\R^d\times\Sp^{d-1}}\Big|\nabla\Gc(x-y)
-\nabla\Gc(x-\tilde X_{\e,N;0}(y))\Big|\,\tilde\nu_0(y,r)\,dyd\sigma(r).
\end{multline}
To estimate this last integral, we split it into two parts, distinguishing between the contributions of $|x-y|\ge2Z_{\e,N}^\infty$ and $|x-y|<2Z_{\e,N}^\infty$. On the one hand, using the pointwise decay of the Stokeslet in form of~\eqref{eq:decay-St-nabG-G}, and noting that the condition $|x-y|\ge2Z_{\e,N}^\infty$ implies $|x-\tilde X_{\e,N;0}(y,r)|\ge\frac12|x-y|$, we get
\begin{align*}
&\sup_{x\in\R^d\setminus(\Ic_{\e,N}+\delta B)}\iint_{\{(y,r):|x-y|\ge2Z_{\e,N}^\infty\}}\Big|\nabla\Gc(x-y)
-\nabla\Gc(x-\tilde X_{\e,N;0}(y))\Big|\,\tilde\nu_0(y,r)\,dyd\sigma(r)\\
&\hspace{2cm}\lesssim~Z_{\e,N}^\infty\sup_{x\in\R^d}\int_{\{y:|x-y|\ge2Z_{\e,N}^\infty\}}|x-y|^{-d}\,\tilde\mu_0(y)\,dy\\
&\hspace{2cm}\lesssim~Z_{\e,N}^\infty\log\big(2+\tfrac1{Z_{\e,N}^\infty}\big)\|\tilde\mu_0\|_{\Ld^1\cap\Ld^\infty(\R^d)}.
\end{align*}
On the other hand, recalling $(\tilde X_{\e,N;0},\tilde R_{\e,N;0})_*\tilde\nu_0=\nu_{\e,N}$ and using the pointwise decay of the Stokeslet, we can bound
\begin{eqnarray*}
\lefteqn{\sup_{x\in\R^d\setminus(\Ic_{\e,N}+\delta B)}\iint_{\{(y,r):|x-y|<2Z_{\e,N}^\infty\}}\Big|\nabla\Gc(x-y)
-\nabla\Gc(x-\tilde X_{\e,N;0}(y))\Big|\,\tilde\nu_0(y,r)\,dyd\sigma(r)}\\
&\lesssim&\sup_{x\in\R^d\setminus(\Ic_{\e,N}+\delta B)}\iint_{\{(y,r):|x-y|<2Z_{\e,N}^\infty\}}\Big(|x-y|^{1-d}+|x-\tilde X_{\e,N;0}(y,r)|^{1-d}\Big)\,\tilde\nu_0(y,r)\,dyd\sigma(r)\\
&\lesssim&Z_{\e,N}^\infty\|\tilde\mu_0\|_{\Ld^\infty(\R^d)}+\sup_{x\in\R^d\setminus(\Ic_{\e,N}+\delta B)}\tfrac1N\sum_{n=1}^N|x- X_{\e,N}^n|^{1-d}\,\mathds1_{|x-X_{\e,N}^n|\le 3Z_{\e,N}^\infty},
\end{eqnarray*}
and thus, estimating the last sum as in~\eqref{eq:bnd-x-alpha-reint}, and appealing to Lemma~\ref{lem:Richard},
\begin{multline*}
\sup_{x\in\R^d\setminus(\Ic_{\e,N}+\delta B)}\iint_{\{(y,r):|x-y|<2Z_{\e,N}^\infty\}}\Big|\nabla\Gc(x-y)
-\nabla\Gc(x-\tilde X_{\e,N;0}(y))\Big|\,\tilde\nu_0(y,r)\,dyd\sigma(r)\\
\,\lesssim\,Z_{\e,N}^\infty\big(1+(N^\frac1d\dd_{\e,N}^{\min})^{1-d}\big)\|\tilde\mu_0\|_{\Ld^1\cap\Ld^\infty(\R^d)}+\tfrac1\lambda\delta(\tfrac\e\delta)^{d}.
\end{multline*}
Inserting these estimates into~\eqref{eq:estim-S2-presplit}, we are led to
\begin{equation*}
S_2\,\lesssim_h\,Z_{\e,N}^\infty\Big(\log\big(2+\tfrac1{Z_{\e,N}^\infty}\big)+(N^\frac1d\dd_{\e,N}^{\min})^{1-d}\Big)\|\tilde\mu_0\|_{\Ld^1\cap\Ld^\infty(\R^d)}
+\tfrac1\lambda\delta(\tfrac\e\delta)^{d}.
\end{equation*}
Noting that the same bound holds for $S_3$, and recalling the a priori estimate~\eqref{eq:estim-tildemu}, the claim~\eqref{eq:estim-udiff-err} follows.

\medskip
\step4 Conclusion.\\
Recall the well-preparedness assumption~\eqref{eq:cond-wellsep-prop52} for some $\theta_0\ge1$.
Given $\theta>\theta_0$, consider the maximal time
\begin{equation}\label{eq:max-time-TepsN}
T_{\e,N}^\theta\,:=\,\sup\Big\{t\in[0,T]~:~\dd_{\e,N}^{\min}(s)\ge(\theta N)^{-\frac1d},~Z_{\e,N}^{\infty}(s)\le3\theta~~\text{for all $0\le s\le t$}\Big\}.
\end{equation}
By Lemma~\ref{lem:Richard}, comparing $\mu_{\e,N}$ to $\tilde\mu_\lambda$, and recalling the a priori estimate~\eqref{eq:estim-tildemu},
we find for all~$t\in[0,T_{\e,N}^\theta]$,
\begin{eqnarray*}
\alpha_{\e,N}^0(t)&\lesssim&\theta \log(\theta+N)\|\mu^\circ\|_{\Ld^1\cap\Ld^\infty(\R^d)},\\
\alpha_{\e,N}^1(t)&\lesssim&\theta^{2-\frac1d}\|\mu^\circ\|_{\Ld^1\cap\Ld^\infty(\R^d)}.
\end{eqnarray*}
Provided that $\lambda\theta\log(\theta+N)\|\mu^\circ\|_{\Ld^1\cap\Ld^\infty(\R^d)}\ll1$ is small enough, we deduce
\[\sup_{t\in[0,T_{\e,N}^\theta]}\lambda\alpha_{\e,N}^0(t)\,\ll\,1.\]
Further assuming that $(\theta N)^{-\frac1d}\ge4\e$, we may then
appeal to the Lipschitz bounds of Proposition~\ref{prop:Lip-est} and deduce for all $t\in[0,T_{\e,N}^\theta]$,
\begin{equation*}
\tfrac{d}{dt}\dd_{\e,N}^{\min}(t)\,\ge\,- C_h\dd_{\e,N}^{\min}(t),
\end{equation*}
hence, by the Gronwall inequality and by the well-preparedness assumption~\eqref{eq:cond-wellsep-prop52},
\begin{equation}\label{eq:diff-dmin}
\dd_{\e,N}^{\min}(t)
\,\ge\,e^{-tC_h}(\theta_0N)^{-\frac1d}.
\end{equation}
The result~\eqref{eq:control-Zinfty} of Step~1 similarly yields for all $t\in[0,T_{\e,N}^\theta]$,
\begin{equation*}
Z_{\e,N}^\infty(t)\,\le\, e^{tC_{h,\mu^\circ}}Z_{\e,N}^\infty(0)+e^{tC_{h,\mu^\circ}}\big(\lambda\theta\log(\theta+ N)+\e\big),
\end{equation*}
which proves~\eqref{eq:estim-fin-Z}.
These last two estimates imply in particular that the maximal time~$T_{\e,N}^\theta$ satisfies
\[T_{\e,N}^\theta\,\ge\,\min\big\{C_{h,\mu^\circ}^{-1}\log({\theta}/{\theta_0})\,,\,T\big\}.\]
Next, using~\eqref{eq:estim-fin-Z} to control the last~$O(\lambda)$ term in the result~\eqref{eq:control-Winfty-2} of Step~2, and using again the above estimates on $\alpha_{\e,N}^0,\alpha_{\e,N}^1$,
we get for all~$t\in[0,T_{\e,N}^\theta]$,
\begin{equation*}
\tfrac{d^+}{dt}W_{\e,N}^\infty\,\lesssim_{t,h,\theta,\mu^\circ}\,W_{\e,N}^\infty
+\lambda|\!\log\lambda|Z_{\e,N}^\infty(0)
+(\lambda|\!\log\lambda|+\e)(\lambda\log N+\e)+\kappa_0\e,
\end{equation*}
and the conclusion~\eqref{eq:estim-fin-W} follows
from the Gronwall inequality.
Finally, combined with~\eqref{eq:estim-fin-Z} and with the above estimates on $\alpha_{\e,N}^0,\alpha_{\e,N}^1$, the result~\eqref{eq:estim-udiff-err} of Step~3 takes on the following guise, for any $\delta\in[\e,1]$,
\begin{multline*}
\|u_{\e,N}-\tilde u_\lambda\|_{\Ld^\infty(\R^d\setminus(\Ic_{\e,N}+\delta B))}
\,\lesssim_{t,h,\theta,\mu^\circ}\,
\lambda|\!\log\lambda|Z_{\e,N}^\infty(0)\\
+(\lambda|\!\log\lambda|+\e)(\lambda\log N+\e)
+(\tfrac\e\delta)^d(\lambda+ \delta),
\end{multline*}
which proves~\eqref{eq:estim-fin-u}.
\end{proof}

\medskip
\section{Conclusion: proof of main results}\label{sec:conclusion}
It remains to ensure the well-posedness of the macroscopic model~\eqref{eq:Doi0} as stated in Proposition~\ref{prop:well-posed-Doi}, and to conclude the proof of Theorem~\ref{th:main}.

\subsection{Proof of Proposition~\ref{prop:well-posed-Doi}}\label{sec:well-posed-Doi}
Let some non-integer $\gamma>1$ be fixed.
We proceed by an iteration argument. For $n=0$, let $\nu_0:=\nu^\circ$ and $u_0:=0$. Next, for all $n\ge0$, we iteratively define $\nu_{n+1}$ as the solution of the linear transport equation
\begin{equation}\label{eq:lin-system-mf/1}
\left\{\begin{array}{l}
\partial_t\nu_{n+1}+\Div_x(\nu_{n+1}u_n)+\Div_r(\nu_{n+1}(\Omega^\circ\nabla u_n)r)=0,\\
\nu_{n+1}|_{t=0}=\nu^\circ,
\end{array}\right.
\end{equation}
and $\nabla u_{n+1}$ as the solution of the linear Stokes equation
\begin{equation*}
\left\{\begin{array}{l}
-\triangle u_{n+1}+\nabla p_{n+1}=(1-\lambda\mu_n)h+\lambda\mu_ne+\lambda\Div\big[2\langle\Sigma^\circ\nu_n\rangle\!\D(u_n)+\kappa_0\langle\Sigma_f^\circ\nu_n\rangle\big],\\
\Div(u_{n+1})=0,
\end{array}\right.
\end{equation*}
which means, in terms of the Stokeslet $\Gc$,
\begin{equation}\label{eq:lin-system-mf/2}
u_{n+1}\,=\,\Gc\ast\big((1-\lambda\mu_n)h+\lambda \mu_ne\big)+\lambda\nabla\Gc\ast\big(2\langle\Sigma^\circ\nu_n\rangle\D(u_n)+\kappa_0\langle\Sigma_f^\circ\nu_n\rangle\big).
\end{equation}
For all $n\ge0$, if we have $u_n\in\Ld^\infty_\loc(\R^+;W^{\gamma+1,\infty}(\R^d)^{d})$, recalling that $\Omega^\circ$ is smooth and that initially $\nu^\circ\in \Pc\cap W^{1,1}\cap W^{\gamma,\infty}(\R^d\times\Sp^{d-1})$, the standard theory of transport equations in Hölder spaces
(e.g.~\cite[Theorem~3.14]{BCD-11})
ensures that~\eqref{eq:lin-system-mf/1} admits a unique weak solution $\nu_{n+1}\in\Ld^\infty_\loc(\R^+;\Pc\cap W^{1,1}\cap W^{\gamma,\infty}(\R^d\times\Sp^{d-1}))$ with
\begin{multline}\label{eq:apriori-nu-n}
\|\nu_{n+1}^t\|_{W^{1,1}\cap W^{\gamma,\infty}(\R^d\times\Sp^{d-1})}\\
\,\le\,\|\nu^\circ\|_{W^{1,1}\cap W^{\gamma,\infty}(\R^d\times \Sp^{d-1})}\exp\Big(C_\gamma t+C_\gamma\int_0^t\|u_n\|_{W^{\gamma+1,\infty}(\R^d)}\Big).
\end{multline}
Moreover, if $\nu_n\in \Ld^\infty_\loc(\R^+;\Pc\cap W^{\gamma,\infty}(\R^d\times\Sp^{d-1}))$ and $u_n\in \Ld^\infty_\loc(\R^+;W^{\gamma+1,\infty}(\R^d)^{ d})$, recalling that $\Sigma^\circ,\Sigma^\circ_f$ are smooth and using the standard theory for the Stokes equation in Hölder spaces, we find that equation~\eqref{eq:lin-system-mf/2} yields $u_{n+1}\in \Ld^\infty_\loc(\R^+;W^{\gamma+1,\infty}(\R^d)^d)$ with
\begin{equation}\label{eq:apriori-u-n}
\|u_{n+1}\|_{W^{\gamma+1,\infty}(\R^d)}\,\lesssim_{h,\gamma}\,1+\lambda \|\nu_n\|_{\Ld^1\cap W^{\gamma,\infty}(\R^d\times\Sp^{d-1})}\big(1+\|u_n\|_{W^{\gamma+1,\infty}(\R^d)}\big).
\end{equation}
By induction, this proves that we can indeed construct unique global weak solutions for the scheme~\eqref{eq:lin-system-mf/1}--\eqref{eq:lin-system-mf/2} with $\nu_{n}\in \Ld^\infty_\loc(\R^+;\Pc\cap W^{1,1}\cap W^{\gamma,\infty}(\R^d\times\Sp^{d-1}))$ and $u_n\in \Ld^\infty_\loc(\R^+;W^{\gamma+1,\infty}(\R^d)^{d})$  for all \mbox{$n\ge0$}.
From the above a priori estimates~\eqref{eq:apriori-nu-n}--\eqref{eq:apriori-u-n}, absorbing the nonlinearity for small~$\lambda$, we conclude the following: given $T>0$, provided that~$\lambda\ll_{T,h,\gamma,\nu^\circ}1$ is small enough, we have for all $n\ge0$ and $t\in[0,T]$,
\begin{eqnarray}
\|\nu_n^t\|_{W^{1,1}\cap W^{\gamma,\infty}(\R^d\times\Sp^{d-1})}&\le&e^{C_{h,\gamma}t}\|\nu^\circ\|_{W^{1,1}\cap W^{\gamma,\infty}(\R^d\times \Sp^{d-1})},\label{eq:estim-nuu-n}\\
\|u_{n}^t\|_{W^{\gamma+1,\infty}(\R^d)}&\le&C_{h,\gamma}.\nonumber
\end{eqnarray}
Further appealing to the Aubin lemma, we may then extract a subsequence of $(\nu_n, u_n)_n$ that converges strongly to some limit $(\nu_\lambda, u_\lambda)$ in $C([0,T];\Ld^1(\R^d\times\Sp^{d-1}))\times C([0,T];W^{1,\infty}(\R^d)^{d})$. Passing to the limit in the iterative scheme~\eqref{eq:lin-system-mf/1}--\eqref{eq:lin-system-mf/2}, we find that the limit $(\nu_\lambda, u_\lambda)$ precisely satisfies the Vlasov--Stokes system~\eqref{eq:Doi0}. In addition, we may also pass to the limit in the a priori estimates~\eqref{eq:estim-nuu-n}: we get for all $t\in[0,T]$,
\begin{eqnarray}
\|\nu_\lambda^t\|_{W^{1,1}\cap W^{\gamma,\infty}(\R^d\times\Sp^{d-1})}&\lesssim_{t,h,\gamma}&\|\nu^\circ\|_{W^{1,1}\cap W^{\gamma,\infty}(\R^d\times\Sp^{d-1})},\label{eq:apriori-nabuesgfv00}\\
\|u_\lambda^t\|_{W^{\gamma+1,\infty}(\R^d)}&\lesssim_{t,h,\gamma}&1.\nonumber
\end{eqnarray}

It remains to establish the uniqueness of the solution of~\eqref{eq:Doi0}.
Let $(\nu_\lambda, u_\lambda),(\nu_\lambda', u_\lambda')$ be two solutions in $\Ld^\infty_\loc([0,T];\Pc\cap W^{1,1}\cap W^{\gamma,\infty}(\R^d\times\Sp^{d-1}))\times \Ld^\infty_\loc([0,T];W^{\gamma+1,\infty}(\R^d)^{d})$
with the same initial data $\nu_\lambda|_{t=0}=\nu_\lambda'|_{t=0}=\nu^\circ$. The equation for the difference $\nu_\lambda-\nu_\lambda'$ can be written as
\begin{multline*}
\partial_t(\nu_\lambda-\nu_\lambda')+\Div_x\big((\nu_\lambda-\nu_\lambda')u_\lambda\big)+\Div_r\big((\nu_\lambda-\nu_\lambda')(\Omega^\circ\nabla u_\lambda)r\big)\\
\,=\,-\Div_x\big(\nu_\lambda'(u_\lambda-u_\lambda')\big)-\Div_r\big(\nu_\lambda'(\Omega^\circ\nabla(u_\lambda-u_\lambda'))r\big),
\end{multline*}
and the equation for $u_\lambda-u_\lambda'$ as
\begin{multline*}
u_\lambda-u_\lambda'\,=\,\lambda\Gc\ast\big((\mu_\lambda-\mu_\lambda')(e-h)\big)+\lambda\nabla\Gc\ast\big(2\langle\Sigma^\circ\nu_\lambda\rangle\D(u_\lambda-u_\lambda')\big)\\
+\lambda\nabla\Gc\ast\big(2\langle\Sigma^\circ(\nu_\lambda-\nu_\lambda')\rangle\D(u_\lambda')+\kappa_0\langle\Sigma_f^\circ(\nu_\lambda-\nu_\lambda')\rangle\big).
\end{multline*}
Similar a priori estimates as in~\eqref{eq:apriori-nu-n}--\eqref{eq:apriori-u-n} then yield for all $t\in[0, T]$,
\begin{align}\label{eq:apriori-u-diff}
&\|\nu_\lambda^t-\nu_\lambda^{\prime,t}\|_{\Ld^1\cap W^{\gamma-1,\infty}(\R^d\times\Sp^{d-1})}\,\lesssim_s\,\exp\Big(C_\gamma t+C_\gamma\int_0^t\|u_\lambda\|_{W^{\gamma,\infty}(\R^d)}\Big)\\
&\hspace{5cm}\times\int_0^t\|\nu_\lambda'\|_{W^{1,1}\cap W^{\gamma,\infty}(\R^d\times\Sp^{d-1})}\|u_\lambda-u_\lambda'\|_{W^{\gamma,\infty}(\R^d\times\Sp^{d-1})},\nonumber\\
&\|u_\lambda-u_\lambda'\|_{W^{\gamma,\infty}(\R^d)}\,\lesssim_{h,\gamma}\,\lambda\|\nu_\lambda\|_{\Ld^1\cap W^{\gamma-1,\infty}(\R^d\times\Sp^{d-1})}\|u_\lambda-u_\lambda'\|_{W^{\gamma,\infty}(\R^d\times\Sp^{d-1})}\nonumber\\
&\hspace{4cm}+\lambda\big(1+\|u_\lambda'\|_{W^{\gamma,\infty}(\R^d\times\Sp^{d-1})}\big)\|\nu_\lambda-\nu_\lambda'\|_{\Ld^1\cap W^{\gamma-1,\infty}(\R^d\times\Sp^{d-1})}.\nonumber
\end{align}
Using a priori bounds on $(\nu_\lambda,u_\lambda)$ and $(\nu_\lambda',u_\lambda')$ and recalling the choice $\lambda\ll_{T,h,\gamma,\nu^\circ}1$ to absorb the nonlinearity in the estimate for $u_\lambda-u_\lambda'$, the conclusion $(\nu_\lambda,u_\lambda)=(\nu_\lambda',u_\lambda')$ follows from the Gronwall inequality.
\qed

\subsection{Proof of Theorem~\ref{th:main}}\label{sec:concl-main}
This is obtained as a simple post-processing of Proposition~\ref{prop:MFL}.
More precisely, there are two remaining tasks:
\begin{enumerate}[---]
\item First, we need to show that the $W_\infty$-estimate of Proposition~\ref{prop:MFL} can be upgraded into a $W_p$-estimate, thus reducing well-preparedness requirements: as inspired by~\cite{Hauray-Jabin-15} (see also~\cite{Mecherbet-19,Hofer-Schubert-21}), this is done by first comparing the empirical measure $\mu_{\e,N}$ to the corresponding continuum solution of~\eqref{eq:tsp-munu} with blob initial condition, cf.~\eqref{eq:blobID} below.
\smallskip\item Second, we need to compare the mean-field transport equations~\eqref{eq:tsp-munu} to the Doi-type kinetic model~\eqref{eq:Doi0} up to $O(\lambda^2)$ errors.
\end{enumerate}
We split the proof into three steps.

\medskip
\step1 Approximation by blob solution.\\
Consider the following blob initial condition, which is a regularized version of the initial empirical measure $\mu_{\e,N}^\circ$,
\begin{equation}\label{eq:blobID}
\nu_N^\circ(x,r)\,:=\,\tfrac1N\sum_{n=1}^N\phi_N(x-X_{\e,N}^{n;\circ})\,\gamma_N(r,R_{\e,N}^{n;\circ})~~\in~\Pc(\R^d\times\Sp^{d-1}),
\end{equation}
where the kernels $\phi_N$ and $\gamma_N$ are smooth, nonnegative, and satisfy
\begin{gather*}
\int_{\R^d}\phi_N=1,\qquad \|\phi_N\|_{\Ld^\infty(\R^d)}\lesssim N,\qquad \phi_N(x)=0~~\text{for $|x|>N^{-\frac1d}$},\\
\int_{\Sp^{d-1}}\gamma_N(\cdot,R)=1,\qquad\gamma_N(r,R)=0~~\text{for $\dist(r,R)>N^{-\frac1d}$},
\end{gather*}
for all $R\in\Sp^{d-1}$.
Under the well-preparedness assumption~\eqref{eq:prepared-thmain}, the associated spatial density~$\mu_N^\circ:=\langle\nu_N^\circ\rangle$ satisfies
\[\|\mu_N^\circ\|_{\Ld^\infty(\R^d)}\,\le\,\sup_{x\in\R^d}\sharp\big\{n:|x-X_{\e,N}^{n;\circ}|\le N^{-\frac1d}\big\}\,\lesssim\,1,\]
and moreover, comparing to the initial empirical measures $\nu_{\e,N}^\circ,\mu_{\e,N}^\circ$, we get by definition
\begin{equation}\label{eq:blob-initWinfty}
W_\infty(\nu_{\e,N}^\circ,\nu^\circ_N)+W_\infty(\mu_{\e,N}^\circ,\mu^\circ_N)\,\lesssim\,N^{-\frac1d}.
\end{equation}
We then consider the solutions $\tilde\nu_{N,0},\tilde\nu_{N,\lambda}$ of the mean-field transport equations~\eqref{eq:tsp-munu} starting from $\nu_N^\circ$, that is,
\begin{equation}\label{eq:tsp-munu-N0}
\left\{\begin{array}{l}
\partial_t\tilde\nu_{N,0}+\Div_x(\tilde\nu_{N,0} (\Gc\ast h))+\Div_r(\tilde\nu_{N,0}\tilde\Omega r)=0,\\
\partial_t\tilde\nu_{N,\lambda}+\Div_x(\tilde\nu_{N,\lambda}\tilde u_{N,\lambda})+\Div_r(\tilde\nu_{N,\lambda}\tilde\Omega r)=0,\\
\tilde\nu_{N,0}|_{t=0}=
\tilde\nu_{N,\lambda}|_{t=0}=\nu^\circ_N,
\end{array}\right.
\end{equation}
where the translational and angular velocity fields $\tilde u_{N,\lambda}$ and $\tilde\Omega$ are given by
\begin{eqnarray}
\tilde u_{N,\lambda}(x)&:=&\big[\Gc\ast \big((1-\lambda\tilde\mu_{N,0})h+\lambda\tilde\mu_{N,0} e\big)\big](x)\label{eq:defin-uNlambda0}\\
&&\qquad+\lambda\big[\nabla\Gc\ast\big(2\langle\Sigma^\circ\tilde\nu_{N,0}\rangle\D(\Gc\ast h)+\kappa_0 \langle \Sigma_f^\circ\tilde\nu_{N,0}\rangle\big)\big](x),\nonumber\\
\tilde \Omega(x,r)&:=&\Omega^\circ(r)\nabla(\Gc\ast h)(x),\nonumber
\end{eqnarray}
in terms of spatial densities $\tilde\mu_{N,0}:=\langle\tilde\nu_{N,0}\rangle$ and $\tilde\mu_{N,\lambda}:=\langle\tilde\nu_{N,\lambda}\rangle$.
As $\nu_N^\circ$ is smooth, the linear transport equation for $\tilde \nu_{N,0}$ indeed admits a unique (global) solution in $C(\R^+;\Pc\cap W^{1,\infty}(\R^d\times\Sp^{d-1}))$, and the linear transport equation for $\tilde\nu_{N,\lambda}$ then admits a unique solution in $C(\R^+;\Pc\cap\Ld^\infty(\R^d\times\Sp^{d-1}))$. In addition,
averaging the equations over $r$, we find
\[\partial_t\tilde\mu_{N,0}+\Div(\tilde\mu_{N,0}(\Gc\ast h))=0,\qquad \partial_t\tilde\mu_{N,\lambda}+\Div(\tilde\mu_{N,\lambda}\tilde u_{N,\lambda})=0,\]
from which we deduce the following a priori estimates for spatial densities,
\begin{equation}\label{eq:Linfty-bound-muN0}
\sup_t\|\tilde\mu_{N,0}^t\|_{\Ld^\infty(\R^d)}\,=\,\sup_t\|\tilde\mu_{N,\lambda}^t\|_{\Ld^\infty(\R^d)}\,=\,\|\mu_{N}^\circ\|_{\Ld^\infty(\R^d)}\,\lesssim\,1.
\end{equation}
We are now in position to apply Proposition~\ref{prop:MFL} with $\tilde\nu_0,\tilde\nu_\lambda$ replaced by $\tilde\nu_{N,0},\tilde\nu_{N,\lambda}$. Using~\eqref{eq:blob-initWinfty} and~\eqref{eq:Linfty-bound-muN0}, and recalling that we have $\e\ll N^{-1/d}$ in the dilute regime $\lambda\ll1$, we obtain the following result under the sole well-preparedness assumption~\eqref{eq:prepared-thmain}: given~$T>0$, provided that $\lambda\log N\ll1$ is small enough (depending on $T,h,\mu^\circ$), we have for all $t\in[0,T]$ and $\delta\in[\e,1]$,
\begin{eqnarray}
W_\infty(\nu_{\e,N}^t,\tilde\nu_{N,\lambda}^t)\!\!\!&\lesssim_{t,h,\mu^\circ}&\!\!\!N^{-\frac1d}+\lambda\log N,\label{eq:estim-Winfty-prepared}\\
W_\infty(\mu_{\e,N}^t,\tilde\mu_{N,\lambda}^t)\!\!\!&\lesssim_{t,h,\mu^\circ}&\!\!\!N^{-\frac1d}+\lambda^2|\!\log\lambda|\log N,\nonumber\\
\|u_{\e,N}^t-\tilde u_{N,\lambda}^t\|_{\Ld^\infty(\R^d\setminus(\Ic_{\e,N}+\delta B))}\!\!\!&\lesssim_{t,h,\mu^\circ}&\!\!\!N^{-\frac1d}+\lambda^2|\!\log\lambda|\log N+(\tfrac\e\delta)^d(\lambda+\delta).\nonumber
\end{eqnarray}

\medskip
\step2 Comparing the blob solution of~\eqref{eq:tsp-munu-N0} to the Doi model~\eqref{eq:Doi0}: given $(\nu_\lambda,u_\lambda)$ the solution of~\eqref{eq:Doi0} up to time $T$ as constructed in Proposition~\ref{prop:well-posed-Doi},
we prove that for all~$t\in[0,T]$ and $1< p<\infty$,
\begin{eqnarray}
W_p(\nu_\lambda^t,\tilde\nu_{N,\lambda}^t)&\lesssim_{t,h,p,\nu^\circ}&\lambda+W_p(\nu^\circ,\nu_N^\circ),\label{eq:estim-nu-tildenu}\\
W_p(\mu_\lambda^t,\tilde\mu_{N,\lambda}^t)&\lesssim_{t,h,p,\nu^\circ}&\lambda^2+\lambda W_p(\nu^\circ,\nu_N^\circ)+W_p(\mu^\circ,\mu_N^\circ),\label{eq:estim-mu-tildemu}\\
\|u_\lambda^t-\tilde u_{N,\lambda}^t\|_{(\Ld^p+\Ld^\infty)(\R^d)}&\lesssim_{t,h,p,\nu^\circ}&\lambda^2+\lambda W_p(\nu^\circ,\nu_N^\circ),\label{eq:estim-u-tildeu-Ws-re}
\end{eqnarray}
and for $p=1$,
\begin{eqnarray}
W_1(\nu_\lambda,\tilde\nu_{N,\lambda})&\lesssim_{t,h,\nu^\circ}&\lambda+W_1(\nu^\circ,\nu_{N}^\circ),\nonumber\\
W_1(\mu_\lambda,\tilde\mu_{N,\lambda})&\lesssim_{t,h,\nu^\circ}&\lambda^2|\!\log\lambda|+\lambda|\!\log\lambda| W_1(\nu^\circ,\nu_{N}^\circ)+W_1(\mu^\circ,\mu_{N}^\circ),\nonumber\\
\|u_\lambda^t-\tilde u_{N,\lambda}^t\|_{(\Ld^1+\Ld^\infty)(\R^d)}&\lesssim_{t,h,\nu^\circ}&\lambda^2|\!\log\lambda|+\lambda|\!\log\lambda|W_1(\nu^\circ,\nu_{N}^\circ).\label{eq:estim-u-tildeu-Ws-re-p1}
\end{eqnarray}
We split the proof into two further substeps.

\medskip
\substep{2.1} Proof that for all $t\in[0,T]$ and $1<p<\infty$,
\begin{equation}\label{eq:estim-ulambd-tildeuNlamstep1}
{\|u_\lambda^t-\tilde u_{N,\lambda}^t\|_{(\Ld^p+\Ld^{\infty})(\R^d)}}
\,\lesssim_{t,h,p,\nu^\circ}\,\lambda^2+\lambda W_p(\nu_\lambda,\tilde\nu_{N,0}),
\end{equation}
and for $p=1$,
\begin{equation}\label{eq:estim-ulambd-tildeuNlamstep1-re}
{\|u_\lambda^t-\tilde u_{N,\lambda}^t\|_{(\Ld^1+\Ld^{\infty})(\R^d)}}
\,\lesssim_{t,h,\nu^\circ}\,\lambda^2+\lambda W_1(\nu_\lambda,\tilde\nu_{N,0})\log\big(2+\tfrac1{W_1(\nu_\lambda,\tilde\nu_{N,0})}\big).
\end{equation}
Comparing the Stokes equation for $u_\lambda$ in~\eqref{eq:Doi0} with the definition~\eqref{eq:defin-uNlambda0} of $\tilde u_{N,\lambda}$, we find
\begin{multline}\label{eq:decomp-ul-tildeul}
u_\lambda-\tilde u_{N,\lambda}\,=\,\lambda\Gc\ast\big((\mu_\lambda-\tilde\mu_{N,0})(e-h)\big)+\lambda\nabla\Gc\ast\big(2\langle \Sigma^\circ\nu_\lambda\rangle\D(u_\lambda-\Gc\ast h)\big)\\
+\lambda\nabla\Gc\ast\big(2\langle\Sigma^\circ(\nu_\lambda-\tilde\nu_{N,0})\rangle\D(\Gc\ast h)
+\kappa_0\langle\Sigma_f^\circ(\nu_\lambda-\tilde\nu_{N,0})\rangle\big).
\end{multline}
Given $1< p<\infty$, choosing $1< p_0\le p$ such that $p_0<d$ and $q_0:=\frac{dp_0}{d-p_0}\ge p$, we appeal to the Sobolev embedding $\|f\|_{\Ld^{q_0}(\R^d)}\le K_{p_0}\|\nabla f\|_{\Ld^{p_0}(\R^d)}$, where $K_{p_0}$ is a constant only depending on $d$ and on an upper bound on $(d-p_0)^{-1}$,
and we appeal to standard Calder\'on--Zygmund theory in form of $\|\nabla^2\Gc\ast\|_{\Ld^{s}(\R^d)\to\Ld^{s}(\R^d)}\lesssim\frac{s^2}{s-1}$ for $1<s<\infty$, to the effect of
\begin{eqnarray*}
\lefteqn{\|u_\lambda-\tilde u_{N,\lambda}\|_{(\Ld^p+\Ld^{\infty})(\R^d)}}\\
&\le&\lambda\big\|\Gc\ast\big((\mu_\lambda-\tilde\mu_{N,0})(e-h)\big)\big\|_{\Ld^{q_0}(\R^d)}+\lambda\big\|\nabla\Gc\ast\big(2\langle \Sigma^\circ\nu_{\lambda}\rangle\D(u_\lambda-\Gc\ast h)\big)\big\|_{\Ld^{\infty}(\R^d)}\\
&&+\lambda\big\|\nabla\Gc\ast\big(2\langle\Sigma^\circ(\nu_\lambda-\tilde\nu_{N,0})\rangle\D(\Gc\ast h)
+\kappa_0\langle\Sigma_f^\circ(\nu_\lambda-\tilde\nu_{N,0})\rangle\big)\big\|_{\Ld^p(\R^d)}\\
&\lesssim&\lambda \tfrac{K_{p_0}}{p_0-1}\|(\mu_\lambda-\tilde\mu_{N,0})(e-h)\|_{\dot W^{-1,p_0}(\R^d)}
+\lambda\|\langle \Sigma^\circ\nu_{\lambda}\rangle\D(u_\lambda-\Gc\ast h)\|_{\Ld^{1}\cap\Ld^\infty(\R^d)}\\
&&+\lambda\tfrac{p^2}{p-1}\|\langle\Sigma^\circ(\nu_\lambda-\tilde\nu_{N,0})\rangle\D(\Gc\ast h)\|_{\dot W^{-1,p}(\R^d)}
+\lambda\tfrac{p^2}{p-1}\|\langle\Sigma_f^\circ(\nu_\lambda-\tilde\nu_{N,0})\rangle\|_{\dot W^{-1,p}(\R^d)}.
\end{eqnarray*}
Now note that, for all $1\le s\le q<\infty$ and $r>d$, we can bound for all $f,g\in C^\infty_c(\R^d)$ with $\supp f\subset B$,
\[\|fg\|_{\dot W^{-1,s}(\R^d)}\,\lesssim_r\, \|f\|_{\dot W^{-1,q}(\R^d)}\|g\|_{W^{1,r}(B)}.\]
In order to apply this to the above, we need to control the support of $\mu_\lambda,\tilde\mu_{N,0}$: as $\mu^\circ$ is compactly supported, as assumption~\eqref{eq:prepared-thmain} ensures that the blob initial condition $\mu_N^\circ$ is also compactly supported (uniformly with respect to $N$), and recalling the uniform velocity estimate~\eqref{eq:apriori-nabuesgfv00}, we have for all $t\in[0,T]$,
\begin{equation}\label{eq:compact-support}
\supp\mu_\lambda^t~\subset~C_{t,h,\mu^\circ}B,\qquad\quad\supp \mu_{N,0}^t~\subset~C_{t,h}B.
\end{equation}
The above then becomes
\begin{multline*}
{\|u_\lambda-\tilde u_{N,\lambda}\|_{(\Ld^p+\Ld^{\infty})(\R^d)}}
\,\lesssim_{t,h,\mu^\circ}\,\lambda \tfrac{K_{p_0}}{p_0-1}\|\mu_\lambda-\tilde\mu_{N,0}\|_{\dot W^{-1,p}(\R^d)}
+\lambda\|\!\D(u_\lambda-\Gc\ast h)\|_{\Ld^{\infty}(\R^d)}\\
+\lambda\tfrac{p^2}{p-1}\|\langle\Sigma^\circ(\nu_\lambda-\tilde\nu_{N,0})\rangle\|_{\dot W^{-1,p}(\R^d)}
+\lambda\tfrac{p^2}{p-1}\|\langle\Sigma_f^\circ(\nu_\lambda-\tilde\nu_{N,0})\rangle\|_{\dot W^{-1,p}(\R^d)}.
\end{multline*}
To estimate the second right-hand side term, we start from the Stokes equation for $u_\lambda$ in~\eqref{eq:Doi0} in the following form, instead of~\eqref{eq:decomp-ul-tildeul},
\begin{equation*}
u_\lambda-\Gc\ast h\,=\,\lambda\Gc\ast(\mu_\lambda(e-h))
+\lambda\nabla\Gc\ast\big(2\langle\Sigma^\circ\nu_\lambda\rangle\D(u_\lambda)
+\kappa_0\langle\Sigma_f^\circ\nu_\lambda\rangle\big),
\end{equation*}
and we note that the a priori estimates~\eqref{eq:apriori-nabuesgfv00} for $\nu_\lambda,u_\lambda$ yield
\begin{equation}\label{eq:estim-order0ula}
\|u_\lambda-\Gc\ast h\|_{W^{1,\infty}(\R^d)}\,\lesssim_{t,h,\nu^\circ}\,\lambda,
\end{equation}
so we are led to
\begin{multline}\label{eq:firstprediffultildeu}
{\|u_\lambda-\tilde u_{N,\lambda}\|_{(\Ld^p+\Ld^{\infty})(\R^d)}}
\,\lesssim_{t,h,\nu^\circ}\,\lambda^2 +\lambda \tfrac{K_{p_0}}{p_0-1}\|\mu_\lambda-\tilde\mu_{N,0}\|_{\dot W^{-1,p}(\R^d)}\\
+\lambda\tfrac{p^2}{p-1}\|\langle\Sigma^\circ(\nu_\lambda-\tilde\nu_{N,0})\rangle\|_{\dot W^{-1,p}(\R^d)}
+\lambda\tfrac{p^2}{p-1}\|\langle\Sigma_f^\circ(\nu_\lambda-\tilde\nu_{N,0})\rangle\|_{\dot W^{-1,p}(\R^d)}.
\end{multline}
Now we recall the following corollary of Loeper's argument in~\cite{Loeper-06} (see e.g.~\cite[Proposition~5.1]{Hofer-Schubert-21}): for all $1\le s<\infty$, we can bound
\[\|\mu_\lambda-\tilde\mu_{N,0}\|_{\dot W^{-1,s}(\R^d)}\,\le\,\big(\|\mu_\lambda\|_{\Ld^\infty(\R^d)}+\|\tilde\mu_{N,0}\|_{\Ld^\infty(\R^d)}\big)^{1-\frac1s}\,W_s(\mu_\lambda,\tilde\mu_{N,0}),\]
and thus, by the a priori estimate~\eqref{eq:Linfty-bound-muN0},
\[\|\mu_\lambda-\tilde\mu_{N,0}\|_{\dot W^{-1,s}(\R^d)}\,\lesssim_{\mu^\circ}\,W_s(\mu_\lambda,\tilde\mu_{N,0}).\]
Using the smoothness of $\Sigma^\circ,\Sigma_f^\circ$, a similar argument yields for $1\le s<\infty$,
\begin{multline*}
\|\langle\Sigma^\circ(\nu_\lambda-\tilde\nu_{N,0})\rangle\|_{\dot W^{-1,s}(\R^d)}+\|\langle\Sigma_f^\circ(\nu_\lambda-\tilde\nu_{N,0})\rangle\|_{\dot W^{-1,s}(\R^d)}\\
\,\lesssim\,\big(\|\mu_\lambda\|_{\Ld^\infty(\R^d)}+\|\tilde\mu_{N,0}\|_{\Ld^\infty(\R^d)}\big)^{1-\frac1s}\,W_s(\nu_\lambda,\tilde\nu_{N,0})
\,\lesssim_{\mu^\circ}\,W_s(\nu_\lambda,\tilde\nu_{N,0}).
\end{multline*}
Inserting these estimates into~\eqref{eq:firstprediffultildeu}, we get
\begin{equation*}
{\|u_\lambda-\tilde u_{N,\lambda}\|_{(\Ld^p+\Ld^{\infty})(\R^d)}}
\,\lesssim_{t,h,\nu^\circ}\,\lambda^2 +\lambda \Big(\tfrac{K_{p_0}}{p_0-1}+\tfrac{p^2}{p-1}\Big)W_{p}(\nu_\lambda,\tilde\nu_{N,0}),
\end{equation*}
which proves the claim~\eqref{eq:estim-ulambd-tildeuNlamstep1} for $1<p<\infty$.
It remains to cover the case $p=1$.
For that purpose, we first note that for $1<p<\frac{d}{d-1}$ we can choose $p_0=p$, so the above actually becomes
\begin{equation*}
{\|u_\lambda-\tilde u_{N,\lambda}\|_{(\Ld^p+\Ld^{\infty})(\R^d)}}
\,\lesssim_{t,h,\nu^\circ}\, \lambda^2+ \tfrac{\lambda}{p-1} W_{p}(\nu_\lambda,\tilde\nu_{N,0}).
\end{equation*}
Now recall that for measures $\mu,\nu$ supported in a common ball of diameter $R$ we have
\[W_p(\mu,\nu)\le R^{1-\frac1p}W_1(\mu,\nu)^{\frac1p}.\]
By~\eqref{eq:compact-support}, this allows to deduce for all $1<p<\frac{d}{d-1}$,
\begin{equation*}
{\|u_\lambda-\tilde u_{N,\lambda}\|_{(\Ld^1+\Ld^{\infty})(\R^d)}}
\,\lesssim_{t,h,\nu^\circ}\, \lambda^2+\tfrac{\lambda}{p-1}W_1(\nu_\lambda,\tilde\nu_{N,0})^{\frac1p},
\end{equation*}
and the claim~\eqref{eq:estim-ulambd-tildeuNlamstep1-re} follows after optimizing with respect to $p$.

\medskip
\substep{2.2} Proof of~\eqref{eq:estim-nu-tildenu}--\eqref{eq:estim-u-tildeu-Ws-re-p1}.\\
For $1\le p<\infty$, as the probability measures $\nu_\lambda$ and $\tilde\nu_{N,\lambda}$ are absolutely continuous, the $p$-Wasserstein distance $W_p(\nu_{\lambda},\tilde\nu_{N,\lambda})$ admits an optimal transport map:
for all $t\in[0,T]$, we denote by $(\tilde X_{\lambda}^t,\tilde R_{\lambda}^t):\R^d\times\Sp^{d-1}\to\R^d\times\Sp^{d-1}$ an optimal transport map such that
\begin{gather*}
(\tilde X_{\lambda}^t,\tilde R_{\lambda}^t)_*\nu_\lambda^t\,=\,\tilde\nu_{N,\lambda}^t,\\
W_p(\nu_{\lambda}^t,\tilde\nu_{N,\lambda}^t)\,=\,\bigg(\iint_{\R^d\times\Sp^{d-1}}\Big(|x-\tilde X_{\lambda}^t(x,r)|^p+|r-\tilde R_{\lambda}^t(x,r)|^p\Big)\,\nu_\lambda^t(x,r)\,dxd\sigma(r)\bigg)^\frac1p.
\end{gather*}
We can compute the time-derivative of this Wasserstein distance using characteristics, comparing the transport equations for $\nu_\lambda$ and $\tilde\nu_{N,\lambda}$, cf.~\eqref{eq:Doi0} and~\eqref{eq:tsp-munu-N0},
\begin{multline*}
\tfrac{d^+}{dt}W_p(\nu_\lambda,\tilde\nu_{N,\lambda})\,\le\,\bigg(\iint_{\R^d\times\Sp^{d-1}}\Big(|u_\lambda(x)-\tilde u_{N,\lambda}(\tilde X_\lambda(x,r))|^p\\
+\big|(\Omega^\circ(r)\nabla u_\lambda(x))r-(\Omega^\circ(\tilde R_\lambda(x,r))\nabla(\Gc\ast h)(\tilde X_\lambda(x,r)))\tilde R_\lambda(x,r)\big|^p\Big)\,\nu_\lambda(x,r)\,dxd\sigma(r)\bigg)^\frac1p.
\end{multline*}
Using the Lipschitz regularity of $\Omega^\circ$ and the $C^2$ regularity of $u_\lambda$, cf.~\eqref{eq:apriori-nabuesgfv00}, and using that~$(\tilde X_{\lambda},\tilde R_{\lambda})_*\nu_\lambda\,=\,\tilde\nu_{N,\lambda}$, we deduce
\begin{multline*}
\tfrac{d^+}{dt}W_p(\nu_\lambda,\tilde\nu_{N,\lambda})\,\lesssim_{t,h}\,W_p(\nu_\lambda,\tilde\nu_{N,\lambda})\\
+\bigg(\iint_{\R^d\times\Sp^{d-1}}\Big(|u_\lambda(x)-\tilde u_{N,\lambda}(x)|^p
+|\nabla( u_\lambda-\Gc\ast h)(x)|^p\Big)\,\tilde\nu_{N,\lambda}(x,r)\,dxd\sigma(r)\bigg)^\frac1p.
\end{multline*}
Now inserting the result~\eqref{eq:estim-ulambd-tildeuNlamstep1} of Substep~{2.1}, as well as~\eqref{eq:estim-order0ula}, we deduce for all $t\in[0,T]$ and~$1<p<\infty$,
\begin{equation*}
\tfrac{d^+}{dt}W_p(\nu_\lambda,\tilde\nu_{N,\lambda})\,\lesssim_{t,h,p,\nu^\circ}\,W_p(\nu_\lambda,\tilde\nu_{N,\lambda})+\lambda\big(1+W_p(\nu_\lambda,\tilde\nu_{N,0})\big).
\end{equation*}
Similar estimates yield
\[\tfrac{d^+}{dt}W_p(\nu_\lambda,\tilde\nu_{N,0})\,\lesssim_{t,h,p,\nu^\circ}\, W_p(\nu_\lambda,\tilde\nu_{N,0})+\lambda,\]
and
\[\tfrac{d^+}{dt}W_p(\mu_\lambda,\tilde\mu_{N,\lambda})\,\lesssim_{t,h,p,\nu^\circ}\, W_p(\mu_\lambda,\tilde\mu_{N,\lambda})+\lambda W_p(\nu_\lambda,\tilde\nu_{N,0})+\lambda^2.\]
By the Gronwall inequality, this proves~\eqref{eq:estim-nu-tildenu} and~\eqref{eq:estim-mu-tildemu}. Combining with the result~\eqref{eq:estim-ulambd-tildeuNlamstep1} of Substep~2.1, the claim~\eqref{eq:estim-u-tildeu-Ws-re} follows as well.
In case $p=1$, rather starting from the result~\eqref{eq:estim-ulambd-tildeuNlamstep1-re} of Substep~{2.1}, the claim~\eqref{eq:estim-u-tildeu-Ws-re-p1} follows similarly.

\medskip
\step3 Conclusion.\\
By the triangle inequality, noting that $W_p\le W_\infty$ and combining~\eqref{eq:estim-Winfty-prepared} and~\eqref{eq:estim-nu-tildenu}, we find for all $t\in[0,T]$ and~$1\le p<\infty$,
\begin{eqnarray*}
W_p(\nu_{\e,N}^t,\nu_\lambda^t)&\le&W_\infty(\nu_{\e,N}^t,\tilde\nu_{N,\lambda}^t)+W_p(\nu_\lambda^t,\tilde\nu_{N,\lambda}^t)\\
&\lesssim_{t,h,p,\nu^\circ}&\lambda\log N+N^{-\frac1d}+W_p(\nu^\circ,\nu_N^\circ),
\end{eqnarray*}
and thus, further using~\eqref{eq:blob-initWinfty},
\begin{equation*}
W_p(\nu_{\e,N}^t,\nu_\lambda^t)
\,\lesssim_{t,h,p,\nu^\circ}\,\lambda\log N+N^{-\frac1d}+W_p(\nu_{\e,N}^\circ,\nu^\circ).
\end{equation*}
Similarly combining~\eqref{eq:estim-Winfty-prepared} with~\eqref{eq:estim-mu-tildemu}--\eqref{eq:estim-u-tildeu-Ws-re-p1}, using the uniform estimate of Proposition~\ref{prop:Lip-est} to estimate the $\Ld^1$ norm of $u_{\e,N}-u_\lambda$ in $\Ic_{\e,N}+\delta B$, and choosing $\delta=N^{-1/d}$, the conclusion follows.
\qed

\section*{Acknowledgements}
The author thanks Alexandre Girodroux-Lavigne, Antoine Gloria, Matthieu Hillairet, Richard Höfer, and Amina Mecherbet for related discussions, and he acknowledges financial support from the F.R.S.-FNRS as well as from the European Union (ERC, PASTIS, Grant Agreement n$^\circ$101075879).\footnote{{Views and opinions expressed are however those of the authors only and do not necessarily reflect those of the European Union or the European Research Council Executive Agency. Neither the European Union nor the granting authority can be held responsible for them.}}

\bibliographystyle{plain}
\bibliography{biblio}

\def\cprime{$'$} \def\cprime{$'$} \def\cprime{$'$}
\begin{thebibliography}{10}

\bibitem{BCD-11}
H.~Bahouri, J.-Y. Chemin, and R.~Danchin.
\newblock {\em Fourier analysis and nonlinear partial differential equations},
  volume 343 of {\em Grundlehren der Mathematischen Wissenschaften [Fundamental
  Principles of Mathematical Sciences]}.
\newblock Springer, Heidelberg, 2011.

\bibitem{Bernou-Duerinckx-Gloria-22}
A.~Bernou, M.~Duerinckx, and A.~Gloria.
\newblock Homogenization of active suspensions and reduction of effective
  viscosity.
\newblock Preprint, arXiv:2301.00166.

\bibitem{Brenner-74}
H.~Brenner.
\newblock Rheology of a dilute suspension of axisymmetric {B}rownian particles.
\newblock {\em Int. J. Multiphase Flow}, 1:195--341, 1974.

\bibitem{Carrillo-Choi-Hauray-14}
J.~A. Carrillo, Y.-P. Choi, and M.~Hauray.
\newblock The derivation of swarming models: {M}ean-field limit and
  {W}asserstein distances.
\newblock In {\em Collective Dynamics from Bacteria to Crowds}, CISM
  International Centre for Mechanical Sciences, pages 1--46. Springer, 2014.

\bibitem{Champion-08}
T.~Champion, L.~De~Pascale, and P.~Juutinen.
\newblock The $\infty$-{W}asserstein distance: local solutions and existence of
  optimal transport maps.
\newblock {\em SIAM J. Math. Anal.}, 40(1):1--20, 2008.

\bibitem{Decoene-Martin-Maury-11}
A.~Decoene, S.~Martin, and B.~Maury.
\newblock Microscopic modelling of active bacterial suspensions.
\newblock {\em Math. Model. Nat. Phenom.}, 6:98--129, 2011.

\bibitem{Degond-19}
S.~Degond, P. ad Merino-Aceituno, F.~Vergnet, and H.~Yu.
\newblock {Coupled Self-Organized Hydrodynamics and Stokes Models for
  Suspensions of Active Particles}.
\newblock {\em J. Math. Fluid Mech.}, 21(6):1--36, 2019.

\bibitem{Doi-81}
M.~Doi.
\newblock Molecular-dynamics and rheological properties of
  concentrated-solutions of rodlike polymers in isotropic and
  liquid-crystalline phases.
\newblock {\em J. Polym. Sci. Polym. Phys. Ed.}, 19:229--243, 1981.

\bibitem{Doi-Edwards-78}
M.~Doi and S.~F. Edwards.
\newblock Dynamics of rod-like macromolecules in concentrated solution. {P}art
  1.
\newblock {\em J. Chem. Soc., Faraday Trans. 2}, 74:560--570, 1978.

\bibitem{Doi-Edwards-88}
M.~Doi and S.~F. Edwards.
\newblock {\em The theory of polymer dynamics}.
\newblock Oxford University Press, 1988.

\bibitem{D-21a}
M.~Duerinckx.
\newblock {On the size of chaos via Glauber calculus in the classical
  mean-field dynamics}.
\newblock {\em Commun. Math. Phys.}, 382:613--653, 2021.

\bibitem{DG-21c}
M.~Duerinckx and A.~Gloria.
\newblock Continuum percolation in stochastic homogenization and the effective
  viscosity problem.
\newblock Preprint, arXiv:2108.09654.

\bibitem{DG-21b}
M.~Duerinckx and A.~Gloria.
\newblock On {E}instein's effective viscosity formula.
\newblock Preprint, arXiv:2008.03837.

\bibitem{DG-16a}
M.~Duerinckx and A.~Gloria.
\newblock Analyticity of homogenized coefficients under {B}ernoulli
  perturbations and the {C}lausius-{M}ossotti formulas.
\newblock {\em Arch. Ration. Mech. Anal.}, 220(1):297--361, 2016.

\bibitem{DG-21a}
M.~Duerinckx and A.~Gloria.
\newblock Corrector equations in fluid mechanics: {E}ffective viscosity of
  colloidal suspensions.
\newblock {\em Arch. Ration. Mech. Anal.}, 239:1025--1060, 2021.

\bibitem{DG-22-review}
M.~Duerinckx and A.~Gloria.
\newblock Effective viscosity of semi-dilute suspensions.
\newblock {\em Séminaire Laurent Schwartz, EDP et applications}, 2021-2022.
\newblock Expos\'e n$^\circ$III.

\bibitem{DG-22a}
M.~Duerinckx and A.~Gloria.
\newblock Quantitative homogenization theory for random suspensions in steady
  stokes flow.
\newblock {\em J. \'Ec. Polytech. - Math.}, 9:1183--1244, 2022.

\bibitem{GV-21}
D.~G\'{e}rard-Varet.
\newblock Derivation of the {B}atchelor-{G}reen formula for random suspensions.
\newblock {\em J. Math. Pures Appl. (9)}, 152:211--250, 2021.

\bibitem{Girodroux-Lavigne-22}
A.~Girodroux-Lavigne.
\newblock Derivation of an effective rheology for dilute suspensions of
  micro-swimmers.
\newblock Preprint, arXiv:2204.04967.

\bibitem{Haines-Aranson-08}
B.~M. Haines, I.~S. Aranson, L.~Berlyand, and D.~A. Karpeev.
\newblock Effective viscosity of dilute bacterial suspensions: a
  two-dimensional model.
\newblock {\em Phys. Biol.}, 5:046003, 2008.

\bibitem{Haines-Aranson-09}
B.~M. Haines, A.~Sokolov, I.~S. Aranson, L.~Berlyand, and D.~A. Karpeev.
\newblock Three-dimensional model for the effective viscosity of bacterial
  suspensions.
\newblock {\em Phys. Rev. E}, 80:041922, 2009.

\bibitem{Hauray-Jabin-07}
M.~Hauray and P.-E. Jabin.
\newblock {$N$}-particles approximation of the {V}lasov equations with singular
  potential.
\newblock {\em Arch. Ration. Mech. Anal.}, 183(3):489--524, 2007.

\bibitem{Hauray-Jabin-15}
M.~Hauray and P.-E. Jabin.
\newblock Particle approximation of {V}lasov equations with singular forces:
  propagation of chaos.
\newblock {\em Ann. Sci. \'Ec. Norm. Sup\'er. (4)}, 48(4):891--940, 2015.

\bibitem{Helzel-Otto-06}
C.~Helzel and F.~Otto.
\newblock Multiscale simulations for suspensions of rod-like molecules.
\newblock {\em J. Comput. Phys.}, 216(1):52--75, 2006.

\bibitem{Hillairet-Sabbagh-23}
M.~Hillairet and L.~Sabbagh.
\newblock {Global solutions to coupled (Navier-)Stokes Newton systems in
  $\R^3$}.
\newblock {\em Asymptotic Anal.}, 132(1-2):217--243, 2023.

\bibitem{Hillairet-Wu-20}
M.~Hillairet and D.~Wu.
\newblock {Effective viscosity of a polydispersed suspension}.
\newblock {\em J. Math. Pures Appl.}, 138:413--447, 2020.

\bibitem{Hinch-Leal-72}
E.~J. Hinch and L.~G. Leal.
\newblock The effect of {B}rownian motion on the rheological properties of a
  suspension of non-spherical particles.
\newblock {\em J. Fluid Mech.}, 52:683--712, 1972.

\bibitem{Hofer-18}
R.~M. H\"{o}fer.
\newblock Sedimentation of inertialess particles in {S}tokes flows.
\newblock {\em Comm. Math. Phys.}, 360(1):55--101, 2018.

\bibitem{Hofer-21}
R.~M. H\"{o}fer.
\newblock Convergence of the method of reflections for particle suspensions in
  {S}tokes flows.
\newblock {\em J. Differential Equations}, 297:81--109, 2021.

\bibitem{Hofer-Leocata-Mecherbet-22}
R.~M. H\"ofer, M.~Leocata, and A.~Mecherbet.
\newblock Derivation of the viscoelastic stress in {S}tokes flows induced by
  non-spherical {B}rownian rigid particles through homogenization.
\newblock Preprint, arXiv:2202.09317.

\bibitem{Hofer-Mecherbet-Schubert-22}
R.~M. H\"ofer, A.~Mecherbet, and R.~Schubert.
\newblock Non-existence of mean-field models for particle orientations in
  suspensions.
\newblock Preprint, arXiv:2210.15382.

\bibitem{Hofer-Schubert-23}
R.~M. H\"ofer and R.~Schubert.
\newblock {Sedimentation of Particles with Very Small Inertia in Stokes Flows
  I: Convergence to the Transport-Stokes Equations}.
\newblock Preprint, arXiv:2302.04637.

\bibitem{Hofer-Schubert-21}
R.~M. H\"ofer and R.~Schubert.
\newblock The influence of {E}instein's effective viscosity on sedimentation at
  very small particle volume fraction.
\newblock {\em Ann. Inst. H. Poincar\'e Anal. Non Lin\'eaire},
  38(6):1897--1927, 2021.

\bibitem{Jabin-14}
P.-E. Jabin.
\newblock A review of the mean field limits for {V}lasov equations.
\newblock {\em Kinet. Relat. Models}, 7(4):661, 2014.

\bibitem{Jabin-Otto-04}
P.-E. Jabin and F.~Otto.
\newblock Identification of the dilute regime in particle sedimentation.
\newblock {\em Comm. Math. Phys.}, 250(2):415--432, 2004.

\bibitem{Jeffery-22}
G.~B. Jeffery.
\newblock The motion of ellipsoidal particles immersed in a viscous fluid.
\newblock {\em Proc. R. Soc. Lond. A}, 102:161--179, 1922.

\bibitem{Loeper-06}
G.~Loeper.
\newblock Uniqueness of the solution to the {V}lasov-{P}oisson system with
  bounded density.
\newblock {\em J. Math. Pures Appl. (9)}, 86(1):68--79, 2006.

\bibitem{Mecherbet-19}
A.~Mecherbet.
\newblock Sedimentation of particles in {S}tokes flow.
\newblock {\em Kinet. Relat. Models}, 12(5):995--1044, 2019.

\bibitem{Mecherbet-20}
A.~Mecherbet.
\newblock A model for suspension of clusters of particle pairs.
\newblock {\em ESAIM: Math. Model. Numer. Anal.}, 54(5):1597--1634, 2020.

\bibitem{Oelschlager-90}
K.~Oelschl\"ager.
\newblock Large systems of interacting particles and the porous medium
  equation.
\newblock {\em J. Differential Equations}, 88(2):294--346, 1990.

\bibitem{Otto-Tzavaras-08}
F.~Otto and A.~E. Tzavaras.
\newblock Continuity of velocity gradients in suspensions of rod-like
  molecules.
\newblock {\em Comm. Math. Phys.}, 277(3):729--758, 2008.

\bibitem{Potomkin-Ryan-Berlyand-16}
M.~Potomkin, S.~D. Ryan, and L.~Berlyand.
\newblock {Effective Rheological Properties in Semi-dilute Bacterial
  Suspensions}.
\newblock {\em Bull. Math. Biol.}, 78:580--615, 2016.

\bibitem{Saintillan-10}
D.~Saintillan.
\newblock The dilute rheology of swimming suspensions: {A} simple kinetic
  model.
\newblock {\em Exp. Mech.}, 50(9):1275--1281, 2010.

\bibitem{Saintillan-18}
D.~Saintillan.
\newblock Rheology of active fluids.
\newblock {\em Annu. Rev. Fluid Mech.}, 50:563--592, 2018.

\bibitem{Saintillan-Shelley-08}
D.~Saintillan and M.~J. Shelley.
\newblock Instabilities, pattern formation, and mixing in active suspensions.
\newblock {\em Phys. Fluids}, 20(12):123304, 2008.

\bibitem{Saintillan-Shelley-13}
D.~Saintillan and M.~J. Shelley.
\newblock Active suspensions and their nonlinear models.
\newblock {\em C. R. Physique}, 14:497--517, 2013.

\end{thebibliography}

\end{document}